\documentclass[11pt]{amsart}
\usepackage{hyperref}
\usepackage{amsfonts,amssymb,amsmath,amsthm,amscd,euscript,array,mathrsfs,stmaryrd}
\usepackage[all]{xy}
\usepackage{bbm}
\usepackage{booktabs}
\usepackage{upgreek}
\usepackage{color}

\newcommand{\Q}{\mathbb Q}

\newcommand{\C}{\mathbb C}

\newcommand{\N}{\mathbb N}

\newcommand{\Z}{\mathbb Z}
\newcommand{\bls}{\backslash}
\newcommand{\spn}{\mathrm{Span}}

\newcommand{\cC}{\mathcal C}

\newcommand{\cA}{\mathcal A}
\newcommand{\cB}{\mathcal B}

\newcommand{\cI}{\mathcal I}

\newcommand{\cR}{\mathcal R}

\newcommand{\cL}{\mathcal L}

\newcommand{\g}{\mathfrak}

\newcommand{\lag}{\langle}
\newcommand{\rag}{\rangle}
\newcommand{\eps}{\varepsilon}

\newcommand{\oline}{\overline}
\newcommand{\sseq}{\subseteq}

\newtheorem{theorem}{\textbf{Theorem}}[section]

\newtheorem{definition}[theorem]{\textbf{Definition}}
\newtheorem{prp}[theorem]{\textbf{Proposition}}
\newtheorem{lem}[theorem]{\textbf{Lemma}}
\newtheorem{cor}[theorem]{\textbf{Corollary}}
\newtheorem{example}[theorem]{\textbf{Example}}

\newtheorem{remark}[theorem]{\textbf{Remark}}
\newenvironment{rmk}{\begin{remark}\rmfamily\upshape}{\end{remark}}
\newenvironment{ex}{\begin{example}\rmfamily\upshape}{\end{example}}
\newenvironment{dfn}{\begin{definition}\rmfamily\upshape}{\end{definition}}

\newtheorem{notation}[theorem]{\textbf{Notation}}

\newtheorem*{notat*}{Notation}

\newtheorem{thmx}{Theorem}

\newtheorem{thmy}{Theorem}

\numberwithin{theorem}{subsection}

\usepackage{enumitem}
\setitemize{itemsep=.5mm,topsep=1.5mm,parsep=0pt,
partopsep=0pt}

\newcommand{\ad}{\mathrm{ad}}

\newcommand{\into}{\hookrightarrow}

\newcommand{\gl}{\mathfrak{gl}}

\newcommand{\sPD}{\mathscr{PD}}

\newcommand{\sD}{\mathscr{D}}
\newcommand{\sP}{\mathscr{P}}

\newcommand{\Hom}{\mathrm{Hom}}

\newcommand{\sfm}{\mathsf{m}}

\newcommand{\sfi}{\mathsf{i}}

\newcommand{\la}{\lambda}

\newcommand{\Mat}{\mathrm{Mat}}

\usepackage{youngtab}

\newcommand{\del}{\partial}

\newcommand{\End}{\mathrm{End}}

\newcommand{\Lm}{\mathscr L}
\newcommand{\Rn}{\mathscr R}
\newcommand{\Lmw}{{\mathscr L_\bullet}}
\newcommand{\Rnw}{{\mathscr R_\bullet}}

\newcommand{\ULw}{\mathring{U}_L}
\newcommand{\URw}{\mathring{U}_R}
\newcommand{\ULRw}{\mathring{U}_{LR}}

\newcommand{\Lhb}{\mathscr L_{\g h,\bullet}}
\newcommand{\Rhb}{\mathscr R_{\g h,\bullet}}

\newcommand{\Emb}{\boldsymbol{e}}

\title[Quantized Weyl algebras, double centralizers, and FFT for $U_q(\gl_n)$]{
Quantized Weyl algebras, 
the double centralizer property, and a new First Fundamental Theorem  for $U_q(\gl_n)$}
\author{Gail Letzter}
\address{Gail Letzter, Mathematics Research Group, National Security Agency}
\email{ gletzter@verizon.net}
\author{Siddhartha Sahi}
\address{Siddhartha Sahi, Department of Mathematics, Rutgers University}
\email{sahi@math.rutgers.edu}
\author{Hadi Salmasian}
\address{Hadi Salmasian, Department of Mathematics and Statistics, University of Ottawa}
\email{ hadi.salmasian@uottawa.ca}

\begin{document}
\maketitle 

\begin{abstract}
Let $\sP:=\sP_{m\times n}$ denote the quantized coordinate ring of the space of $m\times n$ matrices, equipped with natural actions of the quantized enveloping algebras $U_q(\gl_m)$ and $U_q(\gl_n)$. Let $\Lm$ and $\Rn$ denote the images of $U_q(\gl_m)$ and
$U_q(\gl_n)$ in $\mathrm{End}(\sP)$, respectively. 
We define a $q$-analogue  of the algebra of polynomial-coefficient differential operators inside $\End(\sP)$, henceforth denoted by
$\sPD$,
and we prove that $\Lm\cap \sPD$ and $\Rn\cap \sPD$ are mutual centralizers inside $\sPD$. Using this, we establish a new  First Fundamental Theorem of invariant theory for $U_q(\gl_n)$. 
We also compute  explicit formulas in terms of $q$-determinants for generators of the 
algebras $\mathscr L_{\g h}\cap\sPD$ and $\mathscr R_{\g h}\cap\sPD$,
where $\mathscr L_{\g h}$ and $\mathscr R_{\g h}$ denote the images of the 
Cartan subalgebras of $U_q(\gl_m)$ 
and
$U_q(\gl_n)$ in $\mathrm{End}(\sP)$, respectively.
Our algebra $\sPD$  
and the algebra $\mathrm{Pol}(\mathrm{Mat}_{m,n})_{q}$
that is defined in \cite{SSV04} are related by extension of scalars, but we give a new construction of $\sPD$ using deformed twisted tensor products. 
 \end{abstract}

\section{Introduction}
\label{sec:Introduction}
The First Fundamental Theorem (FFT) is one of the pinnacles of invariant theory with a history as old as Hermann Weyl's influential book, \emph{The Classical Groups}~\cite{We39}. In its original form, the FFT for the group
$\mathrm{GL}_n$
describes the generators of the subalgebra of $\mathrm{GL}_n$-invariants in the polynomial algebra $\mathcal P(V^{\oplus k}\oplus (V^*)^{\oplus l})$, where $V:=\C^n$ denotes the standard $\mathrm{GL}_n$-module.
   
It was pointed out by R. Howe~\cite[Sec. 2.3]{Ho95} that 
the FFT has an equivalent formulation as a double centralizer property, which we now recall.   Let $\mathrm{Mat}_{m\times n}$ denote the vector space of complex $m\times n$ matrices. Then $\mathrm{Mat}_{m\times n}$ has a  natural  $\mathrm{GL}_m\times\mathrm{GL}_n$-module structure by left and right matrix multiplication. We equip the 
algebra $\mathcal P:=\mathcal P(\mathrm{Mat}_{m\times n})$ of polynomials on $\mathrm{Mat}_{m\times n}$
and the
algebra $\mathcal{PD}:=\mathcal{PD}(\mathrm{Mat}_{m\times n})$ of polynomial-coefficient differential operators
on $\mathrm{Mat}_{m\times n}$ with their canonical 
$\mathrm{GL}_m\times\mathrm{GL}_n$-module structures. Recall that $\mathcal{P}$ is a $\mathcal{PD}$-module. The (infinitesimal) actions of the  Lie algebras $\gl_m$ and $\gl_n$ on 
$\mathcal{P}$ 
are given by certain differential operators of order one, which are usually called \emph{polarization operators}. 
It follows that there exists a homomorphism of algebras $\phi:U_{m,n}\to\mathcal{PD}$, where 
$U_{m,n}:=U(\gl_m)\otimes U(\gl_n)$ is the tensor product of the  universal enveloping algebras of $\gl_m$ and $\gl_n$,
such that the diagram
\begin{equation}
\label{eq:commdia}
\vcenter{
\xymatrix{
U_{m,n}\otimes  \mathcal P
 \ar[rd]_{x\otimes f\mapsto \phi(x)\otimes f\hspace{5mm}}\ar[rr]^{\hspace{5mm}x\otimes f\mapsto x\cdot f}& & \mathcal{P}\\
& \mathcal{PD}\otimes \mathcal{P} \ar[ru]_{D\otimes f\mapsto D\cdot f}& }
}
\end{equation}
commutes. 
The \emph{operator commutant version} of the FFT, according to~\cite[Thm 2.3.3]{Ho95}, states that
 the subalgebra 
$\mathcal{PD}^{\mathrm{GL}_m}$
of $\mathrm{GL}_m$-invariants in $\mathcal{PD}$ is generated by the image of $U(\gl_n)$. 
Since 
$\mathcal{PD}^{\mathrm{GL}_m}=
\mathcal{PD}^{U(\gl_m)}$, 
the latter assertion is equivalent to the following: the images of $U(\gl_m)$ and $U(\gl_n)$ in $\mathcal{PD}$ are mutual centralizers.

In \cite[Sec. 6]{LZZ11} the authors 
extend the original form of the FFT to the quantized enveloping algebra $U_q(\gl_n)$  by considering a $q$-analogue  of  $\mathcal{P}(V^{\oplus k}\oplus (V^*)^{l})$ that carries a $U_q(\gl_n)$-action, and then describing the generators of the subalgebra of invariants. It is then natural to ask if the operator commutant version of the FFT 
also has a $q$-analogue. It turns out that in the quantized setting, the situation for the operator commutant FFT is more subtle than in the classical case.  One major issue is how to quantize the Weyl algebra $\mathcal{PD}$ and, more importantly, the map $\phi:U_{m,n}\to \mathcal{PD}$. Indeed we provide some justification that the latter map cannot be fully quantized (see Proposition~\ref{rmk-nogo}). 
Nevertheless, our first main result (Theorem~\ref{thm-Main-A})
is a positive answer to the 
above question.

From now on let
$\Bbbk:=\C(q)$ be the field of rational functions in a parameter $q$.
For the operator commutant FFT in the quantized setting we need a quantized Weyl algebra $\sPD:=\sPD_{m\times n}$. The  $\Bbbk$-algebra $\sPD$ that we consider is closely related to the algebra
$\mathrm{Pol}(\mathrm{Mat}_{m,n})_{q}$
of
\cite{SSV04,BKV06} (see Corollary~\ref{cor:POP}).
We give a different construction of $\sPD$ as the \emph{deformed twisted tensor product} of $\sP:=\sP_{m\times n}$, the quantized coordinate ring of $\mathrm{Mat}_{m\times n}$, and $\sD:=\sD_{m\times n}$, the quantized algebra of  constant-coefficient differential operators
on $\mathrm{Mat}_{m\times n}$ 
(see Section~\ref{sec-PDmn} for precise definitions). 
The construction of  $\sP$ and $\sD$ is  analogous to the FRT 
construction~\cite[Sec. 9.1]{KS97}.
Concretely, the algebra  $\sPD$  is  generated by $2mn$ generators $t_{i,j}$ and $\del_{i,j}$, where $1\leq i\leq m$ and $1\leq j\leq n$, modulo the relations that are described in Section~\ref{sec-PDmn} (see Definition~\ref{dfn-PDPD2}).
 From now on we set
\[
U_L:=U_q(\gl_m)\quad,\quad
U_R:=U_q(\gl_n)\quad,\quad
U_{LR}:=U_L\otimes U_R.
\]
Both $\sP$ and $\sPD$ are 
$U_{LR}$-module algebras (the explicit formulas for the $U_{LR}$-action on generators are given  in  Remark~\ref{rmk:actionformulas}).
Furthermore, $\sP$ is naturally a $\sPD$-module.
In particular, we have homomorphisms of associative algebras
\[
\phi_U:U_{LR}\to\End_\Bbbk(\sP)\quad\text{and}\quad
\phi_{PD}:\sPD\to
\End_\Bbbk(\sP).
\]
Since $\sP$ is a faithful $\sPD$-module (see Proposition~\ref{prp:faithful-action}), we can identify $\sPD$ with $\phi_{PD}(\sPD)$. Using the latter identification, 
we set
\[
\Lm:=\phi_U(U_L\otimes 1),
\quad \Rn:=\phi_U
(1\otimes U_R),
\quad
\Lmw:=\Lm\cap \sPD,\quad
\Rnw:=\Rn\cap \sPD.
\]
Note that in general 
we have  $\Lmw\subsetneq \Lm$ and 
$\Rnw\subsetneq \Rn$. In fact if $m\leq n$ then the restriction of $\phi_U$ to $U_L\otimes 1$ yields an isomorphism $U_L\cong \Lm$ but one can show that $\phi_U^{-1}(\Lmw)$ is properly contained in the locally finite part of $U_L$ (see Proposition~\ref{prp:ULWFULUR/}
 and  Example~\ref{ex:1}).
Set
\begin{equation}
\label{eq:Liii}
\mathsf{L}_{i,j}:=\sum_{r=1}^nt_{i,r}\del_{j,r}\ \text{ for }1\leq i,j\leq m \ \text{ and }\  \mathsf{R}_{i,j}:=\sum_{r=1}^mt_{r,i}\del_{r,j}\ \text{ for }1\leq i,j\leq n.
\end{equation}
The $\mathsf L_{i,j}$ and the $\mathsf R_{i,j}$ are natural analogues in $\sPD$ of the 
 polarization operators of the (non-quantized) Weyl algebra $\mathcal{PD}$.  
We have $\mathsf L_{i,j}\in\Lmw $ and $\mathsf R_{i,j}\in\Rnw$ (see Corollary~\ref{cor:Rij-in-L}). Of course, similar inclusions hold in the non-quantized case. However, the situation with the preimages of 
the
$\mathsf L_{i,j}$ and the $\mathsf R_{i,j}$ in $U_L$ and $U_R$ is more complicated. For example,   the root vector $E_{\eps_i-\eps_j}$ of $U_L$ (respectively, $U_R$), where $i<j$, does not lie in the preimage of $\mathsf L_{i,j}$ (respectively, $\mathsf R_{i,j}$).

Henceforth we adopt the following notation: for  subsets $\mathcal Y,\mathcal Z$ of an algebra $\mathcal X$, we set \begin{equation}
\label{eq:YtoZ}
\mathcal Y^\mathcal Z:=\{y\in\mathcal Y\,:\,yz=zy\text{ for all }z\in\mathcal Z\}
.\end{equation} 
Our first main theorem is the following. 
\begin{thmx}
\label{thm-Main-A}
Let $\Lm$, $\Rn$, $\Lmw$, and $\Rnw$ be the subalgebras of $\End_\Bbbk(\sP)$ defined above. We identify $\sPD$ with $\phi_{PD}(\sPD)\sseq \End_\Bbbk(\sP)$. Then the following statements hold.
\begin{itemize}
\item[\rm(i)] 
$\sPD^{\Rnw}=\sPD^\Rn=\Lmw
$. Furthermore, $\Lmw$ is generated 
by the $\mathsf L_{i,j}$ for $1\leq i,j\leq m$.

\item[\rm(ii)]
$
\sPD^{\Lmw}=\sPD^\Lm=\Rnw
$.
Furthermore, $\Rnw$ is generated by the $\mathsf{R}_{i,j}$ for  $1\leq i,j\leq n$.
\end{itemize}

\end{thmx}

Let us elucidate the relation between 
Theorem~\ref{thm-Main-A} and the 
literature on Howe duality and the 
FFT in the quantized setting. 
Quantized analogues of $(\gl_m,\gl_n)$-duality have been established in~\cite{Zh02} and~\cite{NYM93}, but these works do not consider the double centralizer property inside a quantized Weyl algebra. 
To compare our results with those of 
Lehrer--Zhang--Zhang~\cite{LZZ11}, we
 briefly explain their formulation of the FFT for $U_q(\gl_n)$. In~\cite[Sec. 6]{LZZ11} the authors define a $q$-analogue of the algebra
$\mathcal P(V^{\oplus k}\oplus (V^*)^{\oplus l})$, which they call $\mathcal A_{k,l}$ (this algebra tacitly depends on $n$ as well). The algebra $\mathcal A_{k,l}$ is isomorphic to a  twisted tensor product of $\sP_{k\times n}$ and $\sD_{l\times n}$, but the twisting is only with respect to the universal $R$-matrix of $U_q(\gl_n)$. In particular, in the special case $k=l$ the relations on the generators of $\mathcal A_{k,k}$ are not symmetric with respect to their indices. Because of this asymmetry,  $\mathcal A_{k,k}$ does not appear to be  the desired object for proving a double centralizer statement. The twisting that we consider to define $\sPD$ uses  the universal $R$-matrices of both $U_L$ and $U_R$.  
In addition, unlike $\mathcal A_{k,l}$ whose relations are homogeneous, the relation (R6) of $\sPD$ is not homogeneous. From this viewpoint, $\sPD$ resembles the classical Weyl algebra more than $\mathcal A_{k,k}$. 

Our second main result (Theorem~\ref{thm:C} below)
is a new  FFT for $U_q(\gl_n)$, 
in the spirit of the aforementioned result of \cite{LZZ11},
for a family of algebras 
 $\mathscr A_{k,l,n}$ where $k,l,n$ are positive integers. 
The latter algebras generalize $\sPD$ and 
indeed we have
$
\mathscr A_{m,m,n}=\sPD_{m\times n}
$. 
An explicit presentation of $\mathscr A_{k,l,n}$ by generators and relations is given in Proposition~\ref{prp:explicit-prs-Akln}. 
To state Theorem~\ref{thm:C}, we need some notation. 
For integers $1\leq a\leq m$ and $1\leq b\leq n$
there exists an embedding of associative algebras 
\begin{equation}
\label{eq:embeD}
\Emb=\Emb_{a\times b}^{m\times n}:\sPD_{a\times b}\into \sPD_{m\times n}
\end{equation}
that is defined as follows.
We relabel the generators of $\sPD_{a\times b}$ by setting 
\begin{equation}
\label{eq:tildetd}
\tilde{t}_{i,j}:=t_{a+1-i,b+1-j}^{}
\ \ \text{ and }\ \ 
\tilde{\del}_{i,j}:=\del_{a+1-i,b+1-j}^{}
\quad\text{ for }1\leq i\leq a,\ 1\leq j\leq b.
\end{equation}
We relabel the generators of $\sPD_{m\times n}$ similarly, with $a$ and $b$ replaced by $m$ and $n$ respectively.
The map $\Emb$ of~\eqref{eq:embeD} is uniquely determined by the  assignments
$
\Emb(\tilde{t}_{i,j}):=\tilde{t}_{i,j}$ and
$\Emb(\tilde{\del}_{i,j}):=\tilde{\del}_{i,j}$
(see Proposition~\ref{prp:existence(6)}).
Concretely, the map $\Emb$ identifies the $a\times b$ matrices $[t_{i,j}]$ and $[\del_{i,j}]$ formed by the generators of $\sPD_{a\times b}$ with the intersections of the lowest $a$ rows and the rightmost $b$ columns in the analogous $m\times n$ matrices formed by the generators of $\sPD_{m\times n}$.
Now fix integers $k,l,n\geq 1$ and set $m:=\max\{k,l\}$.
We define $\mathscr A_{k,l,n}$  to be the subalgebra 
of $\sPD_{m\times n}$ that is generated by 
the $\tilde t_{i,j}$ and the  $\tilde\del_{i',j}$, 
where 
$1\leq i\leq k$, $1\leq i'\leq l$, and $1\leq j\leq n$.  
The $U_R$-action on $\sPD_{m\times n}$ leaves  $\mathscr A_{k,l,n}$ invariant and thus 
$\mathscr A_{k,l,n}$ is a $U_R$-module algebra. 
The standard degree filtration of $\mathscr A_{k,l,n}$  (corresponding to setting $\deg \tilde t_{i,j}= \deg\tilde \del_{i',j}=1$)
 is $U_R$-stable, hence the associated graded algebra $\mathrm{gr}(\mathscr A_{k,l,n})$ is also a $U_R$-module algebra.  
Let $\epsilon_R$ be the counit of $U_R$ and denote the subalgebra of $U_R$-invariants in $\mathscr A_{k,l,n}$
by 
$\left(\mathscr A_{k,l,n}\right)_{(\epsilon_R)}
$, that is, 
\[
\left(\mathscr A_{k,l,n}\right)_{(\epsilon_R)}
:=
\left\{
D\in\mathscr A_{k,l,n}\,:\,x\cdot D=\epsilon_R(x)D\text{ for }x\in U_R\right\}.
\]
We denote the subalgebra of $U_R$-invariants in  
$\mathrm{gr}(\mathscr A_{k,l,n})$ by 
$\big(\mathrm{gr}(\mathscr A_{k,l,n})\big)_{(\epsilon_R)}$ as well. 
For $1\leq i\leq k$ and $1\leq j\leq l$
we define elements $\tilde{\mathsf L}_{i,j}\in\mathscr A_{k,l,n}$
by the formula 
\begin{equation}
\label{eq:glLij}
\tilde{\mathsf L}_{i,j}:=\sum_{r=1}^n\tilde t_{i,r}\tilde \del_{j,r}
=\sum_{r=1}^nt_{m-i+1,r}^{}\del_{m-j+1,r}^{}.
\end{equation}
 By the same formula we can define analogous elements in   $\mathrm{gr}(\mathscr A_{k,l,n})$. By a slight abuse of the symbol $\mathrm{gr}(\cdot)$, we denote these elements of 
$\mathrm{gr}(\mathscr A_{k,l,n})$ 
 by 
$
\mathrm{gr}(\tilde{\mathsf L}_{i,j})$.
 
\begin{thmx}
\label{thm:C}
The algebras $\left(\mathscr A_{k,l,n}\right)_{(\epsilon_R)}$ 
and 
$\big(\mathrm{gr}(\mathscr A_{k,l,n})\big)_{(\epsilon_R)}$ are
 generated by the $\tilde{\mathsf L}_{i,j}$ 
and the 
$
\mathrm{gr}(\tilde{\mathsf L}_{i,j})$
respectively, 
  for $1\leq i\leq k$ and $1\leq j\leq l$. 
\end{thmx}
It would be interesting to relate Theorem~\ref{thm:C} to the quantized FFT of~\cite[Thm 6.10]{LZZ11} for example by a deformation argument. However, we are unable to establish such a connection. 


Our third main theorem 
(Theorem~\ref{thm:MainthmB} below)
explicitly describes the images in $\sPD=\sPD_{m\times n}$ of the Cartan subalgebras of $U_L$ and $U_R$. To state Theorem~\ref{thm:MainthmB}
 we need to define certain  elements of $\sPD$ that are constructed using $q$-determinants. 
%
Let $\mathbf i:=(i_1,\ldots,i_r)$ and $
\mathbf j:=(j_1,\ldots,j_r)$ be $r$-tuples of integers satisfying   $1\leq i_1<\cdots<i_r\leq m$ and $1\leq j_1<\cdots<j_r\leq n$. Define
quantum minors $M^{\mathbf i}_{\mathbf j}\in\sP$
and
$\oline M^{\mathbf i}_{\mathbf j}\in\sD$
by \begin{equation}
\label{eq:lem:qdet2exp}
M^{\mathbf i}_{\mathbf j}:=
\sum_{\sigma}(-q)^{\ell(\sigma)}
t_{i_{\sigma(1)},j_1}
\cdots
t_{i_{\sigma(r)},j_r}\quad
\text{and}\quad
\oline M^{\mathbf i}_{\mathbf j}:=
\sum_{\sigma}(-q^{-1})^{\ell(\sigma)}
\del_{i_{\sigma(1)},j_1}
\cdots
\del_{i_{\sigma(r)},j_r},
\end{equation}
where the summations are over  permutations in $r$ letters, and $\ell(\sigma)$ denotes the length of $\sigma$ (in the sense of Coxeter groups). 
 For $a,b,r\geq 1$ define $\mathbf D(r,a,b)\in\sPD_{a\times b}$~by 
\begin{equation}
\label{eq:Drab}
\mathbf D(r,a,b):=\sum_{\mathbf i}\sum_{\mathbf j} M^\mathbf i_\mathbf j{\oline M}^\mathbf i_\mathbf j,
\end{equation}
where the summation indices 
$\mathbf i:=(i_1,\ldots,i_r)$ and $\mathbf j:=(j_1,\ldots,j_r)$ 
satisfy
$1\leq i_1<\cdots <i_r\leq a$ and
$1\leq j_1<\cdots <j_r\leq b$.
We also set $\mathbf D(0,a,b)=1$. 
For $0\leq r \leq n$ and $1\leq k\leq n$ 
we define $\mathbf D_{k,r}\in\sPD_{m\times n}$ by 
\[
\mathbf D_{k,r}:=\boldsymbol e_{m\times k}^{m\times n}\left({\mathbf D}(r,m,k)\right).
\]
Similarly, for $0\leq r\leq m$ and $ 1\leq k\leq m$ we define $\mathbf D_{k,r}'\in\sPD_{m\times n}$ by
\[
\mathbf D'_{k,r}:=\boldsymbol e_{k\times n}^{m\times n}\left(\mathbf D(r,k,n)\right)
.\]
Note that $\mathbf D_{k,r}=\mathbf D'_{k,r}=0$ when $r>\min\{k,m,n\}$ and $\mathbf D_{k,0}=\mathbf D'_{k,0}=1$. 
Set
\begin{equation}
\mathbf R_a:=
\sum_{r=0}^{a}(q^2-1)^r\mathbf D_{a,r}
\ \text{ for }
1\leq a\leq n\quad\text{and}\quad
\mathbf L_b:=
\sum_{r=0}^{b}(q^2-1)^r\mathbf D'_{b,r}
\ \text{ for }
1\leq b\leq m.
\end{equation}
Furthermore, let $U_{\g h,L}$ and $U_{\g h,R}$ denote the Cartan subalgebras of $U_L$ and $U_R$, respectively (see Subsection~\ref{subsec:Uqglndf}). 
We set
\begin{equation}
\mathscr L_\g h:=
\phi_U(U_{\g h,L}\otimes 1)\quad,\quad
\Lhb:=\sPD\cap \mathscr L_\g h\quad,\quad
\mathscr R_\g h:=
\phi_U(1\otimes U_{\g h,R})\quad,\quad
\Rhb:=\sPD\cap \mathscr R_\g h. 
\end{equation}
\begin{thmx}

\label{thm:MainthmB}

The following statements hold. 
\begin{itemize}

\item[\rm (i)]
$\Rhb$  is generated by $\mathbf R_1,\ldots,\mathbf R_n$.

\item[\rm (ii)]
$\Lhb$ is  generated by $\mathbf L_1,\ldots,\mathbf L_m$.

\end{itemize}
\end{thmx}

In the proofs of our theorems  we borrow at least two key ideas from~\cite{LZZ11}. First, we 
use the bialgebra structure of $\sP_{n\times n}$ to define a map ${{}\Gamma}_{k,l,n}$ 
from
$\sP_{k\times l}$
onto the subalgebra of $U_R$-invariants
in $\mathrm{gr}(\mathscr A_{k,l,n})$ (assuming $n\geq \max\{k,l\}$). The map 
${{}\Gamma}_{k,l,n}$, given in Definition~\ref{dfn:PsI}, is similar to  the map introduced in~\cite[Lem. 6.11]{LZZ11}. Second, we define a new product on $\sP_{k\times l}$ such that the map ${{}\Gamma}_{k,l,n}$ becomes an isomorphism of algebras (see Definition~\ref{def:prodtPkl}). This product is analogous to the one defined  in~\cite[Lem. 6.13]{LZZ11}. However our product is given by a more complicated (and asymmetric) formula, because it needs to be simutaneously compatible with \emph{two} universal $R$-matrices.  As a consequence, establishing the desired properties of this product requires new ideas (see Section~\ref{sec:specialcasem=n}). 
Because of this, and the fact that unlike~\cite{LZZ11} the generators $\mathsf L_{i,j}$ (respectively, $\mathsf R_{i,j}$) or even their graded analogues are not weight vectors for the Cartan subalgebras of the two copies of $U_L$ (respectively, $U_R$) that act on $\sPD$, the proofs of Theorems~\ref{thm-Main-A} and~\ref{thm:C} become substantially more complicated than the analogous results of~\cite{LZZ11}.  See Subsection~\ref{subsec-prep-A(ii)} for more details.

 The results of this paper were obtained as part of a project on  Capelli operators for quantum symmetric spaces. From 
this standpoint, it is natural to ask if one can define quantized Weyl algebras in the latter setting and then realize the action of  (a large subalgebra of) the quantized enveloping algebra via elements of this Weyl algebra. We  address this question and its connection to Capelli operators in upcoming work~\cite{LSS22a,LSS22b}. 

The structure of this paper is as follows. In Section~\ref{sec:HopfTwis} we review the required background material on Hopf algebras and twisted tensor products. In Section~\ref{sec-PDmn} we construct and study the quantized Weyl algebra $\sPD$ and its variations, namely $\sPD^\mathrm{gr}$, $\mathscr A_{k,l,n}$ and $\mathscr A_{k,l,n}^\mathrm{gr}$.  The main goal of Section~\ref{subsec:Lifting} is to 
prove that under $\phi_U$ the elements $K_{\la_{L,a}}\otimes 1$ and $1\otimes K_{\la_{R,b}}$, defined in~\eqref{eq:xayb},
of the Cartan subalgebra of $U_{LR}$   are mapped into $\sPD$. In 
Section~\ref{section5} we compute explicit formulas for $\phi_U(K_{\la_{L,a}}\otimes 1)$ and $\phi_U(1\otimes K_{\la_{R,b}})$.
Section~\ref{sec-new6} is devoted to some general properties of the polarization operators
$\mathsf L_{i,j}$, $\mathsf R_{i,j}$ and their variants. 
In Section~\ref{sec:pfofThmA} we define the map ${{}\Gamma}_{k,l,n}$ and establish some of its properties. 
 Section~\ref{sec:specialcasem=n} defines 
the new product on $\sP_{k\times l}$ 
 and establishes its properties. Theorems~\ref{thm-Main-A}
and~\ref{thm:C} are proved in 
Section~\ref{subsec"ThmAm=n} and
 Theorem~\ref{thm:MainthmB} is proved in  Section~\ref{sec:genofLhbRhb}. Finally, Section~\ref{appendix} lists the commonly used notation in the paper.

\subsection*{Acknowledgement}
The research of S.S. was partially supported by NSF grants DMS-1939600, DMS-2001537, and Simons Foundation grant 509766. 
The research of H.S.
was partially supported 
by an NSERC Discovery Grant (RGPIN-2018-04044).

\section{Hopf algebras and deformed twisted tensor products}
\label{sec:HopfTwis}
 Throughout this section $\mathbb K$ will denote an arbitrary field and $H$ will be a Hopf algebra over $\mathbb K$. We denote the coproduct, counit, and antipode of $H$ by $\Delta$, $\epsilon$, and $S$. The opposite and co-opposite of $H$  are denoted by $H^\mathrm{op}$ and $H^\mathrm{cop}$
(we use the same notation for bialgebras as well). Throughout the paper, our notation 
for specific Hopf algebras $H$
will remain consistent with these choices. 

If $A$ is an associative algebra and $V$ is an $A$-module, 
a subspace $W\sseq V$ is called \emph{$A$-stable} if $A\cdot W\sseq W$. 
Finally, an associative algebra $A$ is called an \emph{$H$-module algebra} if it is equipped by an $H$-module structure such that the product of $A$ yields an $H$-module homomorphism $A\otimes A\to A$.   
\subsection{Local finiteness modulo an ideal}
\label{subsec:lofin}
Given a two-sided ideal $I$ of $H$ (considered as an associative   algebra),   we set $E(x,I):=\left\{\ad_y(x)+I\,:\,y\in H\right\}$
for $x\in H$, where $\ad_y(x):=\sum y_1xS(y_2)$ is the left adjoint action of $H$ (we use the Sweedler notation $\Delta(y)=\sum y_1\otimes y_2$ for the coproduct). We set
\begin{equation}
\label{eq:FHI}
\mathscr F(H,I):=
\left\{x\in H\,:\,\dim_\mathbb K E(x,I)<\infty\right\}.
\end{equation}
For $I=0$ this is the locally finite part of $H$ (in the sense of~\cite{JL94}), which we will denote by $\mathscr F(H)$. We have $E(xx',I)\sseq E(x,I)E(x',I)$ for $x,x'\in H$. Consequently, $\mathscr F(H,I)$ is a subalgebra of $H$.

\subsection{The finite dual of $H$}
\label{subsec:matrx}
Given a finite dimensional left $H$-module $V$, by the \emph{right dual} of $V$ we mean 
the dual space $V^*$ equipped with the $H$-action defined by 
$\lag x\cdot v^*,v\rag:=\lag v^*,S^{-1}(x)\cdot v\rag$ for $v^*\in V^*$ and $v\in V$, where  $\lag \cdot,\cdot\rag:V^*\otimes V\to \mathbb K$ is the canonical pairing. 
The \emph{matrix coefficients} of $V$ 
are the linear functionals $\sfm_{v^*,v}\in H^*$
defined by
\[
\sfm_{v^*,v}(x):=\lag v^*,x\cdot v\rag\quad\text{ for }x\in H,\ v\in V,\ v^*\in V^*.
\]
Indeed  $\sfm_{v^*,v}\in H^\circ$, where $H^\circ$ denotes the finite dual of $H$ (for the definition of $H^\circ$ see~\cite[Sec. 1.2.8]{KS97}).
Recall that $H^\circ$ has a canonical  Hopf algebra structure. The product of $H^\circ$ is given by 
\begin{equation}
\label{eq:proddual}
\la\mu(x):=\sum\la(x_1)\mu(x_2)\quad \text{for }\la,\mu\in H^\circ\text{ and }x\in H,
\end{equation}
where $\Delta(x)=\sum x_1\otimes x_2$.
Given two finite dimensional $H$-modules 
$V$ and $W$,  
 we have
\[
\sfm_{v^*\otimes w^*,v\otimes w}=\sfm_{v^*, v} \sfm_{w^*, w}\quad\text{for }v,w\in V\text{ and }v^*,w^*\in W.
\]
Let $\Delta^\circ$ denote the coproduct of $H^\circ$, so that $\Delta^\circ(\la)=\sum \la_1\otimes \la_2$ for $\la\in H^\circ$, where $\sum \la_1\otimes \la_2$ is uniquely determined by
\begin{equation}
\label{eq:laxy)copro}
\lag \la,xy\rag=\sum\lag \la_1,x\rag\lag \la_2,y\rag
\quad\text{for }x,y\in H.
\end{equation} 
If  $\{v_i\}_{i=1}^d$ is a basis of $V$ and $\{v_i^*\}_{i=1}^d$ is the dual basis of $V^*$, then 
$
\Delta^\circ(\sfm_{v^*,v})=\sum_{i=1}^d\sfm_{v^*,v_i}\otimes \sfm_{v_i^*,v}
$.
The following remark will be used in Section~\ref{sec-PDmn}.
\begin{rmk} 
\label{rmk:LRactionsHcirc}
Let $H^{\bullet}\sseq H^\circ$ be a sub-bialgebra. Then $H^{\bullet}$ is an $H$-module algebra with respect to \emph{right translation}, where the action is defined by $\lag x\cdot \la,y\rag:=\lag \la,yx\rag$ for $\la\in H^\bullet$ and $x,y\in H$.
If $H$ is equipped with a $\mathbb K$-linear map $x\mapsto x^\natural$ that yields an isomorphism of Hopf algebras $H\to H^\mathrm{op}$, then $H^{\bullet}$ has another $H$-module algebra structure defined by  
$\lag x\cdot \la,y\rag:=\lag \la,x^\natural y\rag$, which we call \emph{left translation}.
Given any homomorphism of associative algebras
$\tau:H\to H$, we can define $\tau$-twisted left and right translation actions of $H$ on $H^{\bullet}$, given respectively by the formulas\[
\lag x\cdot \la,y\rag:=\lag \la, \tau(x)^\natural y\rag
\quad\text{and}\quad
\lag x\cdot \la,y\rag:=\lag \la, y\tau(x)\rag.
\]
When $H^\bullet$ is equipped with either one of the two $\tau$-twisted actions, the following statements hold.
\begin{itemize}
\item[(i)] If $\tau:H\to H$ is a homomorphism of coalgebras, then $H^{\bullet}$ is  an $H$-module algebra. 
\item[(ii)] If $\tau:H\to H$ is an anti-homomorphism of coalgebras, then $H^{\bullet}$ is an  $H^\mathrm{cop}$-module algebra.
\end{itemize}

\end{rmk}

\subsection{The isotypic component of the trivial $H$-module}
\label{subsec-H-inv}
For any $H$-module $V$ we set \[
V_{(\epsilon)}:=\{v\in V\,:\,h\cdot v=\epsilon(h)v\},
\] where as before $\epsilon$ denotes the counit of $H$. 
Let  $\psi:H\to \End_\mathbb K(V)$ be the algebra homomorphism corresponding to this module structure (hence $h\cdot v=\psi(h)v$ for $h\in H$ and $v\in V$). 
We equip $\End_\mathbb K(V)$ with an $H$-module structure, defined by $h\cdot T:=\sum \psi(h_1)T\psi(S(h_2))$ for $h\in H$ and $T\in\End_\mathbb K(V)$, where $\Delta(h)=\sum h_1\otimes h_2$.
\begin{lem}
\label{rmk:Hopf-eps}
$\End_\mathbb K(V)_{(\epsilon)}=\End_\mathbb K(V)^{\psi(H)}$, where the right hand side is defined as in~\eqref{eq:YtoZ}. 
\end{lem}

\begin{proof}
The inclusion $\supseteq$ follows from \[
h\cdot T =\sum \psi(h_1)T\psi(S(h_2))=T\sum \psi(h_1)\psi(S(h_2))=T\psi\left(\sum h_1S(h_2)\right)=\epsilon(h)T,\]
for $T\in\End_\mathbb K(V)^{\psi(H)}$. 
For the inclusion $\subseteq $ note that if $T\in\End_\mathbb K(V)_{(\epsilon)}$
then
\begin{align*}
\psi(h) T=\sum \psi(h_1)\epsilon(h_2)T
&=\sum \psi(h_1)T\psi(S(h_2))\psi(h_3)\\
&=\sum\epsilon(h_1)T\psi(h_2)=
T\psi\left(\sum \epsilon(h_1)h_2\right)=T\psi(h). \qedhere\end{align*}
\end{proof}

\subsection{Braided triples, twisted tensor products and their deformations}
\label{subsec:twisted}
Let $\cC$ be a full subcategory of the category of $H$-modules that 
is closed under direct sums and tensor products. To ensure that $\cC$ is a monoidal category we assume that the trivial $H$-module (the one-dimensional vector space $\mathbb K$  equipped with the action $h\mapsto \epsilon(h)$ for $h\in H$) belongs to $\mathrm{Obj}(\cC)$. 

Assume that $\cC$ is \emph{braided}. The braiding 
$\check R$ of  $\cC$ is a natural family of $H$-module isomorphisms
\[
\check{R}_{V,W}:V\otimes W\to W\otimes V 
\quad
\text{for $V,W\in \mathrm{Obj}(\cC)$}
\]
 that satisfies the usual hexagon axioms (see for example~\cite[Def. 8.1.1]{EGNO15}). 
Henceforth we call $(H,\cC,\check R)$  a \emph{braided triple}.

Let $A$ and $B$ be two $H$-module algebras, with products $m_A$ and $m_B$. Assume that  $A,B\in \mathrm{Obj}(\cC)$.
\begin{lem}
\label{lem:RcheckBAprop}
The map $\check{R}_{B, A}:B\otimes A\to A\otimes B$ satisfies the following relations:
\begin{itemize}
\item[\rm (i)] 
$\check{R}_{B,A}(1\otimes a)=a\otimes 1$ for $a\in A$ and 
$\check{R}_{B,A}(b\otimes 1)=1\otimes b$ for $b\in B$. 
\item[\rm (ii)]
$
\check{R}_{B,A}\circ (\mathrm{id}_B\otimes m_A)=
(m_A\otimes \mathrm{id}_B)
(\mathrm{id}_A\otimes \check{R}_{B,A})
(\check{R}_{B,A}\otimes \mathrm{id}_A)
$.

\item[\rm (iii)] $\check{R}_{B,A}\circ (m_B\otimes \mathrm{id}_A)=(\mathrm{id}_A\otimes m_B)
(\check{R}_{B,A}\otimes \mathrm{id}_B)
(\mathrm{id}_B\otimes \check{R}_{B,A})$.

\end{itemize}

\end{lem}
\begin{proof}
This is well known, but we supply a proof because we did not find a reference. 

(i) Equip $\mathbb K$ with the canonical $H$-module structure induced by the counit 
$\epsilon:H\to \mathbb K$.
It is well known (see for example~\cite[Exer. 8.1.6]{EGNO15}) that one has a commuting triangle
\[
\xymatrix{\mathbb K\otimes  A\ar[rr]^{\check{R}_{\mathbb K,A}}\ar[dr] & & A\otimes \mathbb K\ar[dl]\\
& A}
\]
where the lower sides of the triangle are the canonical left and right unit isomorphisms. It follows immediately that $\check{R}_{\mathbb K,A}(1\otimes a )=a\otimes 1$. Furthermore, by naturality of $\check{R}$ the diagram
\[
\xymatrix@C+2pc{\mathbb K\otimes A \ar[r]^{\check{R}_{\mathbb K,A}} 
\ar[d]_{1_B\otimes \mathrm{id}_A}
&
A\otimes \mathbb K\ar[d]^{\mathrm{id}_A\otimes 1_B} \\
B\otimes A
\ar[r]^{\check{R}_{B,A}}
& A\otimes B}\]
 is commutative, hence $\check{R}_{B,A}(1\otimes a)=a\otimes 1$. The other
relation is proved similarly. 

(ii) Consider the diagram of maps below:
\[
\xymatrix@C+2pc{
A \otimes B\otimes A
\ar[dr]_{\mathrm{id}_A\otimes \check{R}_{B,A}}
 & \ar[l]_{\ \check{R}_{B,A}\otimes \mathrm{id}_A} 
 \ar[d]^{\check{R}_{B,A\otimes A}} B\otimes A\otimes A
\ar[r]^{\ \ \mathrm{id}_B\otimes m_A} & 
B\otimes A\ar[d]^{\check{R}_{B,A}}\\
& A\otimes A\otimes B \ar[r]_{\ \ m_A\otimes\mathrm{id}_B} & A\otimes B
}
\]
The square is commutative by naturality of $\check{R}$ (because $m_A$ and $m_B$ are $H$-module homomorphisms) and the triangle is commutative by the hexagon axiom. The assertion of (ii) follows from comparing two maps $B\otimes A\otimes A\to A\otimes B$ in the diagram: one is obtained by composition of the top and the right  edges of the square, the other is obtained by the outer edges of the triangle and the bottom edge of the square.  

(iii) Similar to (ii). 
 \end{proof}
\begin{dfn} 
\label{dfn:ARBB}
Let $A,B\in\mathrm{Obj}(\mathcal C)$ be $H$-module algebras. We denote the products of $A$ and $B$ by  $m_A$ and $m_B$, respectively. 
The 
\emph{$\check{R}$-twisted tensor product} of $A$ and $B$, denoted by $A\otimes_{\check{R}} B$, is  the vector space $A\otimes B$ equipped with the binary operation 
\begin{equation}
\label{eq:mAmBcirc}
(m_A\otimes m_B)\circ (\mathrm{id}_A\otimes \check{R}_{B,A}\otimes \mathrm{id}_B).
\end{equation}

\end{dfn}
\begin{prp}
$A\otimes_{\check{R}} B$  
is an associative algebra.
\end{prp}
\begin{proof}
This is well known and proved  for example in~\cite[Prop. 2.2]{VV94}.
\end{proof}
From Lemma~\ref{lem:RcheckBAprop}(i) it follows that 
 for $a,a'\in A$ and $b,b'\in B$ the product~\eqref{eq:mAmBcirc} satisfies
 \begin{equation}
 \label{eq:lem241}
 (a\otimes 1)(a'\otimes b)(1\otimes b')=(aa'\otimes bb').
 \end{equation}
\begin{rmk}
\label{rmk:HHHhH}
The vector space $A\otimes_{\check{R}} B=A\otimes B$ carries two module structures: as the outer tensor product 
over $\mathbb K$
of $H$-modules $A$ and $B$, it is 
an
$H\otimes H$-module. As the inner tensor product of $A$ and $B$, it is an $H$-module. Of course the latter $H$-module structure is obtained from the former one via restriction along the coproduct map $H\to  H\otimes H$.  
\end{rmk}

\begin{prp}
\label{prp:H-modual}
$A\otimes_{\check{R}} B$ is an $H$-module algebra with the $H$-module structure of 
Remark~\ref{rmk:HHHhH}. 

\end{prp}
\begin{proof}
 Given
$a,a'\in A$ and  $b,b'\in B$, if we write $\check{R}_{B,A}(b\otimes a')=\sum a''\otimes b''$ then for $x\in H$ we have
\begin{align*}
x\cdot\big( (a\otimes b)(a'\otimes b')\big)&=x\cdot \sum aa''\otimes b''b=
\sum (x_1\cdot (aa''))\otimes (x_2\cdot (b''b)).
\end{align*}
Since $A$ and $B$ are $H$-module algebras, from the latter equalities and~\eqref{eq:lem241}  it follows that 
\begin{align*}
x\cdot\big( (a\otimes b)(a'\otimes b')\big)
&=\sum (x_1\cdot a)(x_2\cdot a'')\otimes (x_3\cdot b'')(x_4\cdot b)\\
&=\sum
((x_1\cdot a)\otimes 1)
((x_2\cdot a'')\otimes \big(x_3\cdot  b'')\big)(1\otimes (x_4\cdot b))
.\end{align*}
Since
$\check R_{A,B}$  is an $H$-module isomorphism, from the latter equalities and~\eqref{eq:lem241}
 we have
\begin{align*}
x\cdot\big( (a\otimes b)(a'\otimes b')\big)
&
=\sum ((x_1\cdot a)\otimes 1)
\big(\check R_{B,A}\big((x_2\cdot b)\otimes (x_3\cdot a')\big)\big)(1\otimes (x_4\cdot b'))\\
&=\sum ((x_1\cdot a)\otimes (x_2\cdot b))((x_3\cdot a')\otimes (x_4\cdot b'))=\sum (x_1\cdot (a\otimes b))(x_2\cdot (a'\otimes b')).
\end{align*} 
Thus $A\otimes_{\check R} B$ is an $H$-module algebra. 
 \end{proof}

Let $E_A\sseq A$ and $E_B\sseq B$ be subspaces that generate $A$ and $B$, respectively.  Thus 
$A\cong T(E_A)/I_A$ and $B\cong T(E_B)/I_B$, where $T(X)$ denotes the tensor algebra on $X$, and  
$I_A$ and $I_B$ denote the corresponding ideals of relations.
We assume that $E_A$ and $E_B$ are $H$-stable. 
\begin{rmk}
\label{rmk:IAIB-Hstable}
The $H$-module structures
on $E_A$ and $E_B$ equip $T(E_A)$ and $T(E_B)$ with canonical $H$-module algebra structures. It is straightforward to verify that the maps $T(E_A)\to A$ and $T(E_B)\to B$ are $H$-module homomorphisms. In particular, $I_A$ and $I_B$ are $H$-stable subspaces of $T(E_A)$ and $T(E_B)$, respectively.
\end{rmk}
Consider the linear map \[
\gamma_{A,B}:E_A\otimes E_B\to T(E_A\oplus E_B)\ ,\ 
\gamma_{A,B}(a\otimes b):=ab.
\] 
Note that we can express $\gamma_{A,B}$ as the composition
\begin{equation}
\label{eq:gammaaa}
E_A\otimes E_B\xrightarrow{\sfi_A\otimes\sfi_B}T(E_A\oplus E_B)\otimes T(E_A\oplus E_B)\xrightarrow{\ a\otimes b\mapsto ab\ } T(E_A\oplus E_B),
\end{equation}
where $\sfi_A:E_A\to  T(E_A\oplus E_B)$ and $\sfi_B:E_B\to  T(E_A\oplus E_B)$ are canonical embeddings.
For the next lemma, recall that the $H$-module structure on $E_A\oplus E_B$ induces a canonical $H$-module algebra structure on $T(E_A\oplus E_B)$.
\begin{lem}
\label{gammaH}
$\gamma_{A,B}$ is an $H$-module homomorphism.
\end{lem}
\begin{proof}
Since $\sfi_A$, $\sfi_B$ and the product of $T(E_A\oplus E_B)$  are $H$-module homomorphisms, the assertion follows from the description of  $\gamma_{A,B}$ in~\eqref{eq:gammaaa}. 
\end{proof}
By the universal property of tensor algebras the map $E_A\oplus E_B\to A\otimes_{\check{R}} B$ given by the assignment $a\oplus b\mapsto a\otimes 1+1\otimes b$
induces a homomorphism of algebras 
\[
\pi:T(E_A\oplus E_B)\to A\otimes_{\check{R}} B.
\]
\begin{prp}
\label{prp:ARBAB}
Let $A$, $B$, $E_A$ and $E_B$ be as above.
Then $\pi$ induces an isomorphism of algebras
 \[
T(E_A\oplus E_B)/I_{A,B}
\cong
A\otimes_{\check{R}} B,
\]
where $I_{A,B}$ denotes the two-sided ideal of $T(E_A\oplus E_B)$ generated by $I_A$, $I_B$ and the relations 
\begin{equation}
\label{eq:ba-gammaAB}
ba-\gamma_{A,B}\circ  \check{R}_{B,A}(b\otimes a)\quad
\text{for }a\in E_A,\ b\in E_B.
\end{equation}
\end{prp}

\begin{proof}
We have  $\pi(E_A\oplus E_B)=E_{A,B}$ where $E_{A,B}:=(E_A\otimes 1)\oplus (1\otimes E_B)$. Furthermore, $A\otimes_{\check R} B$ is generated as an algebra by $E_{A,B}$. Thus $\pi$ is a surjection. 
Next we prove that $\pi$ is an injection.
From the definition of the product of $A\otimes_{\check R} B$ it follows that $I_{A,B}\sseq \ker \pi$. Thus, to complete the proof it suffices to verify the reverse inclusion.

By the relations~\eqref{eq:ba-gammaAB}, every element of
$T(E_A\oplus E_B)/I_{A,B}$
 can be expressed as a linear combination of  products of elements of $E_A$ and $E_B$ in which elements of $E_A$ occur before elements of $E_B$. Thus,  
\begin{equation}
\label{eq:T(EA)_A}
T(E_A\oplus E_B)=T(E_A)T(E_B)+I_{A,B}.
\end{equation}
Now choose $\{a_\alpha\}_{\alpha\in \mathcal I_A}\sseq T(E_A)$ 
and $\{b_\beta\}_{\beta\in \mathcal I_B}\sseq T(E_B)$
such that $\{I_A+a_\alpha\}_{\alpha\in \cI_A}$ is a basis of $T(E_A)/I_A\cong A$ and 
$\{I_B+b_\beta\}_{\beta\in \cI_B}$ is a basis of $T(E_B)/I_B\cong B$.
From~\eqref{eq:T(EA)_A} it follows that the elements $I_{A,B}+a_\alpha b_\beta$ for $\alpha\in I_A$ and $\beta\in I_B$ constitute a spanning set of $T(E_A\oplus E_B)/I_{A,B}$.

From~\eqref{eq:lem241} it follows that $\pi(a_\alpha)=
\oline a_\alpha\otimes 1$ for some $\oline a_\alpha\in A$ and $\pi(b_\beta)=1\otimes \oline b_\beta$ for some $\oline b_\beta\in B$. 
Also, $\pi(I_{A,B}+a_\alpha b_\beta)=\oline a_\alpha\otimes\oline b_\beta$.
The sets $\{\oline a_\alpha\}_{\alpha\in \cI_A}$ and 
$\{\oline b_\beta\}_{\beta\in \cI_B}$ are bases of $A$ and $B$, respectively. 
 Thus $\pi$ maps a spanning set  of $T(E_A\oplus E_B)/I_{A,B}$ bijectively onto the  basis 
$\{a_\alpha \otimes b_\beta\,:\,\alpha\in\cI_A,\beta\in\cI_B\}$ 
of $A\otimes_{\check{R}} B$. It follows that $\ker\pi\sseq I_{A,B}$. 
\end{proof}
Let $\psi:E_B\times E_A\to \mathbb K$ be an $H$-invariant bilinear form, that is
\begin{equation}
\label{eq:Psix1x2}
\sum \psi(x_1\cdot b,x_2\cdot a)=\epsilon(x)\psi(b,a)\quad
\text{for }a\in E_A,\ b\in E_B,\ x\in H,
\end{equation}
where as before $\epsilon:H\to \mathbb K$ denotes the counit of $H$. Let $I_{A,B,\psi}$ denote 
the two-sided ideal of $T(E_A\oplus E_B)$ that is generated by $I_A$, $I_B$, and relations of the form 
\begin{equation}
\label{eq:ba-gammaABpsi}
ba-\gamma_{A,B}\circ  \check{R}_{B,A}(b\otimes a)-\psi(b,a)\quad
\text{for }a\in E_A,\ b\in E_B.
\end{equation} 
\begin{dfn}
\label{dfn:ARpsiBBB}
We call the algebra 
$
A\otimes_{\check{R},\psi} B:=T(E_A\oplus E_B)/I_{A,B,\psi}
$ the
\emph{$\psi$-deformed $\check R$-twisted tensor product} of $A$ and $B$ relative to $E_A$ and $E_B$.
\end{dfn}
The $H$-module structure of $E_A\oplus E_B$ equips $T(E_A\oplus E_B)$ with a canonical $H$-module algebra structure. We have the following statement.
\begin{prp}
\label{prp:Hmd-desc}
The canonical $H$-module algebra structure on $T(E_A\oplus E_B)$ descends to an $H$-module algebra structure on 
$A\otimes_{\check R,\psi} B$. 
\end{prp}
\begin{proof}
It suffices to verify that $I_{A,B,\psi}$ is an $H$-stable subspace. By Remark~\ref{rmk:IAIB-Hstable}
the subspaces $I_A$ and $I_B$ of $T(E_A\oplus E_B)$ are $H$-stable. 
Thus it remains to show that  the relations \eqref{eq:ba-gammaABpsi} span an $H$-submodule of $T(E_A\oplus E_B)$.     

Any  $x\in H$ acts on $T^0(E_A\oplus E_B)\cong \mathbb K$ by $\epsilon(x)$.
Next let $a\in E_A$, $b\in E_B$ and $x\in H$.
By~\eqref{eq:Psix1x2} and Lemma~\ref{gammaH} we have
\begin{align*}
x\cdot \big(ba&-\gamma_{A,B}\circ \check R_{B,A}(b\otimes a)-\psi(b,a)\big)\\
&=
\sum
\big(
 (x_1\cdot b)(x_2\cdot a)-\gamma_{A,B}\circ\check R_{B,A}(x_1\cdot b\otimes x_2\cdot a)
-\psi(x_1\cdot b,x_2\cdot a)\big)
,\end{align*}
which is a sum of relations of the form~\eqref{eq:ba-gammaABpsi}.
\end{proof}
By the universal property of tensor algebras the maps $\sfi_A$ and $\sfi_B$ induce embeddings of associative algebras
$\sfi_A:T(E_A)\to T(E_A\oplus E_B)$ and $\sfi_B:T(E_B)\to T(E_A\oplus E_B)$. The latter maps induce
$H$-equivariant homomorphisms of associative algebras
\[
A\cong T(E_A)/I_A\xrightarrow{\ \ \bar\sfi_A\ \ } A\otimes_{\check R,\psi}B\quad\text{ and } \quad
B\cong T(E_B)/I_B\xrightarrow{\ \ \bar\sfi_B\ \ } A\otimes_{\check R,\psi}B.
\]
By tensoring the latter maps and then composing with the products of the algebras $T(E_A\oplus E_B)$ and  $A\otimes_{\check R,\psi} B$ we obtain the following
commutative diagram:
\begin{align}
\label{eq:commDOId}
\xymatrix@C+2pc{
T(E_A)\otimes T(E_B)
\ar[r]^{\sfi_A\otimes \sfi_B\quad\quad\quad} \ar[d]
&
T(E_A\oplus E_B)
\otimes 
T(E_A\oplus E_B) 
\ar[r]^{\quad \quad \quad a\otimes b\mapsto ab} \ar[d]& 
T(E_A\oplus E_B) \ar[d]\\
A\otimes B
\ar[r]^{\bar\sfi_A\otimes \bar \sfi_B\quad\quad\quad} 
&
(A\otimes_{\check R,\psi} B)\otimes (A\otimes_{\check R,\psi} B)\ar[r]^{\quad\quad\quad a\otimes b\mapsto ab}&  A\otimes_{\check R,\psi} B
}
\end{align}
In the above diagram the vertical maps are the canonical quotients. 
\begin{rmk}
\label{rmk:AoB->ARB}
From Remark~\ref{rmk:IAIB-Hstable}
it follows that 
$\sfi_A\otimes \sfi_B$
and
$\bar\sfi_A\otimes \bar\sfi_B$ are $H\otimes H$-module homomorphisms. Also,  the products of 
$T(E_A\oplus E_B)$ and $A\otimes_{\check R,\psi} B$
 are  $H$-module homomorphisms
(see
Proposition
~\ref{prp:Hmd-desc}). 
Thus the composition of the bottom horizontal maps in~\eqref{eq:commDOId}, which is given by \begin{equation}
\label{eq:AJlfjf}
A\otimes B\to A\otimes_{\check R,\psi}B\ ,\ 
a\otimes b\mapsto I_{A,B,\psi}+ab,
\end{equation}
is an $H$-module homomorphism (recall that an $H\otimes H$-module homomorphism is also an $H$-module homomorphism by restriction along the coproduct map
$H\to H\otimes H$). 
\end{rmk}

\subsection{Locally finite braided triples and their products}
\label{subsec:prdcts}
Given a braided triple
$(H,\cC,\check R)$ we set 
\begin{equation}
\label{eq:RvscheckR}
R_{V,W}:=\sigma_{W,V}^{}\circ \check{R}_{V,W}^{},
\end{equation}
 where \[
 \sigma_{W,V}^{}:W\otimes V\to V\otimes W\quad,\quad
w\otimes v\mapsto v\otimes w\] 
  is the standard flip map.


\begin{dfn}
\label{dfn:braidedfinite}
We say a braided triple $(H,\cC,\check R)$  is   \emph{locally finite} if it satisfies the following conditions.
\begin{itemize}
\item[(i)] Every $V\in\mathrm{Obj}(\cC)$ is a sum of its finite dimensional submodules that belong to $\mathrm{Obj}(\cC)$. 

\item[(ii)] For finite dimensional modules $V,W\in\mathrm{Obj}(\cC)$ there exist $\omega_{V,W},\oline \omega_{V,W}\in H\otimes H$ such that \[
R_{V,W}^{}(v\otimes w)=\omega_{V,W}^{}\cdot (v\otimes w)
\quad\text{and}\quad
R_{V,W}^{-1}(v\otimes w)=
\oline \omega_{V,W}^{}\cdot (v\otimes w)
\quad\text{ for }
v\otimes w\in V\otimes W.
\]
\end{itemize} 
\end{dfn}
Let us briefly explain the idea behind Definition~\ref{dfn:braidedfinite}. The braiding of the category of modules of the  Hopf algebra $U_q(\gl_n)$ 
is given by a formal series 
that does not belong to $U_q(\gl_n)\otimes U_q(\gl_n)$
(see Subsection~\ref{subsec:R-matUq}). Thus, in order to give a rigorous proof of Proposition~\ref{prp:AotBUrUR} below, we need 
to be able to replace this formal series locally by a finite 2-tensor.

\begin{lem}
\label{lem:HCRW1Wr}
Let $(H,\cC,\check R)$ be a locally finite braided triple and let  $V_1,\ldots, V_r,W_1,\ldots,W_r\in\mathrm{Obj}(\cC)$ be finite dimensional modules. 
Let $\omega_{V,W},\oline \omega_{V,W}\in H\otimes H$ be chosen as in Definition~\ref{dfn:braidedfinite}(ii), 
where $
V:=\bigoplus_{i=1}^r V_i
$
and 
$
W:=\bigoplus_{i=1}^r W_i
$. 
Then  \[
R_{V_i,W_j}(v_i\otimes w_j)=\omega_{V,W}\cdot (v_i\otimes w_j)
\quad\text{and} 
\quad
R^{-1}_{V_i,W_j}(v_i\otimes w_j)=\oline\omega_{V,W}\cdot (v_i\otimes w_j),
\]for all $v_i\in V_i$ and $w_j\in W_j$ where
$1\leq i,j\leq r$.
\end{lem}

\begin{proof}
This
follows from naturality of $\check R$ with respect to the canonical maps  $V_i\otimes W_j\into V\otimes W$. 
\end{proof}
We need to work with the braided triples associated to tensor products of Hopf algebras $U_q(\gl_n)$. To this end, we need the following definition. 
\begin{dfn}
\label{dfn:product-loc-fin}
Let $(H,\cC,\check R)$ and 
$(H',\cC',\check R')$ be locally finite braided triples.  
Set $H'':=H\otimes H'$ and equip $H''$ with the canonical tensor product Hopf algebra structure.  
Furthermore, let $\cC''$ be the full subcategory of the category of $H''$-modules whose objects are direct sums of modules of the form $V\otimes V'$ where $V\in \mathrm{Obj}(\cC)$ and $V'\in\mathrm{Obj}(\cC')$. We define a braiding on $\cC''$ as follows. First, for  $V,W\in\mathrm{Obj}(\cC)$ and
$V',W'\in \mathrm{Obj}(\cC')$ 
we set
\begin{equation}
\label{eq:dfnofRcheck}
\check R''_{V\otimes V',W\otimes W'}:=
\left(\check R_{V,W}^{}\right)_{13}
\circ \left(\check R'_{V',W'}\right)_{24}.
\end{equation}
Here 
$\left(\check R'_{V',W'}\right)_{24}$ means that $\check R'_{V',W'}$ acts on the $2^\mathrm{nd}$ and the $4^\mathrm{th}$ components of $(V\otimes V')\otimes (W\otimes W')$, resulting in a map
\[
(V\otimes V')\otimes (W\otimes W')\to (V\otimes W')\otimes (W\otimes V'),
\]
and $\big(\check R_{V,W}\big)_{13}$ is defined analogously. 
Next for $\widetilde{V}:=\bigoplus_i V_i^{}\otimes V'_i$
and 
$\widetilde{W}:=\bigoplus_j W_j^{}\otimes W'_j$
we define \[
\check R''_{\widetilde{V},\widetilde{W}}:=\bigoplus_{i,j}\check R''_{V^{}_i\otimes V'_i,W^{}_j\otimes W_j'}.
\]
\end{dfn}

 \begin{prp}
 \label{prp:H''C''R''}
Let 
$H''$, $\cC''$ and $\check R''$ be as in
Definition~\ref{dfn:product-loc-fin}. Then 
$(H'',\cC'',\check R'')$ is a locally finite braided triple.
\end{prp}
\begin{proof} 
It is trivial to check that $(H'',\cC'',\check R'')$ meets the condition of Definition~\ref{dfn:braidedfinite}(i). 
Next we show that $(H'',\cC'',\check R'')$ is a braided triple. It is straightforward to check that $\cC''$ is closed with respect to tensor products and arbitrary direct sums. The hexagon axioms for $\check R''$ follow from those for $\check R$ and $\check R'$. 
Next we verify naturality of $\check R''$. Since $\check R''$  is defined by expansion on direct sums of modules, 
it suffices to prove commutativity of the 
diagram
\[
\xymatrix@C+3pc{(V\otimes V')\otimes (W\otimes W')\ar[r]^{\check R_{V\otimes V',W\otimes W'}} \ar[d]_{f_{V}\otimes f_{W}}& (W\otimes W')\otimes (V\otimes V')\ar[d]^{f_W\otimes f_V}\\
(\underline V\otimes \underline V')\otimes (\underline W\otimes \underline W')
\ar[r]_{\check R_{\underline V\otimes \underline V',\underline W\otimes \underline W'}} & 
(\underline W\otimes \underline W')\otimes (\underline V\otimes \underline V')
}
\]
for all choices of  $V,W,\underline V,\underline W\in\mathrm{Obj}(\cC)$, $V',W',\underline V',\underline W'
\in\mathrm{Obj}(\cC')$, $f_V\in\mathrm{Mor}_{\cC''}(V\otimes V',\underline V\otimes \underline V')$ and $f_W\in \mathrm{Mor}_{\cC''}(W\otimes W',\underline W\otimes \underline W')$. The subtlety here is that we cannot assume that $f_V$ and $f_W$ can be decomposed into tensor products of maps on the tensor components. 
Since $\cC''$ satisfies the condition of Definition~\ref{dfn:braidedfinite}(i), using naturality of $\check R$ and $\check R'$ we can assume that all of the modules in the commutative diagram are finite dimensional.
Using Lemma~\ref{lem:HCRW1Wr} for 
the triples $(H,\cC,\check R)$ and $(H',\cC',\check R')$ it follows that 
$R_{V\otimes V',W\otimes W'}$
and $R_{\underline V\otimes \underline V',\underline W\otimes \underline W'} $ are given by
the left action of the same 2-tensor
in $H''\otimes H''$. Commutativity of the diagram follows immediately.

Finally,  we show that 
the condition of Definition~\ref{dfn:braidedfinite}(ii) holds. 
If $R_{V,W}$ and $R'_{V',W'}$ are equal to the  actions of $\omega_{V,W}\in H\otimes H$ and $\omega'_{V',W'}\in H'\otimes H'$ respectively, then $ R''_{V\otimes V',W\otimes W'}$ is equal to the  action of  \[
(\omega_{V,W})_{13}^{}(\omega'_{V',W'})_{24}^{}\in H''\otimes H''.
\]
This verifies
Definition~\ref{dfn:braidedfinite}(ii) 
 for pairs of $H''$-modules in $\cC''$ of the form $V\otimes V'$ and $W\otimes W'$. Using Lemma~\ref{lem:HCRW1Wr} for 
the triples $(H,\cC,\check R)$ and $(H',\cC',\check R')$, the claim follows 
for general $H''$-modules in $\cC''$. A similar argument can be given for 
$ \big(R''_{V\otimes V',W\otimes W'}\big)^{-1}$. 
\end{proof}
\subsection{A result on $H''$-stable subalgebras of $A\otimes_{\check R''}B$}

We continue with the notation of
Definition~\ref{dfn:product-loc-fin}.
Let $A$ and $B$ be two $H''$-module algebras such that $A,B\in\mathrm{Obj}(\cC'')$. Recall from Remark~\ref{rmk:HHHhH} and Proposition~\ref{prp:H-modual} that $A\otimes_{\check{R}''}B$ is an $H''$-module algebra and also an $H''\otimes H''$-module (but not necessarily an $H''\otimes H''$-module algebra). For $\omega\in H\otimes H$ let \[
T_\omega:A\otimes_{\check R''} B\to A\otimes_{\check R''} B
\] denote the linear endomorphism obtained by the action of $\omega_{13}$. Here by definition $(x\otimes y)_{13}:=x\otimes 1\otimes y\otimes 1$ for $x,y\in H$. 

\begin{lem}
\label{lem:Uxyreplacement} 
Let 
$(H,\cC,\check R)$ and
$(H',\cC',\check R')$ be locally finite braided triples and let
$(H'',\cC'',\check R'')$ be defined as in
Definition~\ref{dfn:product-loc-fin}. Let $A,B\in\mathrm{Obj}(\cC'')$ be $H''$-module algebras. Let $V''_A\sseq A$ and $V''_B\sseq B$ be finite dimensional $H''$-submodules such that 
$V''_A,V''_B\in \mathrm{Obj}(\cC'')$.
We decompose $V_A''$ and $V_B''$ 
as 
\[
V_A''=\bigoplus_{i\in \Omega_A} V_{i,A}^{}\otimes V_{i,A}'\quad\text{and}\quad
V_B''=\bigoplus_{j\in\Omega_B} V_{j,B}^{}\otimes V_{j,B}',
\]
where  $V_{i,A},V_{j,B}\in\mathrm{Obj}(\cC)$ and 
$V'_{i,A},V'_{j,B}\in\mathrm{Obj}(\cC')$ for  $i\in\Omega_A$ and $j\in\Omega_B$ (here $\Omega_A$ and $\Omega_B$ are index sets).
Let $\omega_{V,W},\oline\omega_{V,W}\in H\otimes H$ be chosen as in Definition~\ref{dfn:braidedfinite}(ii)
where
\[
V:=\bigoplus_{i\in\Omega_A} V_{i,A}\quad\text{and}\quad
W:=\bigoplus_{j\in\Omega_B} V_{j,B}.
\]
For $x,y\in H$ let $\omega_{x,y}\in H\otimes H$ be defined by $\omega_{x,y}:=\oline \omega_{V,W}(y\otimes x)\omega_{V,W}$. 
Then
\[
\left(
(\check R''_{B,A})^{-1}\circ
T_{x\otimes y}
\circ\check R''_{B,A}
\right)
\Big|_{V''_B\otimes V''_A}
=T_{\omega_{x,y}}
\quad\text{for }1\leq i\leq r.
\]

\end{lem}
\begin{proof}
By naturality of $\check R''$ it follows that 
\[
\left(\check R''_{B,A}\right)\big|_{V''_B\otimes V''_A}
=\check R''_{V''_B,V''_A}\quad\text{and}\quad
\left(\check R''_{B,A}\right)^{-1}\big|_{V''_A\otimes V''_B}=
\left(\check R''_{V''_B,V''_A}\right)^{-1}
.
\]
Thus  
from~\eqref{eq:dfnofRcheck} it follows that 
\[
\left(
(\check R''_{B,A})^{-1}\circ
T_{x\otimes y}
\circ\check R''_{B,A}
\right)
\big|_{V_{j,B}\otimes V'_{j,B}\otimes V_{i,A}\otimes V'_{i,A}}
=
\left(\check R_{V_{j,B}, V_{i,A}}\right)^{-1}_{13}\circ T_{x\otimes y}\circ
\left(
\check R_{V_{j,B}, V_{i,A}}\right)_{13}.
\] 
The relation~\eqref{eq:RvscheckR}
 and Lemma~\ref{lem:HCRW1Wr} for the triple $(H,\cC,\check R)$ imply
the assertion of the lemma. 
\end{proof}

\begin{prp}
\label{prp:AotBUrUR}
Let 
$(H,\cC,\check R)$ and
$(H',\cC',\check R')$ be locally finite braided triples and let
$(H'',\cC'',\check R'')$ be defined as in
Definition~\ref{dfn:product-loc-fin}.
Let $A,B\in\mathrm{Obj}(\cC'')$ be $H''$-module algebras. Let 
$\mathcal E$ be a subspace of $A\otimes_{\check R''} B$ and let $\cA$ denote the subalgebra of 
$A\otimes_{\check R''}B$ that is generated by $\mathcal E$.
Finally, let 
$\oline H$ be a  sub-bialgebra of $H$. 
\begin{itemize}
\item[\rm (i)] Assume that for every pair of finite dimensional modules $V,W\in \mathrm{Obj}(\cC)$, we can choose  
$\omega_{V,W},\oline\omega_{V,W}\in\oline H\otimes H$ that satisfy Definition~\ref{dfn:braidedfinite}(ii). If  
$\mathcal E$ is stable under the action of the subalgebra $H\otimes 1\otimes \oline H\otimes 1$ of $H''\otimes H''$, then so is $\cA$. 

\item[\rm (ii)] Assume that for every pair of finite dimensional modules $V,W\in \mathrm{Obj}(\cC)$, we can choose  
$\omega_{V,W},\oline\omega_{V,W}\in H\otimes \oline H$ that satisfy Definition~\ref{dfn:braidedfinite}(ii). If  
$\mathcal E$ is stable under the action of the subalgebra $\oline H\otimes 1\otimes  H\otimes 1$ of $H''\otimes H''$, then so is $\cA$.

 \end{itemize}
\end{prp}
\begin{proof}
For $r\geq 0 $ set $\cA_r:=\spn\left\{w^{(1)}\cdots w^{(r)}\,:\,w^{(i)}\in\mathcal E\text{ for }1\leq i\leq r\right\}$, so that $\cA:=\sum_{r\geq 0}\cA_r$. We prove by induction on $r$ that 
under the assumptions of (i) (respectively, of (ii)), the subspace 
$\cA_r$ is $(H\otimes 1\otimes \oline H\otimes 1)$-stable (respectively, 
$(\oline H\otimes 1\otimes  H\otimes 1)$-stable). 
For $r=0$ this is trivial and for $r=1$ this follows from the assumption on $\mathcal E$. Next assume $r>1$.


Choose any $\alpha,\beta\in A\otimes_{\check R'' }B$. We can express $\alpha$ and $\beta $ as finite sums $\alpha=\sum a\otimes b$
and $\beta=\sum a'\otimes b'$ where $a,a'\in A$ and $b,b'\in B$. For each pair $(b,a')$ that occurs in these summations 
we also express $\check R''_{B,A}(b\otimes a')$ as a summation,  that is,
\begin{equation}
\label{eq:RBA''}
\check R''_{B,A}(b\otimes a')=\sum a''\otimes b'',
\end{equation} where of course the $a''\in A$ and the $b''\in B$ depend on $b$ and $a'$. 
For $h,h'\in H$ we have
\begin{align*}
(h\otimes 1\otimes h'\otimes 1)\cdot \alpha\beta&=
(h\otimes 1\otimes h'\otimes  1)\cdot\left(
\sum (a\otimes b)(a'\otimes b')
\right)\\
&=
(h\otimes 1\otimes h'\otimes  1)\cdot 
\left(
\sum aa''\otimes b''b'\right)
=
\sum \big((h\otimes 1)\cdot (aa'')\big)\otimes \big((h'\otimes 1)\cdot (b''b')\big).
\end{align*}
Since both $A$ and $B$ are $H''$-module algebras, from the above calculation and~\eqref{eq:lem241} we obtain 
\begin{align}
\label{eq:firstdread}
\notag
(h\otimes 1\otimes h'\otimes 1)\cdot \alpha\beta
&=
\sum ((h_1\otimes 1) \cdot a)
((h_2\otimes 1)\cdot  a'')
\otimes 
((h'_1\otimes 1)\cdot b'')
((h'_2\otimes 1)\cdot b')\\
&=
\sum
\big(((h_1\otimes 1)\cdot a)\otimes 1\big)
\big(
(h_2\otimes 1\otimes h'_1\otimes 1)\cdot (a''\otimes b'')
\big)\big(1\otimes ((h'_2\otimes 1)\cdot b')\big),
\end{align}
where $\Delta(h)=\sum h_1\otimes h_2$ and $\Delta(h')=\sum h_1'\otimes h_2'$. From Lemma~\ref{lem:Uxyreplacement} it follows that for each pair $(h_2,h_1')$ that occurs on the right hand side 
of~\eqref{eq:firstdread} 
there exists a 2-tensor $\omega_{h_2,h_1'}=\sum u_{h_2,h_1'}\otimes u'_{h_2,h_1'}$ 
in $H\otimes H$ such that 
\begin{align}
\label{eq:seconddread}
\notag
(h_2\otimes 1\otimes h'_1\otimes 1)\cdot\sum  (a''\otimes b'')&=
\check R''_{B,A}\left(\big(\omega_{h_2,h'_1}\big)_{13}\cdot (b\otimes a')\right)\\
&=
\check R''_{B,A} 
\left(
\big((u^{}_{h_2,h'_1}\otimes 1)\cdot b\big)
\otimes 
\big((u'_{h_2,h'_1}\otimes 1)\cdot a'\big)
\right).
\end{align}
Note that by Lemma~\ref{lem:Uxyreplacement}
the 2-tensors 
$\omega_{h_2,h_1'}$ are of the form
\begin{equation}
\label{eq:barrww}
\omega_{h_2,h_1'}=\oline \omega_{V,W}(h_1'\otimes h_2)\omega_{V,W},
\end{equation}
where $\omega_{V,W},\oline\omega_{V,W}\in H\otimes H$ satisfy the constraint of 
Definition~\ref{dfn:braidedfinite}(ii) for suitable $H$-modules $V,W$. 
By comparing~\eqref{eq:firstdread} and~\eqref{eq:seconddread} and then using~\eqref{eq:lem241} we obtain\begin{align*}
(h\otimes 1\otimes h'\otimes 1)\cdot \alpha\beta
&=
\sum
\left(\big((h_1\otimes 1)\cdot a\big)\otimes 
\big(( u_{h_2,h'_1}\otimes 1)\cdot b\big)
\right)
\left(
\big(( u'_{h_2,h'_1}\otimes 1)\cdot a'\big)\otimes 
\big((h_2'\otimes 1)\cdot b'\big)\right)\\
&=\sum \big((h_1\otimes 1\otimes  u_{h_2,h'_1}\otimes 1)\cdot \alpha\big)
\big((u'_{h_2,h'_1}\otimes 1\otimes h'_2\otimes 1)\cdot \beta\big).
\end{align*}
In what follows we complete the proofs of part (i) and part (ii).

(i) 
It suffices to prove that for $h\in H$ and $h'\in \oline H$ and $w^{(1)},\ldots,w^{(r)}\in \mathcal E$ we have  
\begin{equation}
\label{eq:assert}
(h\otimes 1\otimes h'\otimes  1)\cdot \left(w^{(1)}\cdots w^{(r)}\right)\in \cA_r.
\end{equation} 
We set
 \[
\alpha:=w^{(1)}\quad\text{and}\quad
\beta:=w^{(2)}\cdots w^{(r)}.
\] 
Since $\Delta(\oline H)\sseq \oline H\otimes \oline H$, we can assume that $h'_1,h'_2\in\oline H$, hence by considering the first component on both sides of~\eqref{eq:barrww} we obtain $u_{h_2,h_1'}\in \oline H$.
Thus,
by the induction hypothesis we have
\begin{equation}
\label{eq:indicls}
\big((h_1\otimes 1\otimes u_{h_2,h_1'}\otimes 1)\cdot \alpha\big)\in \mathcal E\sseq \cA_1.
\quad\text{and}\quad
\big(( u'_{h_2,h'_1}\otimes 1\otimes h'_2\otimes 1)\cdot \beta\big)\in\cA_{r-1}.
\end{equation}
The inclusion~\eqref{eq:assert} follows from $\cA_1\cA_{r-1}=\cA_r$. 

(ii) It suffices to prove~\eqref{eq:assert} for 
$h\in\oline H$ and $h'\in H$. We define $\alpha$ and $\beta$ as in (i). As $h\in \oline H$, we can assume that $h_1,h_2\in \oline H$,  hence $u'_{h_2,h_1'}\in\oline H$. Again the induction hypothesis implies~\eqref{eq:indicls} and 
the inclusion~\eqref{eq:assert} follows from $\cA_1\cA_{r-1}=\cA_r$. 
\end{proof}

\subsection{Braidings and matrix coefficients}
\label{subsec::quas}
Let $(H,\cC,\check{R})$ be a locally finite braided triple.
As in~\eqref{eq:RvscheckR} we set 
$R_{V,W}:=\sigma_{W,V}\circ\check R_{V,W}$ for $V,W\in\mathrm{Obj}(\cC)$.
Let $H^{\circ}_\cC\sseq H^\circ$ denote the 
$\Z$-span of matrix coefficients of the $H$-modules that belong to $\mathrm{Obj}(\cC)$.
Since $\cC$ is closed under direct sums and tensor products, indeed $H^{\circ}_\cC$ is a sub-bialgebra of $H^\circ$. 
 
Let $f\in H^\circ_\cC$ be a sum 
of matrix coefficients of $V_1,\ldots,V_{N}\in\mathrm{Obj}(\cC)$,
that is \[
f=\sum_{i=1}^{N}\sfm_{v_i^*,v_i^{}},
\]
where $v_i^{}\in V_i$ and $v_i^*\in V_i^*$. Similarly, let $g\in H^\circ_\cC$ be a sum of matrix coefficients of $W_1,\ldots, W_{N'}\in\mathrm{Obj}(\cC)$, that is $
g=\sum_{i=1}^{N'}\sfm_{w_i^*,w_i^{}}$.
Choose $\omega_{V,W},\oline\omega_{V,W}\in H\otimes H$
that satisfy the condition of Definition~\ref{dfn:braidedfinite}(ii) for 
$V:=\bigoplus_{i=1}^{N} V_i$ and $W:=\bigoplus_{i=1}^{N'} W_i$.
We define 
\begin{equation}
\label{eq:Rfg}
R(f\otimes g):=\omega_{V,W}\cdot (f\otimes g)\quad\text{and}\quad
R^{-1}(f\otimes g):=\oline \omega_{V,W}\cdot (f\otimes g),
\end{equation}
where the actions of $\omega_{V,W}$ and $\oline \omega_{V,W}$ are by right translation on tensor components. This means that for example if $\omega_{V,W}=\sum r\otimes r'\in H\otimes H$, then
\[
\big(\omega_{V,W}\cdot (f\otimes g)\big)(h\otimes h'):=
\sum f(hr)g(h'r').
\]
%
%
\begin{rmk}
(i)
For $x,y\in H$ we have
\begin{align}
\label{RJljdj}
\left(\omega_{V,W}\cdot \left(\sfm_{v_i^*,v_i^{}}\otimes \sfm_{w_j^*,w_j^{}}\right)\right)
(x\otimes y)=
\lag v_i^*\otimes w_j^*,(x\otimes y)\cdot R_{V,W}(v_i\otimes w_j)\rag,
\end{align}
hence the left hand side of~\eqref{RJljdj} is independent of the choice of $\omega_{V,W}$. Summing over  $i$ and $j$ it follows that  $R(f\otimes g)$ is independent of the choice of $\omega_{V,W}$ as well.
The latter observation and Lemma~\ref{lem:HCRW1Wr} imply that  
$R(f\otimes g)$ does not depend on how $f$ and $g$ are expressed as sums of matrix coefficients. An analogous statement holds for $R^{-1}(f\otimes g)$. 

(ii)  From Lemma~\ref{lem:HCRW1Wr} it also follows that the formulas~\eqref{eq:Rfg} extend  to linear maps
\[
R,R^{-1}:H^\circ_\cC\otimes H^\circ_\cC\to H^\circ_\cC\otimes H^\circ_\cC. 
\]
Indeed the latter maps $R$ and $R^{-1}$ are mutual inverses. 
\end{rmk}
We define \begin{equation}
\label{eq:<fog,R>}
\lag f\otimes g,R\rag:=\left(R(f\otimes g)\right)(1\otimes 1).
\end{equation} 
%
%
In the rest of this subsection $\Delta^\circ(f)=\sum f_1\otimes f_2$ and $\Delta^\circ(g)=\sum g_1\otimes g_2$. 
\begin{lem}
\label{lem:fog,Rexpand}
Let $f,g\in H^\circ_\cC$. Then the following relations hold.
\begin{itemize}
\item[\rm(i)]
$
R(f\otimes g)=\sum f_1\otimes g_1\lag f_2\otimes g_2,R\rag$.

\item[\rm(ii)]
$f\otimes g=\sum \lag R^{-1}(f_1\otimes g_1),R\rag f_2\otimes g_2$.
\end{itemize}
\end{lem}
\begin{proof}
It suffices to verify the assertion 
when
$f$ and $g$ are  matrix coefficients of finite dimensional $H$-modules $V,W\in\mathrm{Obj}(\cC)$.
Suppose that  $f:=\sfm_{v^*,v}$ and $g:=\sfm_{w^*,w}$. If $R_{V,W}(v\otimes w)=\sum \tilde v\otimes \tilde w$ then by~\eqref{RJljdj} we have 
$R(f\otimes g)=\sum \sfm_{v^*,\tilde v}\otimes \sfm_{w^*,\tilde w}$.
Choose dual bases $\{v_i\}$ and $\{v_i^*\}$ for $V$ and $V^*$ and dual bases $\{w_j\}$ and $\{w_j^*\}$ for $W$ and $W^*$. 
Using the coproduct and counit identities of $H^\circ$ we have
\begin{align*}
\sum f_1\otimes g_1\lag f_2\otimes g_2,R\rag
&=\sum \sfm_{v^*,v_i}\otimes \sfm_{w^*,w_j}
\left(
\lag v_i^*,\tilde v\rag\lag w_j^*,\tilde w\rag 
\right)\\
&=\sum \lag v_i^*,\tilde v\rag\sfm_{v^*,v_i}\otimes 
\lag w_j^*,\tilde w\rag
\sfm_{w^*,w_j}
=\sum\sfm_{v^*,\tilde v}\otimes \sfm_{w^*,\tilde w}
=R(f\otimes g).
\end{align*}
This proves (i). For (ii) note that if $R^{-1}_{V,W}(v\otimes w)=\sum \tilde v\otimes \tilde w$ then 
$R^{-1}(f\otimes g)=\sum \sfm_{v^*,\tilde v}\otimes \sfm_{w^*,\tilde w}$.
Next we write $R^{-1}_{V,W}(v_i\otimes w_j)=\sum \tilde v^{i,j}\otimes \tilde w^{i,j}$ for each pair of indices $i,j$. Then
\begin{align*}
\sum
\lag R^{-1}(f_1\otimes g_1),R\rag
f_2\otimes g_2&= \sum
\lag \sfm_{v^*,\tilde v^{i,j}}\otimes 
\sfm_{w^*,\tilde w^{i,j}},R\rag \sfm_{v_i^*,v}\otimes \sfm_{w_j^*,w}\\
&=\sum 
\lag v^*,v_i\rag
\sfm_{v_i^*,v}
\otimes 
\lag w^*,w_j\rag
\sfm_{w_j^*,w}=\sfm_{v^*,v}\otimes \sfm_{w^*,w}=f\otimes g.\qedhere
\end{align*}
\end{proof}

Fix finite dimensional $H$-modules $V,W\in\mathrm{Obj}(\cC)$. Let $\{v_i\}_{i=1}^d$ and $\{w_i\}_{i=1}^{d'}$ be bases of $V$ and $W$. Also,  let
$\{v_i^*\}_{i=1}^d$ and $\{w_i^*\}_{i=1}^{d'}$ be the dual bases of $V^*$ and $W^*$. 
We denote the matrix entries of  $R_{V,W}$ in the basis $v_i\otimes w_j$ by $R_{ij}^{kl}$, so that
\[
R_{V,W}(v_i\otimes w_j)=\sum_{k,l}R_{ij}^{kl}v_k\otimes w_l.
\]
Set
$
\mathsf t^V_{a,b}:=\sfm_{v_a^*,v_b}
$ and 
$\mathsf t^W_{a,b}:=
\sfm_{w_a^*,w_b}
$.
Then 
$R_{ij}^{kl}=\lag \mathsf{t}^V_{k,i}\otimes \mathsf{t}^W_{l,j},R\rag$, 
so that as in~\cite[Lem. 7.12]{Ja96}
we have the  well known relations 
\begin{equation}
\label{eq:tVtWRR}
\sum_{k,l}
\lag \mathsf{t}^V_{k,i}\otimes \mathsf{t}^W_{l,j},R\rag 
\mathsf{t}_{a,l}^W\mathsf t_{b,k}^V
=\sum_{k,l}
\lag \mathsf{t}^V_{b,k}\otimes \mathsf{t}^W_{a,l},R\rag
\mathsf t_{k,i}^V\mathsf t_{l,j}^W
\quad
\text{ for all }i,j,a,b.
\end{equation}
From~\eqref{eq:tVtWRR} it follows that 
\begin{equation}
\label{eq:g1f1R}
\sum g_1f_1\lag f_2\otimes g_2,R\rag=\sum f_2g_2\lag f_1\otimes g_1,R\rag\quad
\text{for }f,g\in H^{\circ}_\cC.
\end{equation}


\section{The $q$-Weyl algebra $\sPD$}
\label{sec-PDmn}
In this section we construct $\sPD$ as a deformed twisted tensor product of the algebras $\sP$ and $\sD$ with respect to the univeral $R$-matrix of $U_{LR}$. Recall that $\Bbbk:=\C(q)$.

\subsection{The algebra $U_q(\gl_n)$}
\label{subsec:Uqglndf}
For $n\in\N$, the quantized enveloping algebra   $U_q(\gl_n)$ is the $\Bbbk$-algebra generated by $E_i,F_i$ for $1\leq i\leq n-1$ and $K_{\eps_i}^{\pm 1}$ for $1\leq i\leq n$, that satisfy the  relations 
$K_{\eps_i}^{}K_{\eps_i}^{-1}=K_{\eps_i}^{-1}K_{\eps_i}^{}=1$, 
$K_{\eps_i}K_{\eps_j}=K_{\eps_j}K_{\eps_i}$,
\[
K_{\eps_i}E_jK_{\eps_i}^{-1}=q^{{\llbracket i,j\rrbracket}-\llbracket i,j+1\rrbracket}E_j\ ,\
K_{\eps_i}F_jK_{\eps_i}^{-1}=q^{-{\llbracket i,j\rrbracket }
+{\llbracket i,j+1\rrbracket}}F_j\ ,\
E_iF_j-F_jE_i=
\llbracket i,j\rrbracket
\frac{K_i-K_i^{-1}}{q-q^{-1}},
\] where $K_i:=K_{\eps_i}K_{\eps_{i+1}}^{-1}$
and 
\begin{equation}
\label{eq:[[a,b]]}
\llbracket a,b\rrbracket:=\begin{cases}
1 & \text{if }a=b\\
0&\text{if }a\neq b,
\end{cases}
\end{equation}
as well as the quantum Serre relations. 
For $\la:=\sum_{i=1}^n m_i\eps_i\in
\Z\eps_1+\cdots+\Z\eps_n$ we set \begin{equation}
\label{eq:Klaa;f}
K_\la:=\prod_{i=1}^nK_{\eps_i}^{m_i}.
\end{equation}
The \emph{Cartan subalgebra} of $U_q(\gl_n)$ is the subalgebra spanned by the $K_\la$ for $\la\in\Z\eps_1+\cdots+\Z\eps_n$. We denote the Cartan subalgebras of $U_L\cong U_q(\gl_m)$ and $U_R\cong U_q(\gl_n)$ by $U_{\g h,L}$ and $U_{\g h,R}$, respectively. 
Following~\cite{KS97} for the choice of 
 the coproduct $\Delta$ on $U_q(\gl_n)$,  we set 
\[
\Delta(E_i):=E_i\otimes K_i+1\otimes E_i\quad,\quad
\Delta(F_i):=F_i\otimes 1+K_i^{-1}\otimes F_i\quad,\quad
\Delta(K_{\eps_i}):=K_{\eps_i}\otimes K_{\eps_i}
.\]
The counit and antipode of $U_q(\gl_n)$ are given by
\[\epsilon(E_i)=\epsilon(F_i)=0\ \ ,\ \ \epsilon(K_{\eps_i}^{\pm 1})=1\ \ ,\ \ 
S(E_i)=-E_iK_i^{-1}\ \ ,\ \ S(F_i)=-K_iF_i\ \ ,\ \  S(K_{\eps_i})=K_{\eps_i}^{-1}.
\] 

\subsection{The universal $R$-matrix of $U_q(\gl_n)$}
\label{subsec:R-matUq}
Recall that a Hopf algebra $H$ is called
quasitriangular if it has a \emph{universal $R$-matrix}, i.e., if there exists an invertible 2-tensor $\EuScript R\in H\otimes H$ satisfying
\begin{equation}
\label{eq:DEltacop}
\Delta^{\mathrm{cop}}=\EuScript R\Delta \EuScript R^{-1}
\quad,\quad
(\Delta\otimes \mathrm{id})(\EuScript R)=\EuScript R_{13}\EuScript R_{23}
\quad,\quad
(\mathrm{id}\otimes \Delta)(\EuScript R)=\EuScript R_{13}\EuScript R_{12}.
\end{equation}
Strictly speaking, $U_q(\gl_n)$ is not quasitriangular because 
the formal series that is usually called the universal $R$-matrix of $U_q(\gl_n)$ indeed belongs to a topological tensor product 
$U_h(\gl_n)\widehat{\otimes} U_h(\gl_n)$ where $U_h(\gl_n)$ denotes the  $h$-adic Drinfeld--Jimbo quantum group. 
However, it turns out that 
the setting of braided triples is a rigorous way to work with this universal $R$-matrix.

Let $\cC^{(n)}$ denote the full subcategory of the category of $U_q(\gl_n)$-modules whose objects are direct sums of irreducible finite dimensional $U_q(\gl_n)$-modules with highest weight of the form $q^{\sum_{i=1}^n\la_i\eps_i}$, where $\la_1\geq\cdots\geq \la_n$ are integers. Such modules are sometimes called modules of type $(1,\ldots,1)$. We define a braiding on $\cC^{(n)}$ as follows. First we fix a formal series description of the universal $R$-matrix for $U_q(\gl_n)$.
For more details see \cite[Thm 3.108]{VY20} or \cite[Sec. 8.3.2]{KS97}.

\begin{dfn}
\label{dfn:univRmatgln}
Given $n\in\N$, the standard root vectors of 
$U_q(\gl_n)$ are
\[
E_{\eps_i-\eps_j}:=(-1)^{j-i-1}[E_i,[\ldots,E_{j-1}]_{q^{-1}}]_{q^{-1}}\quad\text{and}\quad
F_{\eps_i-\eps_j}:=(-1)^{j-i-1}[F_{j-1},[\ldots,F_i]_{q}]_{q},
\]
where 
$1\leq i<j\leq n$ and 
$[x,y]_{q^{\pm 1}}:=xy-q^{\pm 1}yx$.
We set
\begin{equation}
\label{eq:formulaRnn}
\EuScript R^{(n)}:=\left(
e^{h\sum_{i=1}^n H_i\otimes H_i}
\right)
\prod_{i=1}^{n\choose 2}
\mathrm{Exp}_q
\left(
(q-q^{-1})E_{\beta_i}\otimes F_{\beta_i}
\right),
\end{equation}
with the conventions  $e^{hH_i}:=K_{\eps_i}$, $e^h:=q$, 
$\mathrm{Exp}_q(x):=\sum_{r\geq 0}q^{r\choose 2}\frac{x^r}{[r]_q!}$ and $
\beta_{i+\frac{j(j-1)}{2}}:=\eps_i-\eps_{j+1}
$ for 
$1\leq i\leq j\leq n-1$. Also, set 
\[\underline{\EuScript{R}}^{(n)}:=
\left(\EuScript R^{(n)}\right)^{-1}_{21}.
\]
\end{dfn}

In what follows, we need $\EuScript R^{(n)}$ to define  $\sP$ and $\sD$, and we need $\underline{\EuScript{R}}^{(n)}$ to define  $\sPD^\mathrm{gr}$ and $\sPD$.
The formal series $\EuScript R^{(n)}$ and $\underline{\EuScript{R}}^{(n)}$ equip the category $\cC^{(n)}$ with two braidings which we describe below. 
Given $V,W\in\mathrm{Obj}(\cC^{(n)})$,  the formal series~\eqref{eq:formulaRnn} defines a linear map
\[
\EuScript R^{(n)}_{V,W}:V\otimes W\to V\otimes W.
\] 
To give sense to the action of 
$\EuScript R^{(n)}$ on $V\otimes W$ we make the following two observations. First, for any $v\otimes w\in V\otimes W$ all but finitely many terms of $
\mathrm{Exp}_q
\left(
(q-q^{-1})E_{\beta_i}\otimes F_{\beta_i}
\right)
$ vanish on $v\otimes w$. Second, if
$v\in V$ and $w\in W$ are weight vectors of 
weights $q^{\mu}$ and $q^\nu$ with 
$\mu:=\sum_{i=1}^n\mu_i\eps_i$
and $\nu:=\sum_{i=1}^n\nu_i\eps_i$ respectively, then 
the action of $
e^{h\sum_{i=1}^n H_i\otimes H_i}
$ on $v\otimes w$ is by multiplication by the scalar $q^{\lag \mu,\nu\rag}$
where \begin{equation}
\label{eq:pairingg}
\lag \mu,\nu\rag:=\sum_{i=1}^n \mu_i\nu_i.
\end{equation} By a similar reasoning, the action of $\underline{\EuScript R}^{(n)}$ yields  linear maps \[
\underline{\EuScript R}^{(n)}_{V,W}:V\otimes W\to V\otimes W.
\] 
It is well known that by setting  \
\begin{equation}
\label{Bradings-eq}
\check{\EuScript R}^{(n)}_{V,W}:=\sigma_{V,W}\circ 
{\EuScript R}^{(n)}_{V,W}\quad\text{and}\quad
\check{\underline{\EuScript R}}^{(n)}_{V,W}:=\sigma_{V,W}\circ 
\underline{\EuScript R}^{(n)}_{V,W}
\end{equation}
we obtain braidings 
	on $\cC^{(n)}$, which we will denote by $\check{\EuScript R}^{(n)}$ and 
$\check{\underline{\EuScript R}}^{(n)}$
\begin{prp}
\label{prp:Uqglnlocfin}
Set $H:=U_q(\gl_n)$, $\cC:=\cC^{(n)}$, and $\check R:=\check{{\EuScript R}}^{(n)}_{}$ or $\check{\underline{\EuScript R}}^{(n)}$. Then $(H,\cC,\check R)$  is a locally finite braided triple.
\end{prp}
\begin{proof}
The only assertion that we need to prove is the property of Definition~\ref{dfn:braidedfinite}(ii). Fix  finite dimensional $V,W\in \mathrm{Obj}\left(\cC^{(n)}\right)$. We construct an element of $U_q(\gl_n)\otimes U_q(\gl_n)$ that acts on $V\otimes W$ as $\EuScript R^{(n)}_{V,W}$.  Since $E_{\beta_i}^{N_i}\otimes F_{\beta_i}^{N_i}$ vanishes on $V\otimes W$ 
when  $N_i$ is sufficiently large,  the exponential factor \[\mathrm{Exp}_q
\left(
(q-q^{-1})E_{\beta_i}\otimes F_{\beta_i}
\right)
\] can be replaced by a finite sum. Next we provide a finite 2-tensor 
that replaces $
e^{h\sum_{i=1}^n H_i\otimes H_i}
$. 
Let $\mu^{(1)},\ldots,\mu^{(N)}$ be the distinct weights of $W$.
Choose $\nu\in \Z\eps_1+\cdots+\Z\eps_n$ such that the values $\lag\nu,\mu^{(i)}\rag$ are mutually distinct numbers. For $1\leq i\leq N$ define $T_i\in U_q(\gl_n)$ by
\[
T_i:=\prod_{\tiny\begin{array}{c}1\leq j\leq N\\j\neq i\end{array}}
\left(
\frac{K_\nu-q^{\lag \nu,\mu^{(j)}\rag}}{q^{\lag\nu,\mu^{(i)}\rag}-q^{\lag \nu,\mu^{(j)}\rag}}
\right).
\]
Then $T_i$ acts by 0 or 1  on the $\mu^{(j)}$-weight space of $W$, depending on if  $j\neq i $ or $j=i$ respectively. It follows that the action of
$
\sum_{i=1}^N
K_{\mu^{(i)}}\otimes T_i
$ on $V\otimes W$ is identical to the action of $
e^{h\sum_{i=1}^n H_i\otimes H_i}
$. 
Thus  the action of $\EuScript R^{(n)}$ on $V\otimes W$ is identical to the action of a (finite) 2-tensor in $U_q(\gl_n)\otimes U_q(\gl_n)$. Analogous constructions can be given for 
$\left(\EuScript R^{(n)}_{V,W}\right)^{-1}$ 
and 
$\left(\underline{\EuScript R}^{(n)}_{V,W}\right)^{\pm 1}$.
\end{proof}


As in Subsection~\ref{subsec:matrx} let $U_q(\gl_n)^\circ$ be the finite dual of $U_q(\gl_n)$. 
\begin{dfn}
\label{dfn:Ubul}
Let $U_q(\gl_n)^\bullet\sseq U_q(\gl_n)^\circ$ denote the sub-bialgebra that is spanned by  matrix coefficients of objects of $\cC^{(n)}$. 
\end{dfn}
For $f,g\in U_q(\gl_n)^\bullet$  we define $\big\lag f\otimes g,\EuScript R^{(n)}\big\rag$ and 
$\big\lag f\otimes g,\underline{\EuScript R}^{(n)}\big\rag$
as in 
\eqref{eq:<fog,R>}.
For finite dimensional $V\in \mathrm{Obj}\left(\cC^{(n)}\right)$ the right dual $V^*$ also belongs to 
$\mathrm{Obj}\left(\cC^{(n)}\right)$. From this and the fact that $S$ and $S^{-1}$ are conjugate by an element of the Cartan subalgebra it follows that 
if $f\in U_q(\gl_n)^\bullet$ then $f\circ S^{\pm 1}\in U_q(\gl_n)^\bullet$. 
It is well known 
(see for example~\cite[Sec. 8.1.1]{KS97}) that
\begin{equation}
\label{dfn:R21R}
\left(\EuScript R^{(n)}\right)_{21}=
\left(\underline{\EuScript R}^{(n)}\right)^{-1}=
(1\otimes S^{-1})\left(\underline{\EuScript R}^{(n)}\right)\quad
\text{and}\quad
(S\otimes 1)
\left({\EuScript R}^{(n)}\right)=\left(
{\EuScript R}^{(n)}\right)^{-1}
\end{equation}
Consequently, for $f,g\in U_q(\gl_n)^\bullet$ we have  
\begin{equation}\label{lem:fxgoS}
\big\lag f\otimes (g\circ S^{-1}),\underline{\EuScript R}^{(n)}\big\rag
=\big\lag g\otimes f,{\EuScript R}^{(n)}\big\rag\quad
\text{and}
\quad
\big\lag (f\circ S)\otimes g,{\EuScript R}^{(n)}\big\rag=
\big\lag g\otimes f,\underline{\EuScript R}^{(n)}
\big\rag.
\end{equation}

\subsection{The involution $x\mapsto x^\natural$}
\label{subsec:naturalUqcirc}
It is well known (for example see~\cite[Sec. 1.4]{No96}) that there exists a unique
$\Bbbk$-linear  isomorphism of Hopf algebras
\begin{equation}
\label{eq:invnatr}
U_q(\gl_n)\to U_q(\gl_n)^\mathrm{op}\ ,\ 
x\mapsto x^\natural,
\end{equation}
such that
\begin{equation}
\label{eq:x->x*}
E_i^\natural:=qK_iF_i\quad,\quad
F_i^\natural :=q^{-1}E_iK_i^{-1}
\quad,\quad
K_{\eps_i}^\natural :=K_{\eps_i}.
\end{equation}
\begin{lem}
\label{lem:rel-S-natural}
$S(x^\natural)=S^{-1}(x)^\natural$ for $x\in U_q(\gl_n)$.
\end{lem}
\begin{proof}
Both sides are automorphisms of the algebra $U_q(\gl_n)$. Therefore it suffices to verify that they agree on the $E_i$, the $F_i$, and the $K_{\eps_i}$. This is a straightforward calculation. 
\end{proof}

By the canonical duality between $U_q(\gl_n)$ and $U_q(\gl_n)^\circ$, the map~\eqref{eq:invnatr} induces an isomorphism of Hopf algebras $U_q(\gl_n)^\circ\to \left(
U_q(\gl_n)^\circ\right)^\mathrm{cop}$. 
We denote the latter map by $u\mapsto u^\natural$ as well, so that 
\begin{equation}
\label{eq:naturalduali}
\lag u^\natural,x\rag=\lag u,x^\natural\rag
\quad\text{for }u\in U_q(\gl_n)^\circ\text{ and }x\in U_q(\gl_n). 
\end{equation}

\subsection{The algebras $\sP_{n\times n}$ and $\sD_{n\times n}$}
\label{subsec:PnnnDnan}
From now on we
denote the standard positive system of the root system of $\gl_n$
by
  \[
\Delta^+_n:=\{\eps_i-\eps_j\,:\,1\leq i<j\leq n\}\] 
Let $V^{(n)}$ denote the irreducible $U_q(\gl_n)$-module of highest weight $q^{-\eps_n}$ (all highest weights are considered with respect to $\Delta_n^+$).
Thus $V^{(n)}\cong \Bbbk^n$ as a vector space and the homomorphism of algebras $U_q(\gl_n)\to \End_{\Bbbk}\left(V^{(n)}\right)$ is uniquely determined by the assignments
\[
K_{\eps_i}\mapsto\mathsf 1+(q^{-1}-1)\mathsf E_{i,i}\quad,\quad
E_i\mapsto\mathsf E_{i+1,i}\quad,\quad
F_i\mapsto \mathsf E_{i,i+1},
\]
where the $\mathsf E_{i,j}$ are the elementary matrix units associated to the standard basis $\{e_i\}_{i=1}^n$ of $V^{(n)}$ and $\mathsf 1:=\sum_{i=1}^n\mathsf E_{i,i}$.
Using~\eqref{eq:formulaRnn} the $R$-matrix of $V^{(n)}\otimes V^{(n)}$ can be computed directly, and we obtain 
\[
\EuScript R^{(n)}_{V^{(n)}, V^{(n)}}=\sum_{1\leq i\leq n} q\mathsf E_{i,i}\otimes \mathsf E_{i,i}
+\sum_{1\leq i\neq j\leq n}\mathsf E_{i,i}\otimes \mathsf E_{j,j}+(q-q^{-1})\sum_{1\leq j<i\leq n}
\mathsf E_{i,j}\otimes\mathsf  E_{j,i}.
\]
For $1\leq i,j\leq n$ let $t_{i,j}$ denote the matrix coefficient $\sfm_{e_i^*,e_j}$ of $V^{(n)}$. By~\eqref{eq:g1f1R}  the  $t_{i,j}$ satisfy the following relations:
\begin{itemize}
\item[(R1)]$t_{k,i}t_{k,j} = qt_{k,j}t_{k,i}$, $t_{i,k}t_{j,k} = qt_{j,k}t_{i,k}$\ \  for $i<j$.
\item[(R2)] $t_{i,l}t_{k,j}  = t_{k,j}t_{i,l}, \ t_{i,j}t_{k,l} -t_{k,l}t_{i,j} = (q-q^{-1})t_{i,l}t_{k,j}$\ \  for $i<k$ and $ j<l$.
\end{itemize}
Similarly, let $\breve V^{(n)}$ denote the irreducible $U_q(\gl_n)$-module with highest weight $q^{\eps_1}$. Again $\breve V^{(n)}\cong \Bbbk^n$ as vector spaces, but  the map $U_q(\gl_n)\to\End_\Bbbk\big(\breve V^{(n)}\big)$ is uniquely determined by the assignments
\[
K_{\eps_i}\mapsto \mathsf 1+(q-1)\mathsf E_{i,i}\quad,\quad
E_i\mapsto \mathsf  E_{i,i+1}\quad,\quad
F_i\mapsto \mathsf  E_{i+1,i}.
\]
Indeed $\breve V^{(n)}\cong \left(V^{(n)}\right)^*$.
The $R$-matrix of $\breve V^{(n)}\otimes \breve V^{(n)}$ 
is
\[
{\EuScript
R}_{\breve V^{(n)},\breve V^{(n)}}^{(n)}=
\sum_{1\leq i\leq n} q\mathsf E_{i,i}\otimes \mathsf E_{i,i}
+\sum_{1\leq i\neq j\leq n}\mathsf E_{i,i}\otimes \mathsf E_{j,j}+(q-q^{-1})\sum_{1\leq i<j\leq n}
\mathsf E_{i,j}\otimes \mathsf E_{j,i}.
\]
If $\del_{i,j}$ for $1\leq i,j\leq n$ denotes the matrix coefficient $\sfm_{e_i^*,e_j}$ of $\breve V^{(n)}$, then again from~\eqref{eq:g1f1R} it follows that  the $\del_{i,j}$ satisfy relations similar to those between the $t_{i,j}$, with $q$ replaced by $q^{-1}$. Equivalently,
\begin{itemize}
\item[(R1$'$)]$\partial_{k,j}\partial_{k,i} = q\partial_{k,i}\partial_{k,j}$, $\partial_{j,k}\partial_{i,k} = q\partial_{i,k}\partial_{j,k}$\ \ for $i<j$.
\item[(R2$'$)] $\partial_{k,j}\partial_{i,l}  = \partial_{i,l}\partial_{k,j}, \partial_{k,l}\partial_{i,j} -\partial_{i,j}\partial_{k,l} = (q-q^{-1})\partial_{k,j}\partial_{i,l}$\ \ for $i<k$ and $j<l$.
\end{itemize} 
\begin{dfn}
\label{dfn-f--P-D}
Let $\sP_{n\times n}$ denote
 the subalgebra of 
$U_q(\gl_n)^\circ$
 generated by the $t_{i,j}$, for $1\leq i,j\leq n$. Similarly, let $\sD_{n\times n}$ denote the subalgebra of $U_q(\gl_n)^\circ$ generated by the $\del_{i,j}$, for $1\leq i,j\leq n$. 
 \end{dfn}
It is well known 
(for example see~\cite{Ta92})
that the relations (R1)--(R2) yield a presentation of $\sP_{n\times n}$ by generators and relations. 
Since $\sD_{n\times n}\cong \sP_{n\times n}^\mathrm{op}$, a similar  statement holds for $\sD_{n\times n}$.
%
%
  From Section~\ref{subsec:matrx} it follows that both  $\sP_{n\times n}$ and $\sD_{n\times n}$ are bialgebras with the coproducts satisfying 
\[
t_{i,j}\mapsto\sum_k t_{i,k}\otimes t_{k,j}\quad\text{and}\quad
\del_{i,j}\mapsto\sum_k \del_{i,k}\otimes \del_{k,j},
\]
and the counits satisfying $t_{i,j},\del_{i,j}\mapsto \llbracket i,j\rrbracket$. Henceforth we denote the coproducts of $\sP_{n\times n}$ and $\sD_{n\times n}$ by $\Delta_\sP$ and $\Delta_\sD$, respectively. 
 In the proof of Lemma~\ref{lem:naturalonTD} we  use the relations 
\[
t_{i,j}(K_{\eps_k})=\llbracket i,j\rrbracket q^{-\llbracket i,k\rrbracket}
\quad,\quad
t_{i,j}(E_{k})=\llbracket i,k+1\rrbracket \llbracket j,k\rrbracket
\quad
,\quad
t_{i,j}(F_{k})=\llbracket i,k\rrbracket \llbracket j,k+1\rrbracket.
\]

\begin{lem}
\label{lem:naturalonTD} 
$t_{i,j}^\natural=t_{j,i}^{}$ and 
$\del_{i,j}^\natural=\del_{j,i}^{}$, where 
$t_{i,j}^\natural$ and $\del_{i,j}^\natural$ are defined by~\eqref{eq:naturalduali}.
\end{lem}
\begin{proof}
We only give the proof for the $t_{i,j}$, as the argument for the $\del_{i,j}$ is similar. The assertion follows if we
verify that 
\begin{equation}
\label{eq:<>=<>xx}
\lag t_{i,j},x^\natural\rag=\lag t_{j,i},x\rag\quad \text{for $x\in U_q(\gl_n)$.}
\end{equation} 
It suffices to check~\eqref{eq:<>=<>xx} when $x$ is a generator of $U_q(\gl_n)$, because if~\eqref{eq:<>=<>xx} holds for  $x,y\in U_q(\gl_n)$ then \[
\lag t_{i,j},(xy)^\natural\rag=\lag t_{i,j},y^\natural x^\natural\rag=\sum_a \lag t_{i,a},y^\natural\rag \lag t_{a,j},x^\natural\rag=\sum_a \lag t_{a,i},y\rag \lag t_{j,a},x\rag=
\lag t_{j,i},xy\rag,\]
hence~\eqref{eq:<>=<>xx} also holds for $xy$. When $x$ is one of the standard generators of $U_q(\gl_n)$, checking~\eqref{eq:<>=<>xx} is a direct calculation. For example for $x=E_k$ we have 
\[
\lag t_{i,j},E_k^\natural\rag=
q\sum_a \lag t_{i,a},K_k\rag\lag t_{a,j},F_k\rag.
\]
The right hand side vanishes unless
$j=k+1$ and $i=k$, in which case we have $\lag t_{i,j},E_k^\natural\rag=1$. 
It follows immediately that $\lag t_{i,j},E_k^\natural\rag=\lag t_{j,i},E_k\rag$ for all $i,j,k$. 
\end{proof}

According to Remark~\ref{rmk:LRactionsHcirc}, the canonical  
$U_q(\gl_n)\otimes U_q(\gl_n)$-module structure of $U_q(\gl_n)$ by left and right translation equips both $\sP_{n\times n}$ and $\sD_{n\times n}$ with $U_q(\gl_n)\otimes U_q(\gl_n)$-module algebra structures. 
Our next goal is to describe the latter actions  explicitly (all of the actions are from the left side).


Let $\cR_\sD$ be the action of $U_q(\gl_n)$ on $\sD_{n\times n}$ by right translation, as in Remark~\ref{rmk:LRactionsHcirc}.
We have
 \[
 \cR_\sD(x) u=\sum \lag u_2,x\rag u_1
 \quad\text{for }x\in U_q(\gl_n),\ u\in \sD_{n\times n},\]
 where as usual
$\Delta(u)=\sum u_1\otimes u_2$.
Similarly, let $\cL_\sD$ be the action of $U_q(\gl_n)$ on $\sD_{n\times n}$ by left translation. Thus
\[
\cL_\sD(x)u=\sum \lag u_1,x^\natural \rag u_2\quad
\text{for }
x\in U_q(\gl_n),\ u\in\sD_{n\times n}.
\]
By Remark~\ref{rmk:LRactionsHcirc} both $\cL_{\sD}$ and $\cR_\sD$ equip $\sD_{n\times n}$ with $U_q(\gl_n)$-module algebra structures. 

Next we define the left and right $U_q(\gl_n)$-actions on $\sP_{n\times n}$. For $c\in \Bbbk$ let $\xi_c$ denote the unique automorphism of $U_q(\gl_n)$ defined by \[
\xi_c(E_i):=cE_i\quad,\quad 
\xi_c(F_i):=c^{-1}F_i\quad,\quad\xi_c(K_{\eps_i})=K_{\eps_i}.
\]
\begin{lem}
\label{lem:howxc}
$t_{i,j}\circ \xi_c=c^{i-j}t_{i,j}$ and 
$\del_{i,j}\circ \xi_c=c^{j-i}\del_{i,j}$.
\end{lem}
\begin{proof}
We only give the proof of the assertion for the $t_{i,j}$. In this case we need to verify the equality 
\begin{equation}
\label{tijoxi_c}
t_{i,j}(\xi_c(x))=c^{i-j}t_{i,j}(x)
\end{equation}
for $x\in U_q(\gl_n)$. This is a straightforward calculation in the special case where  $x$ is one of the standard generators of $U_q(\gl_n)$. To complete the proof
of~\eqref{tijoxi_c}  note that if~\eqref{tijoxi_c} holds for $x$ and $x'$, then it also holds for $xx'$ because
\[
t_{i,j}(\xi_c(xx'))=\sum_{a=1}^n t_{i,a}(\xi_c(x))t_{a,j}(\xi_c(x'))
=
\sum_{a=1}^n
c^{(i-a)+(a-j)}t_{i,a}(x)t_{a,j}(x')=c^{i-j}t_{i,j}(xx').
\qedhere
\]
\end{proof}
The map
\[
\Xi:U_q(\gl_n)\to U_q(\gl_n)^\mathrm{op,cop}\ ,\ 
x\mapsto \xi_{-1/q}(S(x))
\]
is an isomorphism of Hopf algebras. Thus, the  pullback of $\Xi$  induces an isomorphism of Hopf algebras
$
U_q(\gl_n)^\circ\to \left( U_q(\gl_n)^\circ\right)^\mathrm{op,cop}
$, given by  $u\mapsto u\circ\Xi$. 
Set
\begin{equation}
\label{eq:iota-uoXi}
\iota(u):=u\circ \Xi\quad\text{ for }u\in U_q(\gl_n)^\circ.
\end{equation}
\begin{lem}
\label{lem:howiota}
We have 
\begin{equation}
\label{eq:defn-iota-map}
\iota(t_{i,j})=\del_{j,i}\quad\text{for }1\leq i,j\leq n.
\end{equation}
In particular, the restriction of $\iota$ to $\sP_{n\times n}$ is an isomorphism of bialgebras $\iota:\sP_{n\times n}\to\sD_{n\times n}^\mathrm{op,cop}$.
\end{lem}
\begin{proof}
We need to verify $t_{i,j}(\Xi(x))=\del_{j,i}(x)$ for $x\in U_q(\gl_n)$. It suffices to check the latter relation for the standard generators of $U_q(\gl_n)$, and this special case follows from a direct calculation. 
\end{proof}
Next note that the map
\[
\underline\iota:\sP_{n\times n}\to \sD_{n\times n}^\mathrm{op}\ ,\ 
\underline\iota(u):=\iota(u)^\natural
\]
is an isomorphism of bialgebras (indeed $\underline\iota(t_{i,j})=\del_{i,j}$). Let 
$\cR_{\sD,\tau}$ and $\cL_{\sD,\tau}$ denote the $\tau$-twists of $\cR_\sD$ and $\cL_\sD$
(see Remark~\ref{rmk:LRactionsHcirc}), where we set
$\tau(x):=S^{-1}(x)^\natural$ for $x\in U_q(\gl_n)$. For $u\in\sP_{n\times n}$ and $x\in U_q(\gl_n)$ set
\[
\cR_\sP(x)u:=\underline\iota^{-1}\left(\cR_{\sD,\tau}(x)\underline\iota(u)\right)
\quad\text{ and }\quad
\cL_\sP(x)u:=\underline\iota^{-1}\left(\cL_{\sD,\tau}(x)\underline\iota(u)\right).
\]
By Remark~\ref{rmk:LRactionsHcirc}(ii), $\cR_{\sD,\tau}$ and $\cL_{\sD,\tau}$ equip $\sD_{n\times n}$ with $U_q(\gl_n)^\mathrm{cop}$-module algebra structures. It follows that $\cR_\sP$ and $\cL_\sP$ equip $\sP_{n\times n}$ with $U_q(\gl_n)$-module structures. 
%
%
%
%
By a direct calculation
\[
\cR_\sP(x)u=\sum \lag \iota(u_2),S^{-1}(x)\rag u_1
\quad\text{and}\quad
\cL_\sP(x)u=
\sum
\lag \iota(u_1),S^{-1}(x)^\natural \rag u_2
,\]
for $x\in U_q(\gl_n)$ and $u\in \sP_{n\times n}$.
Using~\eqref{eq:iota-uoXi}, Lemma~\ref{lem:rel-S-natural} and the relation $S^2=\xi_{q^2}$ we obtain
\[
\cR_\sP(x)u=\sum \lag u_2\circ\xi_{-1/q},x\rag u_1\quad
\text{and}
\quad
\cL_\sP(x)u=
\sum\lag (u_1\circ\xi_{-q})^\natural,x\rag u_2.
\]

\subsection{The algebras $\sP$ and $\sD$}
\label{subsec:Now3.5}
Our next goal is to extend the constructions of Subsection~\ref{subsec:PnnnDnan} to the $m\times n$ case. 
\begin{dfn}
\label{dfn-PmnPnn}
Let $m$ and $n$ be positive integers and set $N:=\max\{m,n\}$. We define
the algebra $
\sP:=\sP_{m\times n}
$
(respectively, $\sD:=\sD_{m\times n}$) to be the subalgebra of $
\sP_{N\times N}$ 
(respectively, $\sD_{N\times N}$)
that is generated by the $t_{i,j}$ (respectively, the $\del_{i,j}$) where $1\leq i\leq m$ and  $1\leq j\leq n$. 
\end{dfn}
Note that by restricting the $U_q(\gl_N)\otimes U_q(\gl_N)$-module algebra structures
on $\sP_{N\times N}$ and $\sD_{N\times N}$ we obtain $U_{LR}$-module algebra structures on $\sP$ and $\sD$. Let us describe these $U_{LR}$-module algebras more precisely. 
  For convenience we first assume that $m\leq n$.  Then the subalgebra of 
$U_q(\gl_n)$ 
generated by $E_i,F_i,K^{\pm1} _{\eps_j}$ for 
$1\leq i\leq m-1$ and $1\leq j\leq m$
is isomorphic to $U_q(\gl_m)\cong U_L$.  With this identification of $U_L$ with a subalgebra of $U_q(\gl_n)$ we have the following lemma.
\begin{lem}
\label{lem:3.4.2-}
Suppose that $m\leq n$ and we identify $U_L$ with a subalgebra of $U_q(\gl_n)$ as above. Then $\mathcal L_\sD(x)\sD\sseq\sD$ and $\mathcal L_\sP(x)\sP\sseq\sP$ for $x\in U_L$.
\end{lem}
\begin{proof}
We only prove this for $\mathcal L_\sD$ (for $\mathcal L_\sP$ the proof is similar). Since $\sD_{n\times n}$ is a $U_L$-module algebra, it suffices to check that if $i\leq m$ 
and  $x$ is one of the standard generators of  $U_L$
then $\mathcal L_\sD(x)\del_{i,j}\in \sD$. If $x=E_k$ for $1\leq k\leq m-1$ then
\begin{align}
\label{eq:LDcalc}
\mathcal L_\sD(E_k)\del_{i,j}&=
\sum_{a=1}^n\lag\del_{i,a},E_k^\natural\rag \del_{a,j}
=
\sum_{a=1}^n\lag\del_{i,a},qK_{k}F_k\rag \del_{a,j}
=q\sum_{a=1}^n\sum_{b=1}^n
\lag \del_{i,b},K_k\rag\lag\del_{b,a},F_k\rag\del_{a,j}.
\end{align}
We have $
\lag \del_{b,a},F_k\rag=\lag e_b^*,\mathsf E_{k+1,k}e_a\rag=\llbracket a,k\rrbracket\llbracket b,k+1\rrbracket
$.
In particular, $\lag \del_{b,a},F_k\rag=0$ unless $a\leq m-1$. It follows that the right hand side of~\eqref{eq:LDcalc} is a linear combination of the $\del_{a,j}$ where $a\leq m-1$, hence it lies in $\sD$. The calculations for the cases  $x=F_k$ for $1\leq k\leq m-1$ and $x=K_{\eps_k}^{\pm 1}$ for $1\leq k\leq m$ are similar. 
\end{proof}

Lemma~\ref{lem:3.4.2-} implies that 
$\sP$ and $\sD$ are $U_{LR}$-stable subspaces of $\sP_{n\times n}$ and $\sD_{n\times n}$, where we consider 
\[
U_{LR}=U_q(\gl_m)\otimes U_q(\gl_n)
\cong U_L\otimes U_R
\]
as a subalgebra of $U_q(\gl_n)\otimes U_q(\gl_n)$ via the aforementioned embedding $U_L\into U_q(\gl_n)$. Thus, $\sP$ and $\sD$ inherit $U_{LR}$-module algebra structures from $\sP_{n\times n}$ and $\sD_{n\times n}$. 

Henceforth we mostly drop the symbols $\cL_\sP$, $\cL_\sD$, $\cR_\sP$ and $\cR_\sD$ from our notation. Instead, we use the notation 
\[
(x\otimes y)\cdot u
\]
to denote the action of $x\otimes y\in U_{LR}$ on $u\in \sP$ (or $u\in\sD$).  
The actions of  $x\otimes y\in U_L\otimes U_R$ on $u\in\sP$ and on $v\in \sD$ are given explicitly by the formulas
\begin{equation}
\label{eq:xoy=<><>u2}
(x\otimes y)\cdot u
=\sum\lag \iota(u_1),S^{-1}(x)^\natural\rag\lag \iota(u_3),S^{-1}(y)\rag u_2
=
\sum
\lag (u_1\circ \xi_{-q})^\natural,x\rag 
\lag u_3\circ\xi_{-1/q},y\rag u_2
\end{equation}
and 
\begin{equation}
\label{eq:xoy=<><>u3}
(x\otimes y)\cdot v=\sum 
\lag v_1,x^\natural\rag\lag v_3,y\rag v_2
=
\sum \lag v_1^\natural,x\rag\lag v_3,y\rag v_2,
\end{equation}
with $(\Delta_\sP\otimes 1)\circ \Delta_\sP(u)=\sum u_1\otimes u_2\otimes u_3$ and
$(\Delta_\sD\otimes 1)\circ \Delta_\sD(v)=\sum v_1\otimes v_2\otimes v_3$ in Sweedler notation, where $\Delta_\sP$ and $\Delta_\sD$ denote the coproducts of $\sP_{n\times n}$ and $\sD_{n\times n}$, respectively.  
\begin{rmk}
In the case  $m>n$ the construction of the $U_{LR}$-action is the same, except that we embed $U_q(\gl_m)\otimes U_q(\gl_n)$ in $U_q(\gl_m)\otimes U_q(\gl_m)$. However, 
formulas~\eqref{eq:xoy=<><>u2} and \eqref{eq:xoy=<><>u3} remain the same. 
\end{rmk}
\begin{dfn}
\label{dfn:themapPhiU}
The map $\phi_U:U_{LR}\to \End_\Bbbk(\sP)$ is the homomorphism of algebras induced by the action~\eqref{eq:xoy=<><>u2}. 
\end{dfn}


\subsection{$U_{LR}$-module decomposition of $\sP$ and $\sD$}
\label{subsec:ULR-mdec}
For any integer partition $\la$ satisfying $\ell(\la)\leq n$, where $\ell(\la)$ denotes the length of $\la$, let $V_\la$ denote the irreducible finite dimensional $U_R$-module of type 
$(1,\ldots,1)$
with highest weight $q^{\sum_{i}\la_i\eps_i}$ (with respect to $\Delta_n^+$). If $\la$ satisfies $\ell(\la)\leq m$ we use the same notation $V_\la$ to denote the analogously defined module of $U_L$.

The algebras $\sP$ and $\sD$ are naturally graded by degree of monomials. For $d\geq 0$ let $\sP^{(d)}$ (respectively, $\sD^{(d)}$) denote the graded component of degree $d$ of $\sP$ (respectively, $\sD$). Furthermore,  let $\Lambda_{d,r}$ be the set of integer partitions $\la$ such that $\ell(\la)\leq d$ and $|\la|=r$, where $|\la|$ denotes the size of $\la$. 
The following proposition is well known and its proof can be found for example 
in~\cite{NYM93,Ta92,Zh02}.  
\begin{prp}
\label{prp:glmglndecom}
Set $d:=\min\{m,n\}$. 
We have isomorphisms of $U_{LR}$-modules
\[
\sP^{(r)}\cong\bigoplus_{\la\in \Lambda_{d,r}}
V_\la^*\otimes V_\la^*\quad\text{and}\quad
\sD^{(r)}\cong\bigoplus_{\la\in \Lambda_{d,r}}
V_\la\otimes V_\la.
\]
\end{prp}

\begin{rmk}
\label{rmk:actionformulas}
The action of $U_L\otimes U_R$ on the generators of $\sP$ and $\sD$ can be computed explicitly.
For the subalgebra $U_R\cong 1\otimes U_R$ of $U_{LR}$, the action is given by
\[
\begin{array}{ccccc}
E_k\cdot \del_{i,j}=\llbracket k+1,j\rrbracket \del_{i,k}&,&
F_k\cdot \del_{i,j}=\llbracket k,j\rrbracket 
\del_{i,k+1}&,& 
K_{\eps_k}\cdot \del_{i,j}=q^{\llbracket k,j\rrbracket }\del_{i,j},\\
E_k\cdot t_{i,j}=-\llbracket k,j\rrbracket q^{-1} t_{i,k+1}
&,&
F_k \cdot t_{i,j}=
-\llbracket k+1,j\rrbracket qt_{i,k}&,&
K_{\eps_k}\cdot  t_{i,j}=q^{-\llbracket k,j\rrbracket} t_{i,j},
\end{array}
\]
where $1\leq k\leq n-1$, $1\leq i\leq m$ and $1\leq j\leq n$. 
For $U_L\cong U_L\otimes 1$ the formulas are similar but the action occurs in the first index (thus, they are obtained by replacing $\del_{i,j}$ by $\del_{j,i}$ and
$t_{i,j}$ by $t_{j,i}$).

\end{rmk}

\subsection{The algebras ${\sPD^\mathrm{gr}}$
 and $\sPD$} 
\label{subsec:sPDgr}

For $n\geq 1$ let $\EuScript R^{(n)}$ and $\underline{\EuScript R}^{(n)}$ be as in Definition~\ref{dfn:univRmatgln}.

\begin{dfn}
\label{dfn:RLRRuRLuRR}
We set $
\cC_L:=\cC^{(m)}$ and 
$\cC_R:=\cC^{(n)}$. Furthermore, we set 
\[
{\EuScript R}_L:={\EuScript R}^{(m)}\quad,\quad
{\underline{\EuScript R}}_L:={\underline{\EuScript R}}^{(m)}\quad,\quad {\EuScript R}_R:={\EuScript R}^{(n)}
\quad,\quad
{\underline{\EuScript R}}_R:={\underline{\EuScript R}}^{(n)}.
\] 
We define 
braidings $\check{\EuScript R}_L$ and 
$\check{\underline{\EuScript R}}_L$
on $\cC_L$, and  
$\check{\EuScript R}_R$ and 
$\check{\underline{\EuScript R}}_L$ on
$\cC_R$ as in~\eqref{Bradings-eq}.  
\end{dfn}

 From Proposition~\ref{prp:Uqglnlocfin} it follows that
 $(U_L,\cC_L,\check{\underline{\EuScript R}}_L)$ and $(U_R,\cC_R,\check{\underline{\EuScript R}}_R)$ are locally finite braided triples. Thus the product 
(in the sense of 
Definition~\ref{dfn:product-loc-fin})
 of 
 $(U_L,\cC_L,\check{\underline{\EuScript R}}_L)$ and $(U_R,\cC_R,\check{\underline{\EuScript R}}_R)$ 
 is also a locally finite braided triple of the form $(U_{LR},\cC_{LR},\check{\underline{\EuScript R}}_{LR})$, where 
$\check{\underline{\EuScript R}}_{LR}$ is defined as in~\eqref{eq:dfnofRcheck}.  
Furthermore, Proposition~\ref{prp:glmglndecom} implies that
$\sP,\sD\in\mathrm{Obj}(\cC_{LR})$. 


Recall that $\sD^{(1)}$ and $\sP^{(1)}$ are $U_{LR}$-modules. Let 
$\psi_{\circ}:\sD^{(1)}\times \sP^{(1)}\to\Bbbk$ be the $U_{LR}$-invariant $\Bbbk$-bilinear form, in the sense of \eqref{eq:Psix1x2}, that is defined by 
\[
\psi_{\circ}(\del_{i,j},t_{k,l}):=\llbracket i,k\rrbracket \llbracket j,l\rrbracket\quad
\text{for }
1\leq i,k\leq m,\text{ and }
1\leq j,l\leq n.
\]
\begin{dfn}
\label{dfn-PDPD1}
We define the algebras 
$\sPD^\mathrm{gr}$ and $\sPD$ by 
\begin{equation}
\sPD^\mathrm{gr}:=\sP\otimes_{\check{\EuScript R}}\sD\quad
\text{and}
\quad
\sPD:=\sP\otimes_{\check{\EuScript R},\psi_{\circ}}\sD,
\end{equation}
according to Definition~\ref{dfn:ARBB} and 
Definition~\ref{dfn:ARpsiBBB}, with 
$A:=\sP$, $E_A:=\sP^{(1)}$, $B:=\sD$, $E_B:=\sD^{(1)}$, $\psi:=\psi_\circ$, and  $\check{\EuScript R}:=\check{\underline{\EuScript R}}_{LR}$. 
\end{dfn}

It turns out that there is an equivalent description of  $\sPD$ and $\sPD^\mathrm{gr}$ by  generators and relations. Recall the notation $\llbracket a,b\rrbracket$ that was defined in~\eqref{eq:[[a,b]]}. We set
\[
\llbracket a,b\rrbracket_q:=\begin{cases}
q& \text{ if }a=b,\\
q-q^{-1}&\text{if }a\neq b.
\end{cases}
\]
\begin{dfn}
\label{dfn-PDPD2}
The algebra $\sPD$ is  generated by $2mn$ generators $t_{i,j}$ and $\del_{i,j}$ for $1\leq i\leq m$ and $1\leq j\leq n$, modulo the relations {\rm (R1)}, {\rm (R2)}, {\rm (R1$'$)}, {\rm (R2$'$)} of Subsection~\ref{subsec:PnnnDnan}
 and the relations
\begin{align}\label{eq:Sah}
\del_{\oline a_1,\oline a_2}t_{a_1,a_2}
=
\llbracket 
a_1,\oline a_1
\rrbracket
\llbracket
a_2,\oline a_2
\rrbracket
+
\sum_{b_1\geq a_1}\sum_{\oline b_1\geq \oline a_1} 
\sum_{b_2\geq a_2}\sum_{\oline b_2\geq \oline a_2} 
\left(
\diamondsuit_1+\square_1
\right)
\left(
\diamondsuit_2+\square_2
\right)
t_{b_1,b_2}\del_{\oline b_1,\oline b_2},
\end{align}
where 
\[
\diamondsuit_i:=\llbracket
a_i,\oline a_i
\rrbracket
\llbracket
b_i,\oline b_i
\rrbracket
\llbracket
a_i,b_i
\rrbracket_q
\quad\text{ and }\quad
\square_i:=
\left(1-
\llbracket
a_i,\oline a_i
\rrbracket
\right) \llbracket
a_i,b_i
\rrbracket
\llbracket
\oline a_i,\oline b_i
\rrbracket.
\]
The algebra
$\sPD^\mathrm{gr}$ is also generated by
$2mn$ generators $t_{i,j}$ and $\del_{i,j}$ for $1\leq i\leq m$ and $1\leq j\leq n$ modulo
 the same  relations, except that 
$\llbracket 
a_1,\oline a_1
\rrbracket
\llbracket
a_2,\oline a_2
\rrbracket
$ does not occur on  
 the 
right hand side of~\eqref{eq:Sah}.

\end{dfn}

\begin{rmk}
\label{rmk:R3R6}
The relation~\eqref{eq:Sah} of $\sPD$ can be written more explicitly as the  relations  
{\rm (R3)--(R6)} below:
\begin{itemize}
\item[\rm (R3)] $\partial_{c,b}t_{d,a} =t_{d,a} \partial_{c,b}$ if $b\neq a $ and $c\neq d$.\\
  
\item[\rm (R4)] $\displaystyle\partial_{c,b}t_{c,a} =qt_{c,a} \partial_{c,b}+ \sum_{c'>c} (q-q^{-1})t_{c',a} \partial_{c',b}$\ \  if $b\neq a $.
\item[\rm (R5)] $\displaystyle\partial_{c,a}t_{d,a} = qt_{d,a} \partial_{c,a}+\sum_{a'>a}(q-q^{-1})t_{d,a'} \partial_{c,a'}$\ \  if $c\neq d$.
\item[\rm (R6)] $\displaystyle
\partial_{c,d}t_{c,d} = 1+\sum_{c'\geq c}\,\sum_{d'\geq d}q^{\llbracket c',c\rrbracket +\llbracket d',d\rrbracket }(q-q^{-1})^{2-\llbracket c',c\rrbracket -\llbracket d',d\rrbracket }
t_{c',a'} \partial_{c',a'}$.
\end{itemize}
For $\sPD^\mathrm{gr}$ the
relation~\eqref{eq:Sah} has the same explicit form,  except that (R6) should be replaced by
\begin{itemize}
\item[\rm (R6$'$)] $\displaystyle
\partial_{c,d}t_{c,d} = \sum_{c'\geq c}\,\sum_{d'\geq d}q^{\llbracket c',c\rrbracket +\llbracket d',d\rrbracket }(q-q^{-1})^{2-\llbracket c',c\rrbracket -\llbracket d',d\rrbracket }
t_{c',d'} \partial_{c',d'}$.
\end{itemize}

\end{rmk}

\begin{prp}
Definition~\ref{dfn-PDPD1} and Definition~\ref{dfn-PDPD2} are equivalent. 
\end{prp}
\begin{proof}
We just need to explain how to compute the mixed relations~\eqref{eq:ba-gammaAB} and~\eqref{eq:ba-gammaABpsi}.
 As a $U_{LR}$-module, 
\[
\sD^{(1)}\cong \breve V^{(m)}\otimes \breve V^{(n)}\quad\text{and}\quad
\sP^{(1)}\cong V^{(m)}\otimes V^{(n)},
\]
where the isomorphisms are $\partial_{i,j}\mapsto e_i\otimes e_j$ and $t_{i,j}\mapsto e_i\otimes e_j$. 
By a direct calculation using Definition~\ref{dfn:univRmatgln}
we obtain
\[
\left(\underline{\EuScript R}_{L}\right)_{\breve V^{(m)},V^{(m)}}=q\sum_{1\leq i\leq m}
\mathsf E_{i,i}\otimes \mathsf E_{i,i}+\sum_{1\leq i\neq j\leq m}
\mathsf E_{i,i}\otimes \mathsf E_{j,j}+(q-q^{-1})\sum_{1\leq i<j\leq m}
\mathsf E_{j,i}\otimes \mathsf E_{j,i}.
\]
The formula for $\left(\underline{\EuScript R}_{R}\right)_{\breve V^{(n)},V^{(n)}}$ is similar, with $m$ replaced by $n$. 
The mixed relations~\eqref{eq:ba-gammaABpsi} of $\sPD$ and~\eqref{eq:ba-gammaAB}
of
 $\sPD^{\mathrm{gr}}$ can now be computed
explicitly based on Definition~\ref{dfn:ARBB} and Definition~\ref{dfn:ARpsiBBB}.
\end{proof}

\subsection{Bases of monomials for $\sPD$ and $\sPD^\mathrm{gr}$}
Consider the monomials
\begin{equation}
\label{eq:bassis}
t_{1,1}^{a_{1,1}}\cdots t_{1,n}^{a_{1,n}}\cdots
t_{m,1}^{a_{m,1}}\cdots t_{m,n}^{a_{m,n}}
\del_{m,n}^{b_{m,n}}\cdots \del_{m,1}^{b_{m,1}}\cdots
\del_{1,n}^{b_{1,n}}\cdots \del_{1,1}^{b_{1,1}}
,\quad a_{i,j},b_{i,j}\in\Z^{\geq 0}.
\end{equation}
The expression~\eqref{eq:bassis} makes sense both as an element of $\sPD$ and an element of $\sPD^\mathrm{gr}$. 
\begin{prp}
\label{diamond-app}
The monomials~\eqref{eq:bassis} form a $\Bbbk$-basis of $\sPD$.
\end{prp}
\begin{proof}
By a standard straightening argument we can show that by using the relations (R1), (R1$'$), 
(R2), (R2$'$) and (R3)--(R6) any product of 
the $t_{i,j}$ and the $\del_{i,j}$ can be expressed as a linear combination of 
the monomials~\eqref{eq:bassis}. 
The fact that the latter monomials are indeed linearly independent follows from Bergman's Diamond Lemma and some straightforward (although tedious) computations. 
This was also pointed out  in~\cite[Sec. 10]{SSV04}.
In~\cite{LSS22a} we give a more conceptual proof of this assertion using the theory of PBW deformations of quadratic algebras.
 \end{proof}

\begin{prp}
\label{prp:P-D-bases}
The algebra $\sP$ has a basis consisting of monomials~\eqref{eq:bassis}
where $b_{i,j}=0$ for all $i,j$. The algebra $\sD$ has a basis consisting of monomials~\eqref{eq:bassis}
where $a_{i,j}=0$ for all $i,j$.  

\end{prp}

\begin{proof}
This follows from Proposition~\ref{diamond-app}. 
It is also proved for example 
in~\cite[Thm 1.4]{NYM93}.
\end{proof}
By Proposition~\ref{prp:ARBAB}
the algebra $\sPD^\mathrm{gr}$ is a quotient of the free algebra on $2mn$ generators $t_{i,j}$ and $\del_{i,j}$. Note that by a slight abuse of notation we use the same notation for generators of $\sPD$ and $\sPD^\mathrm{gr}$. In the next proposition we describe a basis for this quotient. 
\begin{prp}
\label{diamond-app-gr}
The monomials~\eqref{eq:bassis}
form a basis of $\sPD^\mathrm{gr}$. 
\end{prp}
\begin{proof}
This follows from the vector space decomposition $\sPD^\mathrm{gr}=\sP\otimes \sD$ and
Proposition~\ref{prp:P-D-bases}.  
\end{proof}

\begin{rmk}
\label{rmk:presentationofPDgr}
From the results of this subsection it follows that $\sPD^\mathrm{gr}$ 
has two realizations:
\begin{itemize}
\item[(i)] According to Definition~\ref{dfn-PDPD1} we have $\sPD^\mathrm{gr}\cong \sP\otimes \sD$ as a $\Bbbk$-vector space. Thus, $\sPD^\mathrm{gr}$  is generated by $2mn$ generators $t_{i,j}\otimes 1$ and $1\otimes \del_{i,j}$. 

\item[(ii)] By Proposition~\ref{prp:ARBAB}
we can realize 
$\sPD^\mathrm{gr}$
as a quotient of the free $\Bbbk$-algebra generated by $2mn$ generators: the  $t_{i,j}$ and the $\del_{i,j}$.
\end{itemize}
\end{rmk}

\begin{dfn}
\label{dfn:precc}
Define a  total order $\prec$ on the set of pairs $(i,j)$ with $1\leq i\leq m$ and $1\leq j\leq n$ as follows: we set 
$(i,j)\prec (i',j')$ if either 
$i+j<i'+j'$, or $i+j=i'+j'$ and $i<i'$.
\end{dfn}

\begin{rmk}
\label{rmk:precorderbasis}
The algebra $\sPD$ has another basis consisting of monomials of the form
\begin{equation}\label{eq:monomialprec}
\left(\prod_{i,j}t_{i,j}^{a_{i,j}}\right)\left(\prod_{i,j} \del_{i,j}^{b_{i,j}}\right),
\end{equation}
where the $\del_{i,j}$ (respectively, the $t_{i,j}$) occur in ascending (respectively, descending) order relative to the total order $\prec$. 
This can be deduced from Proposition~\ref{diamond-app}. Indeed by an elementary argument one can show that any monomial of the $t_{i,j}$ of total degree $d$ 
that is sorted in the order 
given in Proposition~\ref{diamond-app} can be expressed as a linear combination of monomials of the $t_{i,j}$ of total degree $d$ that are sorted in the order given in~\eqref{eq:monomialprec}. 
A similar assertion holds for monomials in the $\del_{i,j}$. 
Thus the monomials of the form~\eqref{eq:monomialprec} span $\sPD$. A dimension counting argument implies that the latter monomials also form a basis. 

By an analogous reasoning  we can also show that $\sPD$ has a basis that consists of the monomials 
\begin{equation}
\label{eq:basisII}
t_{m,n}^{a_{m,n}}\cdots t_{m,1}^{a_{m,1}}\cdots
t_{1,n}^{a_{1,1}}\cdots t_{1,1}^{a_{1,1}}
\del_{1,1}^{b_{1,1}}\cdots \del_{1,n}^{b_{1,n}}\cdots
\del_{m,1}^{b_{m,1}}\cdots \del_{m,n}^{b_{m,n}}
,\quad a_{i,j},b_{i,j}\in\Z^{\geq 0}.
\end{equation}
Here the $\del_{i,j}$ (respectively, the $t_{i,j}$) are sorted according to the lexicographic order (respectively, the reverse lexicographic order) on indices. 
\end{rmk}

\subsection{The algebras $\mathscr A_{k,l,n}$ and $\mathscr A_{k,l,n}^\mathrm{gr}$}

\label{subsec:Now3.9}
In this subsection we consider two families of algebras,
the $\mathscr A_{k,l,n}$ and the $\mathscr A_{k,l,n}^\mathrm{gr}$, 
 that
slightly  generalize $\sPD$ and $\sPD^\mathrm{gr}$.
\begin{dfn}
\label{dfn:Akln}
Fix integers $k,l,n\geq 1$ and set $m:=\max\{k,l\}$. 
Let
$\tilde{t}_{i,j}$ and 
$\tilde{\del}_{i,j}$ be as in~\eqref{eq:tildetd} where $a=m$ and $b=n$. 
We define $\mathscr A_{k,l,n}$
(respectively, $\mathscr A_{k,l,n}^\mathrm{gr}$)
 to be the subalgebra of $\sPD=\sPD_{m\times n}$ 
(respectively, 
$\sPD^\mathrm{gr}\cong\sP\otimes \sD$) 
 that is generated by the $\tilde t_{i,j}$ and the $\tilde\del_{i',j}$ 
(respectively, 
the $\tilde t_{i,j}\otimes 1$ and the $1\otimes \tilde\del_{i',j}$) 
 where
\begin{equation}
\label{eq:ii'jrange}
1\leq i\leq k\quad,\quad
1\leq i'\leq l\quad\text{and}\quad 1\leq j\leq n.
\end{equation} 
\end{dfn}
\begin{prp}
\label{prp:explicit-prs-Akln}
The algebras 
$\mathscr A_{k,l,n}$ and $\mathscr A_{k,l,n}^\mathrm{gr}$ have the following presentations:
\begin{itemize}
\item[\rm(i)]
$\mathscr A_{k,l,n}$ is isomorphic to the quotient of the free $\Bbbk$-algebra generated by the symbols $t_{i,j}$ and $\del_{i',j}$ with $i$, $i'$, $j$ satisfying~\eqref{eq:ii'jrange}, 
 modulo the relations 
{\rm (R1)}, {\rm (R2)}, {\rm (R1$'$)}, {\rm (R2$'$)} of Subsection~\ref{subsec:PnnnDnan}
 and the relations 
{\rm (R3)--(R6)} of Remark~\ref{rmk:R3R6}.

\item[\rm(ii)]
$\mathscr A^\mathrm{gr}_{k,l,n}$ is isomorphic to the quotient of the free $\Bbbk$-algebra generated by the symbols $t_{i,j}$ and $\del_{i',j}$ with $i$, $i'$, $j$ satisfying~\eqref{eq:ii'jrange}, 
 modulo the relations 
{\rm (R1)}, {\rm (R2)}, {\rm (R1$'$)}, {\rm (R2$'$)} of Subsection~\ref{subsec:PnnnDnan}
 and the relations 
{\rm (R3)--(R5)} and {\rm (R6$'$)} of Remark~\ref{rmk:R3R6}.

 \end{itemize}
\end{prp}

\begin{proof}
(i)  Denote the  quotient of the free algebra
by $\mathscr F_{k,l,n}$. Since $\mathscr A_{k,l,n}$ is a subalgebra of $\sPD=\sPD_{m\times n}$ for $m:=\max\{k,l\}$, from the explicit description of the relations of $\sPD$
it follows that there exists a natural epimorphism $\mathsf f_{k,l,n}:{\mathscr F}_{k,l,n}\to \mathscr A_{k,l,n}$ that is uniquely defined by the assignments $t_{i,j}\mapsto t_{m-k+i,j}$ and $\del_{i',j}\mapsto \del_{m-l+i',j}$. 
A standard straightening argument proves that every element of $\mathscr F_{k,l,n}$ is a linear combination of monomials of the form~\eqref{eq:bassis}.  Proposition~\ref{diamond-app}
implies that 
 $\mathsf f_{k,l,n}$ maps the latter monomials to 
a linearly independent set of elements of $\mathscr A_{k,l,n}$. Thus $\mathsf f_{k,l,n}$ is an isomorphism. 

(ii) Similar to the proof of (i), with Proposition~\ref{diamond-app} replaced by 
Proposition~\ref{diamond-app-gr}.
\end{proof}

\begin{dfn}
For any $1\leq r\leq n$ we can identify $U_q(\gl_r)$ with a Hopf subalgebra of $U_q(\gl_n)$ via the monomorphism of associative algebras
\begin{equation}
\label{eq:mapKappa}
\boldsymbol\kappa_{r,n}:U_q(\gl_r)\to U_q(\gl_n),
\end{equation}
defined by $\boldsymbol\kappa_{r,n}(E_i):=E_{i+n-r}$, $\boldsymbol\kappa_{r,n}(F_i):=F_{i+n-r}$, and $\boldsymbol\kappa_{r,n}(K_{\eps_i}^{\pm 1}):=K_{\eps_{i+n-r}}^{\pm 1}$.
\end{dfn}
In the next proposition we establish the existence of the map~\eqref{eq:embeD}.  
Recall 
from 
Remark~\ref{rmk:presentationofPDgr}(ii)
that we consider both $\sPD$ and $\sPD^\mathrm{gr}$ as algebras generated by $2mn$ generators $t_{i,j}$ and $\del_{i,j}$.

\begin{prp}
\label{prp:existence(6)}
Fix $1\leq m'\leq m$ and $1\leq n'\leq n$. Let
$\tilde t_{i,j},\tilde \del_{i,j}\in\sPD_{m'\times n'}$ 
(respectively, 
$\tilde t_{i,j},\tilde \del_{i,j}\in\sPD^\mathrm{gr}_{m'\times n'}$)
for $1\leq i\leq m'$ and $1\leq j\leq n'$ be as in~\eqref{eq:tildetd} for $a:=m'$ and $b:=n'$,
that is 
\[
\tilde t_{i,j}:=t_{m'+1-i,n'+1-j}\quad
\text{and}\quad
\tilde \del_{i,j}:=\del_{m'+1-i,n'+1-j}.
\] 
Also, let  $\tilde t_{i,j},\tilde\del_{i,j}\in\sPD_{m\times n}$
(respectively, 
$\tilde t_{i,j},\tilde\del_{i,j}\in\sPD^\mathrm{gr}_{m\times n}$)
for $1\leq i\leq m$ and $1\leq j\leq n$
be as in~\eqref{eq:tildetd} for $a:=m$ and $b:=n$, that is  
\[
\tilde t_{i,j}:=t_{m+1-i,n+1-j}\quad
\text{and}\quad
\tilde \del_{i,j}:=\del_{m+1-i,n+1-j}.
\] 
Then the following assertions hold. 
\begin{itemize}
\item[\rm (i)]
The assignments
$\tilde t_{i,j}\mapsto\tilde t_{i,j}$ and 
$\tilde \del_{i,j}\mapsto \tilde \del_{i,j}$ for $1\leq i\leq m'$ and $1\leq j\leq n'$ define
unique embeddings 
of algebras \[
\boldsymbol e:=\boldsymbol e_{m'\times n'}^{m\times n}:\sPD_{m'\times n'}\to\sPD_{m\times n}\quad
\text{and}\quad
\Emb^\mathrm{gr}:=(\Emb^\mathrm{gr})_{m'\times n'}^{m\times n}:\sPD_{m'\times n'}^\mathrm{gr}\to\sPD_{m\times n}^\mathrm{gr}.
\]
\item[\rm (ii)] If we identify $U_q(\gl_{m'})\otimes U_q(\gl_{n'})$ with a subalgebra of $U_q(\gl_m)\otimes U_q(\gl_n)$ via  $\boldsymbol \kappa_{m',m}\otimes \boldsymbol\kappa_{n',n}$ then 
the maps $\Emb$ 
and $\Emb^\mathrm{gr}$ are 
 $U_q(\gl_{m'})\otimes U_q(\gl_{n'})$-equivariant. 


\end{itemize}
\end{prp}
\begin{proof}
We only give the details of the proofs of these assertions for $\sPD$. The arguments for $\sPD^\mathrm{gr}$ are analogous. 

(i) From Definition~\ref{dfn-PDPD2} and Remark~\ref{rmk:R3R6} it follows that 
the generators $\tilde t_{i,j}$ and $\tilde \del_{i,j}$ of $\sPD_{m'\times n'}$ and $\sPD_{m\times n}$ satisfy identical relations. It follows that there exists a homomorphism of algebras $\sPD_{m'\times n'}\to \sPD_{m\times n}$.
By Proposition~\ref{diamond-app} the latter map takes a basis of $\sPD_{m'\times n'}$ to a basis of $\sPD_{m\times n}$, hence it is an injection. The uniqueness assertion is trivial.

(ii) We give the proof for $\Emb$ only, since the proof for $\Emb^\mathrm{gr}$ is similar. 
Since $\sPD_{m'\times n'}$ and $\sPD_{m\times n}$ are module algebras it suffices to verify equivariance for standard generators of $U_q(\gl_{m'})$ and $U_q(\gl_{n'})$ on the $\tilde t_{i,j}$ and the $\tilde{\del}_{i,j}$. This can be done using the explicit formulas of Remark~\ref{rmk:actionformulas}.
%
\end{proof}

%
%

\subsection{The action of $\sPD$ on $\sP$ and the map $\phi_{PD}$}
\label{subsec:sPDgr6}
Recall that by Proposition~\ref{prp:Hmd-desc}, $\sPD$ is a $U_{LR}$-module algebra.
We denote the action of $x\in U_{LR}$ on $D\in \sPD$ by $x\cdot D$.

Let $\mathscr I$ denote the left ideal of $\sPD$ that is generated by $\sD^{(1)}$.
By Proposition~\ref{diamond-app} we have a  $U_{LR}$-invariant decomposition
\[
\sPD\cong \mathscr I\oplus \sP
.
\] This decomposition equips $\sP\cong \sPD/\mathscr I$ with
a $\sPD$-module structure given by 
\begin{equation}
\label{eq:sPDxP-action}
\sPD\otimes \sP\to \sP\ ,\ 
D\otimes (f+\mathscr I)\mapsto 
D\cdot f\quad
\text{for }D\in\sPD\text{ and }f\in\sP,
\end{equation}
where $D\cdot f:=
(Df)+\mathscr I$.
\begin{dfn}
\label{dfn:themapPhiPD}
The map $\phi_{PD}:\sPD\to\End_\Bbbk(\sP)
$ is the homomorphism of algebras induced by~\eqref{eq:sPDxP-action}.
\end{dfn}
To simplify our notation, henceforth for $X\in \sPD$ and $f\in \sP$ we write $X\cdot f$ instead of $\phi_{PD}(X)f$.
\begin{lem}
\label{lem:thmapequvv}
The map~\eqref{eq:sPDxP-action} is a $U_{LR}$-module homomorphism. 
\end{lem} 
\begin{proof}
This is a consequence of the following general fact: let $H$ be a Hopf algebra, $A$ be an $H$-module algebra, and $I\sseq A$ be an $H$-stable left ideal of $A$. Then the canonical $A$-module structure map $A\otimes A/I\to A/I$ is an $H$-module homomorphism.  
\end{proof}

\subsection{$\sP$ is a faithful $\sPD$-module}
\label{subsec:now3.11}
The goal of this subsection is to provide a purely algebraic proof of  Proposition~\ref{prp:faithful-action}.
This proposition is also proved in~\cite[Thm 2.6]{SSV04} using analytic tools.

\begin{lem}
\label{lem:delUsseq}
Let $\sP^{(\leq k)}:=\bigoplus_{i=0}^k\sP^{(i)}$
for $k\geq 0$.
Then $\del_{i,j}\cdot \sP^{(\leq k)}\sseq
\sP^{(\leq k-1)}$.
\end{lem}
\begin{proof}
Follows by 
induction on $k$ and the mixed relations (R3)--(R6) in Subsection~\ref{subsec:sPDgr}.
\end{proof}

Recall the total order $\prec$ on the set of pairs 
$(i,j)$ with $1\leq i\leq m$ and $1\leq j\leq n$ from Definition~\ref{dfn:precc}.

\begin{lem}
\label{lem:i+j>i1+j1}
Assume that $(i_r,j_r)\prec (i,j)$ for $1\leq r\leq k$. Then 
$\del_{i,j}\cdot (t_{i_1,j_1}\cdots t_{i_k,j_k})=0$.
\end{lem}
\begin{proof}
We use induction on $k$. From $(i_1,j_1)\prec (i,j)$ it follows that either  $i>i_1$ or $j>j_1$. If $i>i_1$ then 
by the mixed relations (R3) or (R5)  we have 
\[
\del_{i,j}t_{i_1,j_1}\cdots t_{i_k,j_k}=c_1t_{i_1,j_1}\del_{i,j}t_{i_2,j_2}\cdots t_{i_k,j_k}+
\delta_{j,j_1}c_2
\sum_{j'>j}t_{i_1,j'}\del_{i,j'}
t_{i_2,j_2}\cdots t_{i_k,j_k},
\]
for some $c_1,c_2\in\Bbbk$. 
The claim now follows from the induction hypothesis, because $i+j'>i+j$ and therefore $(i,j)\prec (i,j')$. When $j>j_1$ the argument is similar.    
\end{proof}
For $a\in\Z$ we set
\begin{equation}
\label{eq:c9a}
\mathbf c(a):=\begin{cases}
\sum_{i=0}^a q^{2i} & \text{ if }a\geq 0,\\
0 &\text{otherwise.}  
\end{cases}
\end{equation}

\begin{lem}
\label{lem:i+jttoa}
Assume that $(i_r,j_r)\prec (i,j)$ for $1\leq r\leq k$.  Then 
\[
\del_{i,j}\cdot(t^a_{i,j}t_{i_1,j_1}\cdots t_{i_k,j_k})=\mathbf c(a-1)t_{i,j}^{a-1}t_{i_1,j_1}\cdots t_{i_k,j_k}
\quad \text{for $a\geq 1$}.
\] 
\end{lem}
\begin{proof}
The mixed relation (R6)  implies \begin{align*}
\del_{i,j}t_{i,j}^at_{i_1,j_1}&\cdots t_{i_k,j_k}
=t_{i,j}^{a-1}t_{i_1,j_1}\cdots t_{i_k,j_k}+q^2t_{i,j}\del_{i,j}
t_{i,j}^{a-1}t_{i_1,j_1}\cdots t_{i_k,j_k}
\\
&+
(q^2-1)
\sum_{i'>i}
t_{i',j}\del_{i',j}
t_{i,j}^{a-1}t_{i_1,j_1}\cdots t_{i_k,j_k}
+
(q^2-1)
\sum_{j'>j}
t_{i,j'}\del_{i,j'}
t_{i,j}^{a-1}t_{i_1,j_1}\cdots t_{i_k,j_k}\\
&+(q-q^{-1})^2
\sum_{i'>i,j'>j}
t_{i',j'}\del_{i',j'}
t_{i,j}^{a-1}t_{i_1,j_1}\cdots t_{i_k,j_k}
.
\end{align*}
Since $\min\{i'+j,i+j',i'+j'\}>i+j$,
by Lemma~\ref{lem:i+j>i1+j1}
the sums on the second and the third line lie in the ideal $\mathscr I$.  The assertion follows by induction on $a$.
\end{proof}

\begin{prp}
\label{prp:faithful-action}
$\sP$ is a faithful $\sPD$-module. 
\end{prp}

\begin{proof}
Let $D\in\sPD$ and assume that $D\neq 0$. Then $D=\sum_{d\geq 0}D_d$ where each $D_d$ is a linear combination of monomials of the form~\eqref{eq:monomialprec} with $\sum_{i,j}b_{i,j}=d$. 
Set $d_\circ:=\min\{d:D_d\neq 0\}$.  By Lemma \ref{lem:delUsseq}, for $f\in \sP^{(d_\circ)}$ we have $D\cdot f=D_{d_\circ}\cdot f$. 
Let $\mathcal T$ denote the set of all the  $mn$-tuples $\mathsf b:=(b_{i,j})$ for which a monomial of the form~\eqref{eq:monomialprec} occurs in  $D_{d_\circ}$ with a nonzero coefficient. We sort the components of the $\mathsf b:=(b_{i,j})$ according to $\prec$ on the pairs $(i,j)$. In other words, we assume that
$
\mathsf b:=(b_{1,1},b_{1,2},b_{2,1},\ldots, b_{m-1,n},b_{m,n-1},b_{m,n})
$.
Let $\tilde{\mathsf b}:=(\tilde{b}_{i,j})$ be the  minimum of $\mathcal T$ in the reverse lexicographic order. 
Thus, we have \[
\tilde b_{m,n}=\min\{b_{m,n}\,:\,(b_{i,j})\in\mathcal T\},
\] 
then also 
$\tilde b_{m-1,n}=\min\{b_{m-1,n}\,:\,(b_{i,j})\in\mathcal T\text{ and }b_{m,n}=\tilde b_{m,n}\}$, and so on. 
 From 
Lemma~\ref{lem:i+j>i1+j1}
and Lemma~\ref{lem:i+jttoa} it follows that 
$D_{d_\circ}\cdot \prod_{i,j}t^{\tilde{b}_{i,j}}\neq 0$.
\end{proof}
\subsection{Two $U_{LR}$-actions on $\sPD$ are identical}
By Proposition~\ref{prp:faithful-action} the map $\phi_{PD}$ is an injection and consequently we can consider $\sPD$ as a subalgebra of $\End_\Bbbk(\sP)$. Thus according to Lemma~\ref{rmk:Hopf-eps} there exists another action of $U_{LR}$ on elements of $\sPD$. We temporarily denote this action by $x\bullet D$ for $x\in U_{LR}$ and $D\in \sPD$.
In the following proposition, we show that the latter action is identical to the action  that is defined in the beginning of Subsection~\ref{subsec:sPDgr6}. 

\begin{prp}
\label{prp:xBDvsxD}
$x\bullet D=x\cdot D$
for $x\in U_{LR}$ and $D\in\sPD$. 
\end{prp}
\begin{proof}
By Lemma~\ref{lem:thmapequvv}, for $f\in \sP$ we have
\begin{align*}
(x\bullet D)\cdot f=
\sum x_1\cdot (D\cdot& (S(x_2)\cdot f))=
\sum (x_1\cdot D)\cdot (x_2\cdot (S(x_3)\cdot f))\\
&=\sum(x_1\cdot D)\cdot (\epsilon(x_2)f)=
\sum (x_1\epsilon(x_2)\cdot D)\cdot f=(x\cdot D)\cdot f.
\end{align*}
Since $\sP$ is a faithful module over $\End_\Bbbk(\sP)$, it follows that $x\bullet D=x\cdot D$.
\end{proof}
Henceforth we only use the notation $x\cdot D$ to denote the $U_{LR}$-action on $\sPD$.

\subsection{The maps $\mathsf P_{k,l,n}$}
\label{subsec:Pkln}
Recall that $\sPD^{\mathrm{gr}}\cong\sP\otimes\sD$ as a vector space. 
Let
\begin{equation}
\label{eq:sfPmn}
\mathsf P:\sPD^{\mathrm{gr}}\to \sPD
\end{equation}
be the linear map uniquely defined by
$\mathsf P(a\otimes b):=ab$ for $a\in\sP$ and $b\in \sD$. 
\begin{prp}
\label{prp:PisequivULR}
The map $\mathsf P$ is an isomorphism of $U_{LR}$-modules. 
\end{prp}
\begin{proof}
The map $\mathsf P$ is identical to the map~\eqref{eq:AJlfjf} 
when
$A:=\sP$, $B:=\sD$ and $H:=U_{LR}$. Hence by Remark~\ref{rmk:AoB->ARB} it  is  
a homomorphism of $U_{LR}$-modules.
From Proposition~\ref{diamond-app} 
and Proposition~\ref{diamond-app-gr}
it follows that $\mathsf P$ maps  a basis of $\sPD^{\mathrm{gr}}$ to a basis of $\sPD$, hence it is indeed an isomorphism of $U_{LR}$-modules. 
\end{proof}
\begin{dfn}
Given $k,l,n\geq 1$, we set $m:=\max\{k,l\}$ and define the map 
\begin{equation}
\label{eq:Pkln}
\mathsf P_{k,l,n}:\mathscr A_{k,l,n}^\mathrm{gr}\to\mathscr A_{k,l,n}\ ,\ 
D\mapsto \mathsf P(D),
\end{equation}
where $\mathsf P:\sPD^\mathrm{gr}\to\sPD$ is as in~\eqref{eq:sfPmn}.  
\end{dfn}


For $r,s\in\Z^{\geq 0}$ we set \begin{equation}
\label{eq:Aklnrs}
\sPD^{\mathrm{gr},(r,s)}:=
\sP^{(r)}\otimes \sD^{(s)}
\quad\text{ and }\quad
\mathscr A_{k,l,n}^{\mathrm{gr},(r,s)}:=
\mathscr A_{k,l,n}^\mathrm{gr}\cap
\sPD^{\mathrm{gr},(r,s)},
\end{equation}
so that 
\[
\sPD^\mathrm{gr}=\bigoplus_{r,s\geq 0}
\sPD^{\mathrm{gr},(r,s)}\quad\text{and} 
\quad
\mathscr A_{k,l,n}^\mathrm{gr}:=
\bigoplus_{r,s\geq 0}
\mathscr A_{k,l,n}^{\mathrm{gr},(r,s)}.
\]
 By
Proposition~\ref{prp:glmglndecom} 
 we obtain an isomorphism of 
$U_q(\gl_k)\otimes U_q(\gl_l)\otimes U_q(\gl_n)\otimes U_q(\gl_n)$-modules  
\begin{equation}
 \label{eq:PD(r,s)d}
 \mathscr A_{k,l,n}^{\mathrm{gr},(r,s)}\cong
\bigoplus_{\footnotesize\begin{array}{c}\la\in\Lambda_{\underline k,r}\\[-1mm]\mu\in\Lambda_{\underline l,s}\end{array}}
(V_\la^*\otimes  V_\mu)\otimes
(V_\la^*\otimes V_\mu),
\end{equation}
where $\underline k:=\min\{k,n\}$ and $\underline l:=\min\{l,n\}$.
Here we consider the left copy of $V_\la^*\otimes V_\mu$ as a module  for $U_q(\gl_k)\otimes U_q(\gl_l)$ and the right copy of $V_\la^*\otimes V_\mu$ as a module for $U_R\otimes U_R\cong U_q(\gl_n)\otimes U_q(\gl_n)$. Of course by restriction along the coproduct map $U_R\to U_R\otimes U_R$ we can also consider the right copy  as a $U_R$-module. 

\begin{prp}
\label{rmk:PtildeRvsPD}
For $a\otimes a'\in\sP^{(r)}\otimes_{\tilde{\EuScript R}} \sD^{(r')}$ and $b\otimes b'\in\sP^{(s)}\otimes_{\tilde{\EuScript R}} \sD^{(s')}$ we have 
\begin{equation}
\label{eq::Paa'bb'}
\mathsf P(a\otimes a')
\mathsf P(b\otimes b')
-\mathsf P
\big((a\otimes a')(b\otimes b')\big)\in
\bigoplus_{i=1}^{\min\{u,u'\}}\sPD^{(u-i,u'-i)},
\end{equation}
where $u:=r+s$ and $u':=r'+s'$.
\end{prp}
\begin{proof}
We have  $\mathsf P(a\otimes a')\mathsf P(b\otimes b')=aa'bb'$.
Using the explicit relations of $\sPD$ (see Remark~\ref{rmk:R3R6}) we can move the $\del_{i,j}$ past the $t_{i,j}$ to express
$a'b$ as $a'b=\sum a''b''$ where $a''\in \sP$ and $b''\in\sD$. Thus
\[\mathsf P(a\otimes a')\mathsf P(b\otimes b')
=\sum aa''b''b'.
\]
Similarly, using the relations of $\sPD^\mathrm{gr}$ we can move the $1\otimes \del_{i,j}$ past the $t_{i,j}\otimes 1$ and as a result we obtain $(1\otimes a')(b\otimes 1)=\sum \underline a''\otimes \underline b''$ where $\underline a''\in \sP$ and $\underline b''\in\sD$. It follows that 
\[
\mathsf P\big((a\otimes a')(b\otimes b')\big)=\sum a\underline a''\underline b''b'.
\]
The only difference between the relations of $\sPD$ and $\sPD^\mathrm{gr}$ is (R6) vs. (R6$'$). Since (R6$'$) is the homogenized form of (R6$'$),
it follows that \[
\sum a''b''-\sum \underline a''\underline b''\in\bigoplus_{i=1}^{\min\{r',s\}}\sPD^{(s-i,r'-i)}.
\] From the latter inclusion~\eqref{eq::Paa'bb'} follows immediately.   
 \end{proof}

Recall that  $\mathrm{gr}(\mathscr A_{k,l,n})$  denotes the associated graded algebra corresponding to the degree filtration on $\mathscr A_{k,l,n}$, i.e., the filtration obtained by setting $\deg(t_{i,j})=\deg(\del_{i,j})=1$. Note that we have a canonical isomorphism $\mathscr A_{k,l,n}^\mathrm{gr}\cong \mathrm{gr}(\mathscr A_{k,l,n}^\mathrm{gr})$ since $\mathscr A_{k,l,n}^\mathrm{gr}$ is graded. 

The maps $\mathsf P_{k,l,n}:\mathscr A_{k,l,n}^\mathrm{gr}\to \mathscr A_{k,l,n}$ do \emph{not} induce  isomorphisms of associative algebras. However,  the following statement holds.
\begin{cor}
\label{cor:Agrkln-vs-gr(A)}
The associated graded map $\mathrm{gr}(\mathsf P_{k,l,n})$ induces a $U_R$-equivariant isomorphism between 
$\mathscr A_{k,l,n}^\mathrm{gr}\cong \mathrm{gr}(\mathscr A_{k,l,n}^\mathrm{gr})$ and $\mathrm{gr}(\mathscr A_{k,l,n})$. 
When $k=l=m$, the latter map is $U_{LR}$-equivariant. \end{cor}
\begin{proof}
From Proposition~\ref{rmk:PtildeRvsPD} it follows that the associated graded map $\mathrm{gr}(\mathsf  P_{k,l,n})$ is an isomorphism of associative algebras from $\mathrm{gr}(\mathscr A_{k,l,n}^\mathrm{gr})\cong \mathscr A_{k,l,n}^\mathrm{gr}$ onto 
$\mathrm{gr}(\mathscr A_{k,l,n})$.
 The equivariance statements follow from  Proposition~\ref{prp:PisequivULR}. 
\end{proof}

\subsection{The algebras $\ULw$, $\URw$ and $\ULRw$}
\label{subsec:Pi}
We set
\[
\ULRw:=\phi_U^{-1}(\sPD):=\{x\in U_{LR}\,:\,\phi_U(x)\in\sPD\}.
\]
Also, we set
\[
\ULw:=\{x\in U_L\,:\,x\otimes 1\in\ULRw\}\quad
\text{and}\quad 
\URw:=\{x\in U_R\,:\,1\otimes x\in\ULRw\}.
\]
Recall that the adjoint action of $U_{LR}$ is $\ad_y(x):=\sum y_1xS(y_2)$ for $x,y\in U_{LR}$. 
%
Recall that we equip $\End_\Bbbk(\sP)$ with the 
$U_{LR}$-module structure of 
Lemma~\ref{rmk:Hopf-eps} and we  denote the latter action by 
$x\cdot T$ for 
 $x\in U_{LR}$ and $T\in \End_\Bbbk(\sP)$. 
Then it is straightforward from the definition of $\phi_U$ that
\begin{equation}
\label{eq:phiUadYY}
\phi_U({\ad_y(x)})=y\cdot \phi_U(x)
\quad\text{
for $x,y\in U_{LR}$. 
 }
 \end{equation}
\begin{lem}\label{lem:ax->adyx}
$\ad_x(\ULRw)\sseq \ULRw$ for $x\in U_{LR}$.
\end{lem}
\begin{proof}
Follows immediately from~\eqref{eq:phiUadYY} and Proposition~\ref{prp:xBDvsxD}.
\end{proof}
\begin{rmk}
When $m\leq n$,  Proposition~\ref{prp:glmglndecom} implies that the map $\cL_\sP:U_L\to \End_\Bbbk(\sP)$ is an injection. The argument is similar to the proof of \cite[Thm 7.1.5.13]{KS97}.\end{rmk} 
Let $\mathscr K_n$ denote the kernel of  
 $\cR_\sP:U_R\to \End_\Bbbk(\sP)$. 
For the next proposition recall the notation $\mathscr F(H,I)$  defined 
in~\eqref{eq:FHI}.  
 
\begin{prp}
\label{prp:ULWFULUR/}
Assume that $m\leq n$. Then $\ULw\sseq \mathscr F(U_L)$ and
$\URw\sseq
\mathscr F(U_R,\mathscr K_n)$. 
\end{prp}

\begin{proof}
Since the actions of $U_{LR}$ on $\sP$ and $\sD$ are degree preserving, it follows that 
$\sPD$ is a locally finite $U_{LR}$-module. The assertions of the proposition follow from the fact that the maps 
$\ULw\xrightarrow{\phi_U}\sPD$ and $\URw/(
\URw\cap \mathscr K_n)\xrightarrow{\phi_U}\sPD$ are injective and $U_{LR}$-equivariant (this is equivalent to
~\eqref{eq:phiUadYY}).
\end{proof}

The following proposition is a ``no-go theorem'' that provides evidence that the commutative diagram~\eqref{eq:commdia} cannot be fully quantized. 
\begin{prp}
\label{rmk-nogo}
There does not exist a $\Bbbk$-algebra $\widetilde{\sPD}$ with the following properties:
\begin{itemize}
\item[\rm (i)] $\widetilde{\sPD}$ is a locally finite $U_{LR}$-module.
\item[\rm (ii)] $\sP$ is a $\widetilde{\sPD}$-module and the action map
$\widetilde{\sPD}\otimes \sP\to \sP$ is a homomorphism of $U_{LR}$-module.
\item[\rm (iii)] There exists a homomorphism of algebras $\widetilde\phi:U_{LR}\to \widetilde{\sPD}$ such that the diagram
\begin{equation*}\vcenter{
\xymatrix{
U_{LR}\otimes  \sP
 \ar[rd]_{x\otimes f\mapsto \widetilde{\phi}(x)\otimes f\hspace{5mm}}\ar[rr]^{\hspace{5mm}x\otimes f\mapsto x\cdot f}& & \sP\\
& \widetilde{\sPD}\otimes\sP \ar[ru]_{D\otimes f\mapsto D\cdot f}& }
}
\end{equation*}
is commutative. 
\end{itemize}
\end{prp}

\begin{proof}
Let us denote the $U_{LR}$-action of (i) by $x\cdot D$ for $x\in U_{LR}$ and $D\in \widetilde{\sPD}$. 
Let $\mathscr K\sseq \widetilde{\sPD}$ denote the kernel of the map $\widetilde{\sPD}\to \End_\Bbbk(\sP)$ that is induced by the $\widetilde{\sPD}$-module structure on $\sP$.
We define a new $U_{LR}$-action on $\widetilde{\sPD}$ by setting
\[
x\bullet D:=\sum \widetilde \phi(x_1)D\widetilde\phi(S(x_2))\quad\text{for }x\in U_{LR}\text{ and }D\in\widetilde{\sPD}. 
\] 
By the  proof of Proposition~\ref{prp:xBDvsxD} we obtain $(x\bullet D)-(x\cdot D)\in \mathscr K$ for $x\in U_{LR}$ and $D\in\widetilde{\sPD}$.
In particular, for $x,y\in U_{LR}$ if we set $D:=\widetilde\phi(y)$ then we have
\begin{equation}
\label{eq:jfdls+}
\widetilde\phi(\ad_x(y))+\mathscr K=x\bullet D+\mathscr K=x\cdot D+\mathscr K.
\end{equation}
Now assume that $m\leq n$, so that the restriction of $\widetilde\phi$ to a map $U_L\otimes 1\to \End_\Bbbk(\sP)$ is faithful. 
By~\eqref{eq:jfdls+} it follows that  if $y\in U_L\otimes 1$ then the image of $\widetilde\phi(\ad_{U_L\otimes 1}(y))$ in $\widetilde{\sPD}/\mathscr K$ is finite dimensional. But $\widetilde\phi(\ad_{U_L\otimes 1}(y))\cap\mathscr K=\{0\}$, hence 
$\ad_{U_L\otimes 1}(y)$ is also finite dimensional.
Consequently, we have shown that $U_L=\mathscr F(U_L)$, which is a contradiction. 
\end{proof}

\subsection{Relation between 
Theorem~\ref{thm-Main-A}(i) and 
Theorem~\ref{thm-Main-A}(ii)} 
\label{subsec:relthms}
Our goal in this subsection is to prove
Lemma~\ref{lem:etam,n} below, which implies that
Theorem~\ref{thm-Main-A}(ii) follows by symmetry from 
Theorem~\ref{thm-Main-A}(i).
%

 From the symmetry of the defining relations of $\sPD_{m\times n}$ with respect to the two indices of the generators $t_{i,j}$ and  $\del_{i,j}$ it follows that there exists an isomorphism of algebras \[
\eta_{m,n}:\sPD_{m\times n}\to \sPD_{n\times m},
\] such that  $\eta_{m,n}(t_{i,j})=t_{j,i}$ and $\eta_{m,n}(\del_{i,j})=\del_{j,i}$ for $1\leq i\leq m$ and $1\leq j\leq n$. Note that $\eta_{m,n}$ restricts to an isomorphism $\sP_{m\times n}\cong \sP_{n\times m}$. This  naturally results in an isomorphism of algebras $\End_\Bbbk(\sP_{m\times n})\cong \End_\Bbbk(\sP_{n\times m})$. 
\begin{lem}
\label{lem:etam,n}
The following assertions hold.
\begin{itemize}
\item[\rm(i)]
For $x\otimes y\in  U_q(\gl_m)\otimes U_q(\gl_n)$ and $D\in \sPD_{m\times n}$
we have
\[\eta_{m,n}((x\otimes y)\cdot D)=(y\otimes x)\cdot \eta_{m,n}(D).
\]
\item[\rm (ii)] The induced isomorphism 
$\End_\Bbbk(\sP_{m\times n})\cong \End_\Bbbk(\sP_{n\times m})$ maps the images of $U_q(\gl_m)$ and  $U_q(\gl_n)$ in $\End_\Bbbk(\sP_{m\times n})$ onto the images of $U_q(\gl_m)$ and $U_q(\gl_n)$ in 
$\End_\Bbbk(\sP_{n\times m})$.
\end{itemize}
\end{lem}
\begin{proof}
(i) It suffices to prove the assertion when $x$ and $y$ are selected from the standard generators of $U_q(\gl_m)$ and $U_q(\gl_n)$, respectively. If $D=t_{i,j}$ or $D=\del_{i,j}$, then the assertion follows from symmetry of the effect of the generators on the indices of the $t_{i,j}$ and the $\del_{i,j}$ (see
Remark~\ref{rmk:actionformulas}). For general $D\in\sPD_{m\times n}$ the assertion follows from the fact that  $\sPD_{m\times n}$ and $\sPD_{n\times m}$ are module algebras over $U_q(\gl_m)\otimes U_q(\gl_n)$ and $U_q(\gl_n)\otimes U_q(\gl_m)$, respectively.  

(ii) Follows immediately from (i).
\end{proof}

\subsection{Relation to $\C[\mathrm{Mat}_{m,n}]_q$ and $\mathrm{Pol}(\mathrm{Mat}_{m,n})_q$}
\label{subsec:rela}
We can now relate our algebras $\sP$ and  $\sPD$ to the $\C$-algebras 
$\C[\mathrm{Mat}_{m,n}]_q$ and 
$\mathrm{Pol}(\mathrm{Mat}_{m,n})_q$ (where $0<q<1$) that are introduced in~\cite{SSV04,BKV06}. 
We remark that in~\cite{BKV06} only the special case  $m=n$ is considered, and  the latter algebras
 are denoted by 
$\C[\Mat_{n}]_{q}$ and 
$\mathrm{Pol}(\mathrm{Mat}_{n})_{q}$, 
respectively.
The algebra $\mathrm{Pol}(\mathrm{Mat}_{m,n})_{q}$ is defined in~\cite[Sec. 2]{SSV04} in terms of the generators $z^i_j$ and $(z^i_j)^*$ where $1\leq i\leq m$ and $1\leq j\leq n$. These generators satisfy
the relations (2.1)--(2.7) of~\cite{SSV04}.
For the reader's convenience we describe the relations of $\mathrm{Pol}(\mathrm{Mat}_{m,n})_{q}$.  The relations among the $z^i_j$ (which are (2.1)--(2.3) of~\cite{SSV04}) are identical to the relations among the $t_{i,j}$, the only difference being that $q$ becomes a complex-valued parameter. In a similar way, the relations among the $(z^i_j)^*$ 
(which are (2.4)--(2.6) of~\cite{SSV04})
are identical to the relations among the $\partial_{i,j}$. 
The mixed relations (which correspond to (2.7) of~\cite{SSV04}) are 
\[
(z^i_j)^*z^k_l=q^2\sum_{1\leq a,b\leq n}
\sum_{1\leq c,d\leq m}
(r^{b,a}_{j,l})
(r^{d,c}_{i,k})z^c_a(z^d_b)^*
+\llbracket j,l\rrbracket\llbracket i,k\rrbracket (1-q^2), 
\] 
where $r^{i,i}_{i,i}=1$, $r^{i,j}_{i,j}=q^{-1}$ for $i\neq j$,  $r^{j,j}_{i,i}=1-q^{-2}$ for $j>i$, and $r^{i,j}_{k,l}=0$ otherwise. 
The algebra $\C[\mathrm{Mat}_{m,n}]_q$ is the subalgebra 
of
$\mathrm{Pol}(\mathrm{Mat}_{m,n})_{q}$ that is generated by the $z^i_j$.

Throughout this subsection we set $A:=\Z[q,q^{-1}]$. The relations of $\sPD$ in Definition~\ref{dfn-PDPD2} are defined over  $A$. Thus we obtain an integral form $\sPD^{A}$ of $\sPD$ by considering the free $A$-submodule of $\sPD$ that is generated by the monomials~\eqref{eq:bassis}. 
Evaluation at $q_\circ$ for $0<q_\circ<1$ results in a ring homomorphism $A\to \C$.
Set $\sPD_{\lag q_\circ\rag}^{}:=\sPD^A\otimes_A\C$.  
\begin{cor}
\label{cor:POP}
The algebras $\sPD_{\lag q_\circ\rag}^{}$ and 
$\mathrm{Pol}(\mathrm{Mat}_{m,n})_{q_\circ}$
are isomorphic
by the assignments $t_{i,j}\mapsto \frac{z^i_j}{\sqrt{1-q_\circ^2}}$ and 
$\del_{i,j}\mapsto \frac{(z^i_j)^*}{\sqrt{1-q_\circ^2}}$.
\end{cor}
\begin{proof}
It is straightforward to check that these assignments  intertwine the relations (R1)--(R6)  with the relations (2.1)--(2.7) of~\cite{SSV04}.
\end{proof}
In the rest of this subsection we relate the actions of quantized enveloping algebras on $\sP$ and  
on  $\C[\mathrm{Mat}_{m,n}]_q$.
Set $
\sP^{A}:=\sP\cap \sPD^{A}$. Then 
$\sP^{A}$
is a free $A$-module, with an $A$-basis that consists of the monomials \[
t_{1,1}^{a_{1,1}}\cdots t_{1,n}^{a_{1,n}}\cdots t_{m,1}^{a_{m,1}}\cdots t_{m,n}^{a_{m,n}}.
\] 
Let  $U_q^{A}(\gl_n)$ denote the restricted integral $A$-form of $U_q(\gl_n)$. For integral forms of quantized enveloping algebras see for example~\cite[Sec. 9.3]{CP94}. 
The explicit description of  $U_q^A(\gl_n)$ 
 is given 
for example in~\cite{RT10}. We denote the analogous integral form of the algebra $U_{LR}\cong U_q(\gl_m)\otimes U_q(\gl_n)$  by $U_{LR}^A$. Thus \[U_{LR}^A\cong U_q^A(\gl_m)\otimes_A^{} U_q^A(\gl_n).\]  
For $0<q_\circ<1$ 
set $U_{q_\circ}(\gl_n):=U_q^{ A}(\gl_n)\otimes_{ A}^{}\C$. 
By Remark~\ref{rmk:actionformulas} 
the map $U_{LR}\otimes \sP\to\sP$ that describes the $U_{LR}$-module structure on $\sP$ restricts to a map
$U_{LR}^{A}\otimes \sP^{ A}\to 
\sP^{ A}$. After the scalar extension $(\--)\otimes_A\C$  and using the isomorphism $\sP^{A}\otimes_{A}\C\cong
\C[\mathrm{Mat}_{m,n}]_{q_\circ}$ we obtain a structure of a 
$U_{q_\circ}(\gl_m)\otimes
U_{q_\circ}(\gl_n)$-module  
on $\C[\mathrm{Mat}_{m,n}]_{q_\circ}$ that corresponds to a map
\begin{equation}
\label{eq:Uqmmmnn}
\big(U_{q_\circ}(\gl_m)\otimes
U_{q_\circ}(\gl_n) 
\big)
\otimes 
\C[\mathrm{Mat}_{m,n}]_{q_\circ}
\to
\C[\mathrm{Mat}_{m,n}]_{q_\circ},
\end{equation}
or equivalently a homomorphism of algebras
\begin{equation}
\label{eq:phiUq0}
\phi_{U,q_\circ}:U_{q_\circ}(\gl_m)\otimes U_{q_\circ}(\gl_n)\to\End_\C\left(\C[\mathrm{Mat}_{m,n}]_{q_\circ}\right).
\end{equation}
The next proposition relates the  latter module structure 
to the one given in~\cite[Sec. 9--10]{SSV04} and~\cite[Sec. 3]{BKV06}. We remark that in~\cite{SSV04,BKV06}, the  module algebra structure on
$\C[\mathrm{Mat}_{m,n}]_{q_\circ}$
is with respect to 
$U_{q_\circ}(\gl_m)^\mathrm{cop}\otimes
U_{q_\circ}(\gl_n)^\mathrm{cop}$. We denote the latter module structure by the map
\[
\phi_\mathrm{SSV}:U_{q_\circ}(\gl_m)^\mathrm{cop}\otimes
U_{q_\circ}(\gl_n)^\mathrm{cop}\to\End_\C\left(\C[\mathrm{Mat}_{m,n}]_{q_\circ}\right).\] 
Let $x\mapsto x^\natural$ for $n\geq 1$ be the $\C$-linear 
isomorphism of Hopf algebras $
U_{q_\circ}(\gl_n)\to U_{q_\circ}(\gl_n)^\mathrm{op}$
that is given by the same relations as
\eqref{eq:x->x*} but for $q:=q_\circ$.

%
%
%

\begin{prp}
\label{prp:SSVvsUS}
With $\phi_{U,\lag q_\circ\rag}$ and $\phi_\mathrm{SSV}$ as above, we have 
\[
\phi_{U,q_\circ}(x\otimes y)=
\phi_\mathrm{SSV}\left(\vartheta_L(S(x^\natural))\otimes\vartheta_R(S(y^\natural))\right)\quad\text{for }x\otimes y\in U_{q_\circ}(\gl_m)\otimes U_{q_\circ}(\gl_n),
\] 
where
$\vartheta_L$ and $\vartheta_R$ are the automorphisms of the Hopf algebra $U_{q_\circ}(\gl_n)$ that are uniquely defined by setting 
\[
\vartheta_L(E_i):=q_\circ^{-\frac12}F_{m-i}\quad,\quad
\vartheta_L(F_i):=q_\circ^{\frac12}E_{m-i}\quad,\quad
\vartheta_L(K_{\eps_i}):=K_{\eps_{m+1-i}}^{-1},
\]
and 
\[
\vartheta_R(E_i):=q_\circ^{-\frac12}F_{i}\quad,\quad
\vartheta_R(F_i):=q_\circ^{\frac12}E_{i}\quad,\quad
\vartheta_L(K_{\eps_i}):=K_{\eps_i}.
\]
\end{prp}

\begin{proof}
Follows from comparing Remark~\ref{rmk:actionformulas} with
\cite[Sec. 8]{SSV04} or~\cite[Eqs (14), (15)]{BKV06}. Note that the coproduct of $U_{q_\circ}(\gl_n)$ in~\cite{BKV06} is co-opposite to the coproduct considered in the present paper, but this is corrected by composing with $x\mapsto S(x^\natural)$. 
\end{proof}

\subsection{Some technical statements about the action of  $\sPD$  on $\sP$}
\label{subsec:sometech}
In this subsection we prove several technical statements about the interaction between the $\del_{i,j}$ on $\sP$. We will need these statements in the upcoming sections of this paper. In order to make our exposition more organized  we have collected all of them in one subsection.  The reader may find it easier to skip this subsection and return to it whenever there is a reference.

Recall that the action of $D\in\sPD$ on $f\in\sP$ is denoted by $D\cdot f$ (see Subsection~\ref{subsec:sPDgr6}).
\begin{lem}
\label{lem:delijonarbr}
Assume that  either $i\not\in\{a_1,\ldots,a_r\}$ or $j\not\in \{b_1,\ldots,b_r\}$, then
$\del_{i,j}t_{a_1,b_1}\cdots t_{a_r,b_r}$ belongs to the left ideal of $\sPD$ that is generated by the $\del_{i',j'}$ satisfying $i'\geq i$ and $j'\geq j$. In particular, 
$
\del_{i,j}\cdot (t_{a_1,b_1}\cdots t_{a_r,b_r})=0$. 
\end{lem}
\begin{proof}
We use induction on $r$. For $r=1$ the assertion follows from relations (R3)--(R6). Next suppose $r>1$. If $i\neq a_1$ and $j\neq b_1$ then $\del_{i,j}t_{a_1,b_1}=t_{a_1,b_1}\del_{i,j}$ and we can use the induction hypothesis. If $i=a_1$ then $j\not\in\{b_1,\ldots,b_r\}$ and we can write 
\[
\del_{i,j}t_{a_1,b_1}\cdots t_{a_r,b_r}
=q
t_{i,b_1}\del_{i,j}t_{a_2,b_2}\cdots t_{a_r,b_r}+(q-q^{-1})\sum_{i'>i}
t_{i',b_1}\del_{i',j}t_{a_2,b_2}\cdots t_{a_r,b_r},
\] 
and again the induction hypothesis is applicable to each summand on the right hand side. The argument for the case $j=b_1$ is similar. 
\end{proof}

\begin{lem}
\label{lem:5.1.2}
Assume that either $\{i_1,\ldots,i_s\}\nsubseteq\{a_1,\ldots,a_r\}$
and 
$\{j_1,\ldots,j_s\}\nsubseteq\{b_1,\ldots,b_r\}$. Then 
\[
\del_{i_1,j_1}\cdots \del_{i_s,j_s}\cdot\left(t_{a_1,b_1}\cdots t_{a_r,b_r}\right)
=0.\]
\end{lem}
\begin{proof}
Without loss of generality assume that $i_1\not\in\{a_1,\ldots,a_r\}$. 
Relations 
(R$1'$) and (R$2'$) imply that we  can replace 
$
\del_{i_1,a}\del_{b,c}
$ by
either 
$\del_{b,c}\del_{i_1,a}$ or 
$q^{\pm 1}\del_{b,c}\del_{i_1,a}$
or $
\del_{b,c}\del_{i_1,a}\pm (q-q^{-1})\del_{b,a}\del_{i_1,c} 
$. Using the latter replacements  we can express $\del_{i_1,j_1}
\cdots
\del_{i_s,j_s}$ as a linear combination of monomials that belong to  the left ideal $\breve{\mathscr I}:=\sum_{j=1}^n\sD\del_{i_1,j}$ of $\sD$. From 
Lemma~\ref{lem:delijonarbr} it follows that elements of $\breve{\mathscr I}$ annihilate 
$t_{a_1,b_1}\cdots t_{a_r,b_r}$.
\end{proof}

For the following corollary recall that $\oline M^\mathbf i_\mathbf j$ is the quantum minor defined 
in~\eqref{eq:lem:qdet2exp}.

\begin{cor}
\label{lem:sigmastaus-}
Let $\mathbf i:=(i_1,\ldots,i_r)$ and $\mathbf j:=(j_1,\ldots,j_r)$ be $r$-tuples of integers that satisfy \[
1\leq i_1<\cdots <i_r\leq m\quad\text{and}\quad
1\leq j_1<\cdots<j_r\leq n.
\] 
 Then $\oline M^\mathbf i_\mathbf j\cdot \left(t_{a_1,b_1}\cdots t_{a_s,b_s}\right)=0$ when $a_i\geq i_1+1$ for all $1\leq i\leq s$.
\end{cor}
\begin{proof}
This follows from Lemma~\ref{lem:5.1.2} since $\oline M^\mathbf i_\mathbf j$ is a linear combination of monomials of the form 
$
\del_{i_{\sigma(1)},j_1}
\cdots
\del_{i_{\sigma(r)},j_r}
$.  
\end{proof}

\begin{lem}
\label{lem:316}
Suppose that $f,g\in\sP$ satisfy 
$\del_{i_1,j_1}\cdots \del_{i_r,j_r}\cdot f=g$ for some
 $1\leq i_1,\ldots,i_r\leq m$ and $1\leq j_1,\ldots,j_r\leq n$.
Then for any  $1\leq i'_1\leq \ldots \leq i'_s\leq m$ and
$1\leq j'_1\leq \ldots \leq j'_s\leq n$ that satisfy 
either 
$
\min\{j_u\}_{u=1}^r>\max\{j'_u\}_{u=1}^s
$ 
or
$
\min\{i_u\}_{u=1}^r>\max\{i'_u\}_{u=1}^s
$
we have  
\[
\del_{i_1,j_1}\cdots \del_{i_r,j_r}\cdot (ft_{i'_1,j'_1}\cdots t_{i'_s,j'_s})=
gt_{i'_1,j'_1}\cdots t_{i'_s,j'_s}.
\]
\end{lem}
\begin{proof}
We assume $
\min\{j_u\}_{u=1}^r>\max\{j'_u\}_{u=1}^s
$ (the other case follows by symmetry). 
Recall that $\mathscr I$ is the left ideal of $\sPD$ generated by $\sD^{(1)}$ (see Subection~\ref{subsec:sPDgr6}).
Set
$
f':=\del_{i_r,j_r}\cdot f
$. Then
$
\del_{i_r,j_r}f=f'+
\sum_{(i',j')}b_{i',j'}\del_{i',j'}
$, where the $b_{i',j'}\in\sPD$ and the sum is over all
pairs $(i',j')$ that satisfy $i_r\leq i'\leq m$ and  $j_r\leq j'\leq n$. In particular 
$j'\not\in\{j'_1,\ldots,j'_s\}$, hence by 
Lemma~\ref{lem:delijonarbr} we obtain
\[
\del_{i_r,j_r}ft_{i'_1,j'_1}\cdots t_{i'_s,j'_s}=
f't_{i'_1,j'_1}\cdots t_{i'_s,j'_s}+
\sum_{(i',j')}b_{i',j'}\del_{i',j'}t_{i'_1,j'_1}\cdots t_{i'_s,j'_s}\in f't_{i'_1,j'_1}\cdots t_{i'_s,j'_s}+\mathscr I.
\]
 This means $\del_{i_r,j_r}\cdot (ft_{i'_1,j'_1}\cdots t_{i'_s,j'_s})=f't_{i'_1,j'_1}\cdots t_{i'_s,j'_s}$. The proof is completed by induction on $r$.  
\end{proof}

Recall the operators $\mathbf D'_{k,r}$ from Section~\ref{sec:Introduction}. We have
\[
\mathbf D'_{1,0}+(q^2-1)\mathbf D'_{1,1}=
1+(q^2-1)\sum_{i=1}^nt_{m,i}\del_{m,i}.
\]
The next lemma is a consequence  of~\cite[Thm 1]{BKV06}
but we give an elementary, independent proof. This also makes the proofs of Theorem~\ref{thm-Main-A} and Theorem~\ref{thm:C} independent of~\cite[Thm 1]{BKV06}.
\begin{lem}
\label{lem:D'10}
Set 
$D:=
\mathbf D'_{1,0}+(q^2-1)\mathbf D'_{1,1}
$.
Then 
$D\cdot t_{a_1,b_1}\cdots t_{a_r,b_r}=q^{2\sum_{i=1}^r\llbracket m,a_i\rrbracket}t_{a_1,b_1}\cdots t_{a_r,b_r}$.
\end{lem}

\begin{proof}
By Remark~\ref{rmk:precorderbasis} the monomials \[
f=f(a_{1,1},\ldots,a_{m,n}):=t_{m,n}^{a_{m,n}}\cdots t_{m,1}^{a_{m,1}}\cdots t_{1,n}^{a_{1,1}}\cdots t_{1,1}^{a_{1,1}}
\] form a basis of $\sP$. From relations (R1) and (R2) it follows that it suffices to prove the assertion for such monomials.  By Lemma~\ref{lem:316} the assertion is reduced to the case where $a_{i,j}=0$ for $i<m$. For $j>i$ we have 
$t_{m,i}\del_{m,i}t_{m,j}=q^2t_{m,j}t_{m,i}\del_{m,i}$. By successive application of the latter relation, followed by Lemma~\ref
{lem:i+jttoa} and 
Lemma~\ref{lem:316}, we obtain
\[
(q^2-1)t_{m,i}\del_{m,i}
\cdot f=
q^{2\sum_{j=i+1}^n a_{m,j}}
(q^{2a_{m,i}}-1)f.\]
After summing up over $1\leq i\leq n$, the assertion of the lemma is reduced to the algebraic identity
\[
1+\sum_{i=1}^n
q^{2\sum_{j=i+1}^n a_{m,j}}
(q^{2a_{m,i}}-1)=q^{2\sum_{i=1}^n a_{m,i}}
,\]
which can be verified by a  straightforward computation.
\end{proof}

\begin{dfn}
Given any two ordered pairs of integers $(i,j)$ and $(i',j')$, we write $(i,j)\lhd(i',j')$ if $i\leq i'$ and $j\leq j'$  and at least one of the latter inequalities is strict.\end{dfn}
 Let $\mathscr I_{a,b}$ denote the left ideal of $\sPD$ that is generated by the $\del_{i,j}$ where $i\geq a$ and $j\geq b$.

\begin{lem}
\label{lem:del1kt1k-a}
Let $a\geq 0$ and let $1\leq k\leq n$. Then $\del_{1,k}t_{1,k}^{a+1}=\mathbf c(a)t_{1,k}^a+D$ where 
$\mathbf c(a)$ is as in~\eqref{eq:c9a} and
 $D\in\mathscr I_{1,k}$.
\end{lem}
\begin{proof}
Follows by induction on $a$. For $a=0$ the assertion follows from the relation
\begin{equation}
\label{eq:del1kt1k}
\del_{1,k}t_{1,k}=1+q^2t_{1,k}\del_{1,k}+D_1\quad\text{where }
D_1\in
\sum_{(1,k)\lhd (i,\ell)}\sPD\del_{i,\ell}.
\end{equation}
Suppose that for a given $a\geq 0$ we have $\del_{1,k}t_{1,k}^{a}=\mathbf c(a-1)t_{1,k}^{a-1}+D_2$ with $D_2\in\mathscr I_{1,k}$. Using~\eqref{eq:del1kt1k} we obtain 
\[
\del_{1,k}t_{1,k}^{a+1}=
\left(1+q^2t_{1,k}\del_{1,k}+D_1\right)t_{1,k}^a
=(1+q^2\mathbf c(a-1))t_{1,k}^a+q^2t_{1,k}D_2+D_1t_{1,k}^a.
\]
From Lemma~\ref{lem:delijonarbr}
it follows that $D_1t_{1,k}^a\in
\mathscr I_{1,k}$. Finally note that
$\mathbf c(a)=1+q^2\mathbf c(a-1)$.  
\end{proof}

\begin{lem}
\label{lem:del1k-bt1k-a}
Let $a,b\geq 0$ and let $1\leq k\leq n$. 
\begin{itemize}
\item[\rm (i)] If $b>a$ then $\del_{1,k}^{b+1}t_{1,k}^{a+1}\in\mathscr I_{1,k}$.
\item[\rm (ii)] If $b\leq a$ then 
 $\del_{1,k}^{b+1}t_{1,k}^{a+1}=\mathbf c(a,b)t_{1,k}^{a-b}+D$ for some
$\mathbf c(a,b)\in\Bbbk$, 
 where 
$D\in\mathscr I_{1,k}$. Furthermore $\mathbf c(a,0)=\mathbf c(a)$ and $\mathbf c(a,b+1)=\mathbf c(a,b)\mathbf c(a-b-1)$ for $b<a$.
\end{itemize}
\end{lem}
\begin{proof}
(i) Follows from the equality 
$\del_{1,k}^{b+1}t_{1,k}^{a+1}=\del_{1,k}^{b-a}\del_{1,k}^{a+1}t_{1,k}^{a+1}$ and Lemma~\ref{lem:del1kt1k-a}.

(ii)
We use induction on $b$. For $b=0$ this is Lemma~\ref{lem:del1kt1k-a}. If  $b+1\leq a$ then
\begin{align*}
\del_{1,k}^{b+2}t_{1,k}^{a+1}
&=
\del_{1,k}
\del_{1,k}^{b+1}t_{1,k}^{a+1}
=\mathbf c(a,b)\del_{1,k}t_{1,k}^{a-b}
+\del_{1,k}D=
\mathbf c(a,b)\mathbf c(a-b-1)t_{1,k}^{a-b-1}+\mathbf c(a,b)D_1+\del_{1,k}D,
\end{align*}
where $D_1,D\in
\mathscr I_{1,k}$. 
Part (ii) follows immediately. 
\end{proof}
\begin{lem}
\label{lem:a,bk<=nfin}
Let $a,b\geq 0$ and let $1\leq k\leq n$. Assume that $f\in\sP$ is a product of the $t_{1,j}$ for $j\leq k-1$. 
Then the following hold:
\begin{itemize}
\item[\rm (i)] 
If $b>a$ then $\del_{1,k}^bt_{1,k}^af\in \mathscr I_{1,k}$. 
\item[\rm (ii)] If $b\leq a$ then 
$\del_{1,k}^{b+1}t_{1,k}^{a+1}f=\mathbf c(a,b)t_{1,k}^{a-b}f+D$ where 
$D\in\mathscr I_{1,k}$
and $\mathbf c(a,b)$ is as in Lemma~\ref{lem:del1k-bt1k-a}.
\end{itemize}\end{lem}

\begin{proof}
(i) Follows from 
Lemma~\ref{lem:del1k-bt1k-a}(i) and 
Lemma~\ref{lem:delijonarbr}.

(ii) 
From Lemma~\ref{lem:del1k-bt1k-a}(ii) we have 
$
\del_{1,k}^{b+1}t_{1,k}^{a+1}f=
\mathbf c(a,b)t_{1,k}^{a-b}f+Df
$,
where $D\in
\mathscr I_{1,k}$. The assumption on $f$ and Lemma~\ref{lem:delijonarbr} imply  that $Df\in
\mathscr I_{1,k}$. 
\end{proof}

\begin{rmk}
\label{rmk:cab}
It is easy to verify that 
$\mathbf c(a,b)=\mathbf c(a)\mathbf c(a-1)\cdots \mathbf c(a-b)$ for $a\geq b\geq 0$. We extend the domain of $\mathbf c(a,b)$ to pairs $(a,b)$ satisfying $a,b\geq -1$ by setting 
$\mathbf c(a,b)=0$ for $-1\leq a< b$ and $\mathbf c(a,b)=1$ for $a\geq b=-1$. 
Note that  $\mathbf c(a,b)$ is always a polynomial in $q^2$ with integer coefficients. Furthermore, when $a\geq b$ the degree of $\mathbf c(a,b)$ as a polynomial in $q$ is $(b+1)(2a-b)$. 
\end{rmk}
For   a $k$-tuple of non-negative integers $\mathsf a:=(a_1,\ldots,a_k)$, where $k\leq n$,  we set $t^\mathsf a:=t_{1,k}^{a_k}\cdots t_{1,1}^{a_1}$ and $\del^\mathsf a:=\del_{1,1}^{a_1}\cdots \del_{1,k}^{a_k}$. 
\begin{lem}
\label{lem:rl;kei}
Let $1\leq k'<k_r<\ldots<k_1\leq n$. Also, let 
$a_1,\ldots, a_{k'}\geq 0$ and  $b_1,\ldots,b_r\geq 0$. 
Set $\mathsf a:=(a_1,\ldots,a_{k'})$  and  $f:=t^\mathsf a:=t_{1,k'}^{a_{k'}}\cdots t_{1,1}^{a_1}$.  Then
\[
\del_{1,k'}^bt_{1,k_1}^{b_1}\cdots t_{1,k_r}^{b_r}f=f_1+D,
\]
where $f_1\in\sP$ and $D\in\mathscr I_{1,k'}$. If $a_{k'}<b$ then
$f_1=0$. If $a_{k'}\geq b$ then 
\[
f_1=q^{b(b_1+\cdots +b_r)}\mathbf c(a_{k'}-1,b-1)t_{1,k_1}^{b_1}\cdots t_{1,k_r}^{b_r} t^{\mathsf a'}
\quad
\text{ where }\quad\mathsf a':=(a_1,\ldots,a_{k'-1},a_{k'}-b).
\]
\end{lem}

\begin{proof}
The assertion is trivial for $b=0$. If $b_1=\cdots=b_r=0$ then the assertion follows from Lemma~\ref{lem:a,bk<=nfin}(ii) and 
Remark~\ref{rmk:cab}.  
Next assume without loss of generality that $b_1\geq 1$. First suppose that $b=1$. Using Lemma~\ref{lem:delijonarbr} we obtain
\begin{align*}
\del_{1,k'}t_{1,k_1}^{b_1}
\cdots
t_{1,k_r}^{b_r}
f&=
qt_{1,k_1}\del_{1,k'}t_{1,k_1}^{b_1-1}
t_{1,k_2}^{b_2}\cdots t_{1,k_r}^{b_r}
f\\
&+(q-q^{-1})
\sum_{1<i\leq m}t_{i,k_1}\del_{i,k'}t_{1,k_1}^{b_1-1}
t_{1,k_2}^{b_2}\cdots t_{1,k_r}^{b_r}
f=
qt_{1,k_1}\del_{1,k'}t_{1,k_1}^{b_1-1}
t_{1,k_2}^{b_2}\cdots t_{1,k_r}^{b_r}f+D_1,
\end{align*}
where $D_1\in
\mathscr I_{2,k'}
$. By repeating  the above calculation 
and then using Lemma~\ref{lem:del1kt1k-a} we obtain 
\begin{align*}
\del_{1,k'}t_{1,k_1}^{b_1}
\cdots
t_{1,k_r}^{b_r} f
&=q^{b_1+\cdots +b_r} t_{1,k_1}^{b_1}
\cdots
t_{1,k_r}^{b_r} \del_{1,k'}t^{\mathsf a}+D_2,\\
 &=
q^{b_1+\cdots +b_r}
\mathbf c(a_{k'}-1) t_{1,k_1}^{b_1}
\cdots
t_{1,k_r}^{b_r}t^{\mathsf a-\mathsf e_{k'}}+D_2,
\end{align*}
where $D_2\in\mathscr I_{1,k'}$, 
$\mathsf a-\mathsf e_{k'}:=(a_1,\ldots,a_{k'-1},a_{k'}-1)$
and we define
$\mathbf{c}(-1):=0$. This completes the proof for $b=1$. For $b>1$ we just repeat the 
above argument.   
\end{proof}
\begin{lem}
\label{lem:Dbta(i)-(iii)}
Let $\mathsf a:=(a_1,\ldots, a_n)$ and 
$\mathsf b:=(b_1,\ldots,b_n)$ be $n$-tuples of non-negative integers. Then the following statements hold. 
\begin{itemize}
\item[\rm (i)] $\del^{\mathsf b}\cdot t^{\mathsf a}=0$ if $b_i>a_i$ for at least one $1\leq i\leq n$. 
\item[\rm (ii)] Assume that $a_i\geq b_i$ for all $1\leq i\leq n$. Then
\begin{equation}
\label{eq:dlkjf's|}
\del^{\mathsf b}\cdot t^{\mathsf a}=
\left(q^{\sum_{i=2}^n(a_i-b_i)(b_1+\cdots +b_{i-1})}
\prod_{i=1}^n\mathbf c(a_i-1,b_i-1)
\right)
t^{\mathsf a-\mathsf b},
\end{equation}
and
\begin{equation}
\label{eq:dlkjf's||}
t^{\mathsf b} \del^{\mathsf b}\cdot t^{\mathsf a}
=
\left(q^{\sum_{i=2}^n(2a_i-2b_i)(b_1+\cdots +b_{i-1})}
\prod_{i=1}^n\mathbf c(a_i-1,b_i-1)
\right)
t^{\mathsf a}
.\end{equation}
\end{itemize}
\end{lem}

\begin{proof}
(i) Follows from Lemma~\ref{lem:a,bk<=nfin} and Lemma~\ref{lem:rl;kei}.  

(ii) 
By Lemma~\ref{lem:rl;kei}, 
$
\del^{\mathsf b}t^{\mathsf a}
=\mathbf c(a_n-1,b_n-1)
\del^{\mathsf b'}
t_{1,n}^{a_n-b_n}t^{\mathsf a'}+D_1
$
where $\mathsf a':=(a_1,\ldots,a_{n-1})$, $\mathsf b':=(b_1,\ldots,b_{n-1})$ and  $D_1\in\mathscr I_{1,n}$. 
Again by Lemma~\ref{lem:rl;kei}, 
\[
\del^{\mathsf b'}
t_{1,n}^{a_n-b_n}
t^{\mathsf a'}
=q^{(a_n-b_n)b_{n-1}}\mathbf c(a_{n-1}-1,b_{n-1}-1)\del^{\mathsf b''}t_{1,n}^{a_n-b_n}
t_{1,n-1}^{a_{n-1}-b_{n-1}}
t^{\mathsf a''}+D_2,
\]
where $\mathsf a'':=(a_1,\ldots,a_{n-2})$, $\mathsf b'':=(b_1,\ldots,b_{n-2})$ and $D_2\in\mathscr I_{1,n-1}$. Continuing in this fashion we finally obtain~\eqref{eq:dlkjf's|}. 
For~\eqref{eq:dlkjf's||} we should compute the scalar relating 
$t^\mathsf bt^{\mathsf{a-b}}$ and $t^\mathsf a$. This is straightforward using the relations 
 $t_{1,i}t_{1,j}=qt_{1,j}t_{1,i}$ for $i<j$.\end{proof}

\section{Differential operators associated to the $K_\la$}
\label{subsec:Lifting}
Let $\ULw$, $\URw$ and $\ULRw$ be defined as in Subsection~\ref{subsec:Pi}.
The main goal of this section is to prove that certain elements of the Cartan subalgebras of $U_L$ and $U_R$ belong to $\ULw$ and $\URw$. 
This is established in
 Proposition~\ref{prp:xm-yn-EF}.
\subsection{Cartan elements in $\ULw$, $\URw$ and $\ULRw$} 
\label{subsec:CEin}
 For  $1\leq a\leq m$ and $1\leq b\leq n$ we set
\begin{equation}
\label{eq:xayb}
\la_{L,a}:=-\sum_{i=a}^m2\eps_i\quad\text{and}\quad
\la_{R,b}:=-\sum_{i=b}^n 2\eps_i.
\end{equation}
As in~\eqref{eq:Klaa;f} these weights correspond to $K_{\la_{L,a}}\in U_{\g h,L}$ and $K_{\la_{R,b}}\in U_{\g h,R}$, respectively.

\begin{prp}
\label{prp:xm-yn-EF}
Let $K_{\la_{L,a}}\in U_L$ and $K_{\la_{R,b}}\in U_R$ be as in~\eqref{eq:xayb}, where $1\leq a\leq m$ and $1\leq b\leq n$. Then $K_{\la_{L,a}}\in \ULw$ and $K_{\la_{R,b}}\in\URw$.   
\end{prp}
\begin{proof}
We only prove the assertion  for $K_{\la_{L,a}}$ (for $K_{\la_{R,b}}$ the argument is similar).
 First we verify the case $a=m$. By a straightforward computation based on Remark~\ref{rmk:actionformulas} we have  \[
 K_{\la_{L,m}}\cdot t_{a_1,b_1}\cdots t_{a_r,b_r}=q^{2\sum_{i=1}^r\llbracket m,a_i\rrbracket}t_{a_1,b_1}\cdots t_{a_r,b_r}.
 \] By Lemma~\ref{lem:D'10}  the action of $\mathbf D'_{1,0}+(q^2-1)\mathbf D'_{1,1}$ on $\sP$ is the same as the action of $K_{\la_{L,m}}$. Thus, by Proposition~\ref{prp:faithful-action} we obtain 
\begin{equation}
\label{eq:Jahfl}
\phi_U(K_{\la_{L,m}}\otimes 1)=\mathbf D'_{1,0}+(q^2-1)\mathbf D'_{1,1}.
\end{equation} To complete the proof,
by Lemma~\ref{lem:ax->adyx}
 it suffices to verify that  for any $a<m$, the $\ad(U_L)$-invariant subalgebra of $U_L$ that is generated by  $K_{\la_{L,a+1}}$ and $K_{\la_{L,m}}$ also contains $K_{\la_{L,a}}$.  
Denoting the standard generators of $U_L$ by $E_i$, $F_i$, $K_i^{\pm 1}$, 
we set
\[
E'_{\eps_i-\eps_j}:=[E_i,[\ldots,E_{j-1}]_{q}]_{q}\quad\text{ and }
F'_{\eps_i-\eps_j}:=[F_{j-1},[\ldots,F_i]_{q^{-1}}]_{q^{-1}}\quad\text{ for $1\leq i<j\leq m$}.
\]
Let  $u:=E'_{\eps_a-\eps_m}K_{-\eps_a-\eps_m}$
and $v:=F'_{\eps_a-\eps_m}
K_{\la_{L,a+1}}$. 
By a simple induction we can verify that  
\[
u
=(1-q^2)^{-1}\ad_{E_a}\cdots \ad_{E_{m-1}}(K_{\la_{L,m}})
\ \text{ and }\
v=
(1-q^{-2})^{-1}
\ad_{F_{m-1}}\cdots \ad_{F_a}(K_{\la_{L,a+1}}),\] 
so that $u,v\in \ULw$ by Lemma~\ref{lem:ax->adyx}.
For $x,y\in U_L$ set $[x,y]:=xy-yx$. 
Since $\ULw$ is an algebra, 
\begin{equation}
\label{eq:[E,E]=}
\left[E'_{\eps_a-\eps_m},
F'_{\eps_a-\eps_m}\right]K_{-\eps_a-\eps_m}K_{\la_{L,a+1}}=
uv-q^{-2}vu\in\ULw.
\end{equation}
But the left hand side of~\eqref{eq:[E,E]=} is equal to \[
(q-q^{-1})^{-1}(K_{\eps_a-\eps_m}-K_{\eps_a-\eps_m}^{-1})K_{-\eps_a-\eps_m}K_{\la_{L,a+1}}=(q-q^{-1})^{-1}(K_{-2\eps_m}K_{\la_{L,a+1}}-K_{\la_{L,a}}).
\]
It follows that \[
K_{\la_{L,a}}
=-(q-q^{-1})(uv-q^{-2}vu)+K_{\la_{L,m}}K_{\la_{L,a+1}}
\in \ULw.\qedhere
\] 
\end{proof}

\begin{prp}
\label{prp:UwvsU(g)}
$U_L$ is generated as an algebra by $\ULw$ and $\{K_{\eps_i}\}_{i=1}^m$. Similarly, 
$U_R$ is generated as an algebra by $\URw$ and $\{K_{\eps_i}\}_{i=1}^n$. \end{prp}

\begin{proof}
We give the proof for $U_L$ (for $U_R$ the proof is similar). 
Let $\cA$ denote the subalgebra of $U_L$ generated by $\ULw$ and $\{K_{\eps_i}\}_{i=1}^m$.
Set $\rho:=\sum_{i=1}^m i\eps_i$. 
Then $K_{-2\rho}=K_{\la_{L,1}}\cdots K_{\la_{L,m}}$, hence by Proposition~\ref{prp:xm-yn-EF} we have $K_{-2\rho}\in \ULw$. 
Lemma~\ref{lem:ax->adyx} implies that 
$E_iK_{-2\rho}=(1-q^2)^{-1}\ad_{E_i}(K_{-2\rho})\in \ULw$, so that $E_i\in\cA$. 
By a  similar argument we can prove that  $F_i\in \cA$ as well. Also by our assumption $K_{\eps_i}\in\cA$ for $	1\leq i\leq m$, hence $\cA=U_L$. 
\end{proof}

\section{Explicit formulas for $\phi_U(K_{\la_{L,a}}\otimes 1)$ and $\phi_U(1\otimes K_{\la_{R,b}})$}
\label{section5}
In this section we
compute explicit formulas for $\phi_U(K_{\la_{L,a}}\otimes 1)$ and $\phi_U(1\otimes K_{\la_{R,b}})$, where $K_{\la_{L,a}}$ and $K_{\la_{R,b}}$ are defined 
in 
\eqref{eq:xayb}. 
These explicit formulas are used in the proof of Theorem~\ref{thm:MainthmB}.

\subsection{Eigenvalues of $\mathbf D_{n,r}$ and $q$-factorial Schur polynomials}

\label{rmk:qfacSchurm}
For  any integer partition $\nu$ such that $\ell(\nu)\leq n$, let $s_\nu$ denote the $q$-factorial Schur polynomial in $n$ variables associated to $\nu$, defined by
\[
s_\nu(x_1,\ldots,x_n;q):=\frac{
\det\left(\prod_{k=0}^{\nu_j+n-j-1}
(x_i-q^k)
\right)_{1\leq i,j\leq n}}{\prod_{1\leq i<j\leq n}(x_i-x_j)
}.
\]
Recall that $\mathbf D(r,a,b)\in\sPD_{a\times b}$ is the $q$-differential operator defined in~\eqref{eq:Drab}. 
We will need the following statement, which is a variation of~\cite[Thm 1]{BKV06}.

\begin{prp}
\label{prp:55.22.1}
Let $\la$ be an integer partition satisfying  
$\ell(\la)\leq n$. 
Then the restriction of 
$\mathbf D(r,n,n)\in\sPD_{n\times n}$ 
to the irreducible $U_{LR}$-submodule $V_\la^*\otimes V_\la^*$ of $\sP_{n\times n}$
is a scalar multiple of identity, 
the scalar being
\begin{equation}
\label{eq:scalarr}
\boldsymbol\varphi_{\la,r,n}(q):=\frac{(-1)^rq^{r-r^2-2r(n-r)}}{(1-q^2)^r}
s_{(1^r)}(q^{2(\lambda_1+n-1)},\ldots,q^{2(\lambda_{n-1}+1)},
q^{2\lambda_{n}}
;q^2).
\end{equation}
\end{prp}
\begin{proof}
We show that the assertion  follows from an analogous result in the setting of operators in $\mathrm{Pol}(\mathrm{Mat}_{n})_q$ acting on $\C[\mathrm{Mat}_n]_q$ (see Subsection~\ref{subsec:rela})
that is proved in~\cite[Thm 1]{BKV06}. In the following proof we use the notation introduced in Subsection~\ref{subsec:rela}. In particular we set 
$A:=\Z[q,q^{-1}]$.

\textbf{Step 1.}
 For each irreducible component $V_\la^*\otimes V_\la^*$ of $\sP_{n\times n}$ (see Proposition~\ref{prp:glmglndecom}) we choose a  lowest weight vector 
$v_\la\in\sP_{n\times n}^A$ 
 for the $U_{LR}$-action  and set  $W_\la^{A}:=U_{LR}^{A}\cdot v_\la$. From the explicit formulas of the action of $U_{LR}$ (see Remark~\ref{rmk:actionformulas}) it follows that $W_\la^A\sseq \sP_{n\times n}^A$. 
Furthermore, 
the canonical map $W_\la^A\otimes_A\Bbbk\to V_\la^*\otimes V_\la^*$ is an isomorphism. 

\textbf{Step 2.}
By Corollary~\ref{cor:POP} we obtain  a map
\begin{equation}
\label{eq:basechange}
\sPD_{n\times n}^{A}
\xrightarrow{\ D\mapsto D\otimes 1\ }
\sPD_{n\times n}^{ A}\otimes_{A}\C
\xrightarrow{\ \ \cong\ \ }
\mathrm{Pol}(\mathrm{Mat}_{n})_{q_\circ},
\end{equation}
that restricts to  a map 
\begin{equation}
\label{eq:basechange1}
\sP_{n\times n}^{ A}
\xrightarrow{\ D\mapsto D\otimes 1\ }
\sP_{n\times n}^{ A}\otimes_{ A}\C
\xrightarrow{\ \ \cong\ \ }
\C[\mathrm{Mat}_{n}]_{q_\circ}.
\end{equation}
We also have a commutative diagram
\[
\xymatrix{
U_{LR}^A\otimes_A\sP_{n\times n}^A\ar[r]
\ar[d]_{(-)\otimes_A\C}
& \sP_{n\times n}^A \ar[d]^{(-)\otimes_A\C}\\
U_{q_\circ}(\gl_n)\otimes U_{q_\circ}(\gl_n)\otimes \C[\mathrm{Mat}_n]_{q_\circ} \ar[r] & \C[\mathrm{Mat}_n]_{q_\circ}
}
\]where the top horizontal map is the restriction of the $U_{LR}$-module structure on 
$\sP_{n\times n}$ 
and the bottom horizontal map is~\eqref{eq:Uqmmmnn} in the special case $m=n$.
Let us denote both of the maps~\eqref{eq:basechange} 
and~\eqref{eq:basechange1}
by $\boldsymbol\beta_{q_\circ}$. 
Then
$\boldsymbol\beta_{q_\circ}(t_{i,j})=(1-q_\circ^2)_{}^{-\frac12} z^i_j$ and $\boldsymbol\beta_{q_\circ}(\del_{i,j})=(1-q_\circ^2)_{}^{-\frac12} (z^i_j)^*$. 
From the definition of $\mathbf D(r,n,n)$ it is clear that $\mathbf D(r,n,n)\in\sPD_{n\times n}^A$. 
In addition 
$\boldsymbol\beta_{q_\circ}\left((1-q^2)^r\mathbf D(r,n,n)\right)=y_r$, where $y_r$ is the operator defined in~\cite[Eq. (11)]{BKV06}. 

\textbf{Step 3.}
From Proposition~\ref{prp:SSVvsUS} it follows that $\boldsymbol\beta_{q_\circ}(v_\la)$ is the joint highest weight vector for the irreducible submodule 
$\C[\mathrm{Mat}_n]_{q_\circ,\la}$ of 
$\C[\mathrm{Mat}_n]_{q_\circ}$ that is defined in~\cite[Sec. 2]{BKV06}. Thus Step 1 and the commutative diagram of Step 2 imply that $\boldsymbol\beta_{q_\circ}(W_\la^A)=
\boldsymbol\beta_{q_\circ}(U_{LR}^A\cdot  v_\la)\sseq \C[\mathrm{Mat}_n]_{q_\circ,\la}$.

\textbf{Step 4.} Fix $\la$ such that $\ell(\la)\leq n$,  choose any vector $w\in W_\la^A$, and set \[w':=(1-q^2)^r\left(\mathbf D(r,n,n)-\boldsymbol\varphi_{\la,r,n}\right)\cdot w.
\]
From Steps 1--2 above it follows that 
$\boldsymbol\beta_{q_\circ}(w')=y_r\cdot \boldsymbol\beta_{q_\circ}(w)-(1-q_\circ^2)\boldsymbol\varphi_{\la,r,n}(q_\circ)\boldsymbol\beta_{q_\circ}(w)$. Since $\boldsymbol\beta_{q_\circ}(w)\in \C[\mathrm{Mat}]_{q_\circ,\la}$,  by~\cite[Thm 1]{BKV06} we obtain  $\boldsymbol\beta_{q_\circ}(w')=0$. Since 
evaluations at $q_\circ$ for infinitely many $q_\circ$ separate  the points of $\sP_{n\times n}^A$, it follows that $w'=0$. 

\textbf{Step 5.} By Step 4 we have $\mathbf D(r,n,n)\cdot w=\boldsymbol\varphi_{\la,r,n}(q)w$ for $w\in W_\la^A$. Since $W_\la^A$ spans $V_\la^*\otimes V_\la^*$ over $\Bbbk$, the same assertion holds for all $w\in V_\la^*\otimes V_\la^*$. 
\end{proof}

\begin{rmk}
In our forthcoming work~\cite{LSS22b}, we prove a broad extension of 
Proposition~\ref{prp:55.22.1} 
 for Capelli operators on quantum symmetric spaces. 
\end{rmk}

The polynomials $s_\nu$ are  specializations of the \emph{interpolation Macdonald polynomials} $R_\la$ 
defined in~\cite{Sah96} (see also~\cite{Kn97} and~\cite{Ok97}). 
In the rest of this section we follow the notation of~\cite[Sec. 0.3]{Sah11}. Let $R_\la(x_1,\ldots,x_n;q,t)$ denote the unique symmetric polynomial with coefficients in $\Q(q,t)$ that satisfies the following conditions:
\begin{itemize}
\item[(i)]
$\deg R_\la=|\la|$.
\item[(ii)] 
$R_\la(q^{\mu_1},\ldots,q^{\mu_i}t^{1-i},\ldots,q^{\mu_n}t^{1-n};q,t)=0$ for all partitions  
$\mu\neq \la$ that satisfy $|\mu|\leq |\la|$.
\item[(iii)] $R_\la$ can be expressed as 
$R_\la=m_\la+\sum_{\mu\neq \la} c_{\mu,\la}m_\mu$, where the $m_\mu$ denote 
the monomial symmetric polynomials. 
\end{itemize}
It is known  
\cite[Prop. 2.8]{Kn97}
that \[
s_\la(x_1,\ldots,x_n;q)=q^{(n-1)|\la|}R_\la(q^{1-n}x_1,\ldots,q^{1-n}x_n;q,q).
\] 
The proof of Lemma~\ref{lem:qSchurident} below uses Okounkov's binomial theorem for interpolation Macdonald polynomials~\cite{Ok97}. We remark that in~\cite{Ok97} 
the interpolation Macdonald polynomials are defined slightly differently, and are denoted by  the $P_\la^*$, but one can show that  
\begin{equation}
\label{eq:okoun}
P_\la^*(x_1,\ldots,x_n;q,t)=R_\la(x_1,x_2t^{-1},x_nt^{-n+1};q,t).
\end{equation}
For two integer partitions $\la,\mu$ such that $\ell(\la),\ell(\mu)\leq n$, let
$\left[\la\atop\mu\right]_{q,t}$  
 denote the $(q,t)$-binomial coefficient defined in~\cite{Ok97}. Thus
\begin{equation}
\label{eq:la,muqt}
\displaystyle\begin{bmatrix}
\,\,\la\,\,\\ \mu
\end{bmatrix}_{q,t}
:=\frac{P^*_\mu
(q^{\la_1},\ldots,q^{\la_n};q,t)}{
P^*_\la
(q^{\la_1},\ldots,q^{\la_n};q,t)}
.\end{equation}

\begin{lem}
\label{lem:1n1r}
For $0\leq r\leq n$ we have
$\left[
1^n\atop 1^r
\right]_{q,q}
=q^{-r(n-r)}
\frac{(q^n-1)\cdots (q^{n-r+1}-1)}
{(q^r-1)\cdots (q-1)}$.
\end{lem}

\begin{proof}
The proof is a straightforward but somewhat tedious calculation based on a general combinatorial formula 
in~\cite[Thm 0.8]{Sah11}
for the $(q,t)$-binomial coefficients.
We give a brief outline of this calculation. In the notation 
of~\cite{Sah11}, the value of~\eqref{eq:la,muqt} can be expressed as a sum of the form $\sum_T wt(T)$, where $T$ is a standard tableau of shape $\la\bls \mu$. For $\la:=(1^n)$ and $\mu:=(1^r)$, there is only one such tableau. 
By direct calculation one obtains
\[
\la^i=(1^{n-i})\ ,\  
a_{\la^i,\la^{i+1}}=\frac{t^{-n+i+1}(1-t^{n-i})}{1-t}\ ,\ \frac{
|\oline{\la^i}|-|\oline{\la^{i+1}}|
}{
|\oline{\la}|-|\oline{\la^{i+1}}|
}=\frac{t^{i}(1-t)}{1-t^{i+1}}.
\]
From these, the assertion of the lemma follows immediately.  
\end{proof}

\begin{lem}
\label{lem:qSchurident}
Set $\nu_r:=(1^r)$ for $0\leq r\leq n$.
Then
\[
\sum_{r=0}^n
q^{-{r\choose 2}-r(n-r)}
s_{\nu_r}(q^{n-1}x_1,\ldots,q^{n-i}x_i,\ldots, x_n;q)=x_1\cdots x_n
.\] 
\end{lem}
\begin{proof}
This is stated in~\cite[Prop. 10]{BKV06} without a proof. We show that it  is a special case of  Okounkov's binomial 
theorem~\cite[Eq. (1.11)]{Ok97}. 
More specifically,  from~\eqref{eq:okoun} it follows that  
\[
P_{\nu_r}^*
(x_1,\ldots,x_n;q,q)
=q^{(1-n)r}s_{\nu_r}
(q^{n-1}x_1,\ldots,x_n;q)
.
\]
We now consider the 
identity~\cite[Eq. (1.11)]{Ok97}
for $t:=q$ and $\la:=(1^n)$. Then the left hand side of~\cite[Eq. (1.11)]{Ok97} is equal to  
$x_1\cdots x_n$, whereas its right hand side is equal to
\[
\sum_{r=0}^n\begin{bmatrix}
1^n\\ 1^r
\end{bmatrix}_{q,q}
q^{-{r\choose 2}}\frac{(q^r-1)\cdots (q-1)}
{(q^n-1)\cdots (q^{n-r+1}-1)}
s_{\nu_r}(q^{n-1}x_1,\ldots,x_n;q).
\]
 To complete the proof, we use Lemma~\ref{lem:1n1r}.
\end{proof}

%
%
%
%

\subsection{The explicit formulas}
\label{subsec-quantumminor}

Let $\mathbf D_{n,r}$ and $\mathbf D_{m,r}'$ be as in Section~\ref{sec:Introduction}.
For 
 $0\leq r\leq m$  we set
\begin{equation}
\label{eq:Dr=Dnr=Dmr}
\mathbf{D}_r:=\mathbf D_{n,r}^{}=\mathbf D'_{m,r}.
\end{equation}

\begin{prp}
\label{prp:eval-xL}
$
\phi_U(K_{\la_{L,1}}\otimes 1)=\phi_U(1\otimes K_{\la_{R,1}})=
\sum_{r=0}^{m}
(q^2-1)^r\mathbf D_r
$.
%
\end{prp}
\begin{proof}
Both $K_{\la_{L,1}}\otimes 1$ and $1\otimes K_{\la_{R,1}}$ act on
$\sP^{(d)}$ by the scalar $q^{2d}$ (this is easy to verify using Remark~\ref{rmk:actionformulas}).
Since $\sP$ is a faithful $\sPD$-module, 
by Proposition~\ref{prp:glmglndecom} it suffices to verify that for every partition $\la$ that satisfies $\ell(\la)\leq \min\{m,n\}$ and $|\la|=d$, 
the restriction of $\sum_{r=0}^m (q^2-1)^r\mathbf D_r$ to  
the irreducible $U_{LR}$-submodule $V_\la^*\otimes V_\la^*$ of $\sP$ is multiplication by the scalar $q^{2d}$.

\textbf{Step 1.} First we prove the assertion in the case  $m=n$. In this case $\mathbf D_r=\mathbf D(r,n,n)$, hence by Proposition~\ref{prp:55.22.1}  it is enough to verify that\begin{equation}
\label{eq:qr-r2}
\sum_{r=0}^n
q^{r-r^2-2r(n-r)}
s_{\nu_r}(q^{2(\lambda_1+n-1)},\ldots,q^{2(\lambda_{n-1}+1)},
q^{2\lambda_{n}}
;q^2)=q^{2(\la_1+\cdots+\la_n)}
.\end{equation}
Equality~\eqref{eq:qr-r2} follows from
Lemma~\ref{lem:qSchurident} after substituting $q$ by $q^\frac{1}{2}$. 

\textbf{Step 2.} Henceforth assume $m<n$ (by Lemma~\ref{lem:etam,n}
the proof when $m>n$ is similar). 
Let
 \[
 \Emb=\Emb_{m\times n}^{n\times n}:\sPD\to\sPD_{n\times n}
 \] be the embedding of algebras defined in~\eqref{eq:embeD}, so that 
 $\Emb(t_{i,j})=t_{i+n-m,j}$ and $\Emb(\del_{i,j})=\del_{i+n-m,j}$.
Set $\tilde{\mathbf D}_r:=\mathbf D(r,n,n)$. 
%
By Corollary~\ref{lem:sigmastaus-}, $\oline M^\mathbf i_\mathbf j\cdot\mathsf (\Emb(\sP))=0$ unless $\mathbf i=(u_1,\ldots,u_r)$ satifies $u_1\geq n-m+1$. 
Thus, for every $f\in\sP$ we have $\tilde {\mathbf D}_r\cdot \Emb(f)=\Emb(\mathbf D_r)\cdot \Emb(f)$ when $0\leq r\leq m$, and $\tilde {\mathbf D}_r\cdot \Emb(f)=0$ when $m<r\leq n$. 

\textbf{Step 3.}
Recall that $\mathscr I$ denotes the left ideal of $\sPD$ that is generated by $\sD^{(1)}$. Let $\mathscr I'$ denote the left ideal of $\sPD_{n\times n}$ that is generated by $\sD_{n\times n}^{(1)}$. Let $\Emb:\sPD\to\sPD_{n\times n}$ be as in Step 2. For $D\in\sPD$ and $f\in\sP$ we have
$
(D\cdot f-Df)\in\mathscr I
$, hence \[
\Emb(D\cdot f)-\Emb(D)\Emb(f)=\Emb(D\cdot f-Df)\in\mathscr I'.
\]
 But also 
$\Emb(D)\cdot \Emb(f)-\Emb(D)\Emb(f)\in\mathscr I'$. From the last two relations we obtain $\Emb(D)\cdot \Emb(f)-\Emb(D\cdot f)\in\mathscr I'$. But in addition $\Emb(D)\cdot \Emb(f)-\Emb(D\cdot f)\in\sP_{n\times n}$, hence 
$
\Emb(D)\cdot \Emb(f)=\Emb(D\cdot f)
$.

\textbf{Step 4.} Let $f\in\sP^{(d)}$. From Step 3 and Step 2 it follows that 
\begin{equation}
\label{eq:PUJ}
\Emb\left(
\sum_{r=0}^m(q^2-1)^r\mathbf D_r\cdot f
\right)
=
\Emb\left(\sum_{r=0}^m(q^2-1)^r\mathbf D_r\right)\cdot \Emb(f)=\sum_{r=0}^n
(q^2-1)^r
\tilde {\mathbf D}_r\cdot \Emb(f).
\end{equation}
From Step 1 it follows that 
$\sum_{r=0}^n(q^2-1)^r\tilde{\mathbf  D}_r\cdot \Emb(f)=q^{2d}\Emb(f)$. Since $\Emb$ is an injection, from~\eqref{eq:PUJ} we obtain
$\sum_{r=0}^m (q^2-1)^r\mathbf D_r\cdot f=q^{2d}f$. 
\end{proof}

%
%

\begin{prp}
\label{prp:phiUxaphiUyb}
For $1\leq a\leq m$ and $1\leq b\leq n$ we have 
\begin{equation}
\label{eq:phixaphiyb}
\phi_U(K_{\la_{L,a}}\otimes 1)=
\sum_{r=0}^{m-a+1}(q^2-1)^r\mathbf D'_{m-a+1,r}
\quad\text{and}\quad
\phi_U(1\otimes K_{\la_{R,b}})=
\sum_{r=0}^{n-b+1}(q^2-1)^r\mathbf D_{n-b+1,r}.
\end{equation}
\end{prp}
\begin{proof}
We give the proof for  $K_{\la_{L,a}}\otimes 1$ only (the proof for $1\otimes K_{\la_{R,b}}$ is similar). Every element of $\sP$ is expressible as a linear combination of monomials of the form
$t_{i_1,j_1}\cdots t_{i_k,j_k}$ where $i_1\geq \cdots \geq i_k$. Choose $k'\leq k$ such that $i_{k'}\geq a$
and $i_{k'+1}<a$.  Then \[
(K_{\la_{L,a}}\otimes 1)\cdot t_{i_1,j_1}\cdots t_{i_k,j_k}=q^{2k'}t_{i_1,j_1}\cdots t_{i_k,j_k}.
\] 
Set $D:=\sum_{r=0}^{m-a+1}(q^2-1)\mathbf D_{m-a+1,r}'
$. From Lemma~\ref{lem:316} it follows that 
\[
D\cdot \left(
t_{i_1,j_1}\cdots t_{i_k,j_k}
\right)=
(D\cdot t_{i_1,j_1}\cdots t_{i_{k'},j_{k'}})
t_{i_{k'+1},j_{k'+1}}\cdots t_{i_k,j_k}.
\]
Hence it suffices to prove that \begin{equation}
\label{eq:tterjk}
D\cdot t_{i_1,j_1}\cdots t_{i_{k'},j_{k'}}=q^{2k'}
t_{i_1,j_1}\cdots t_{i_{k'},j_{k'}}.
\end{equation}
Set $\tilde m:=m-a+1$. Let $\Emb=\Emb_{\tilde m\times n}^{m\times n}:\sPD_{\tilde m\times n}\to\sPD $ be the embedding of algebras defined in~\eqref{eq:embeD}, so that 
$\Emb(t_{i,j})=t_{i+a-1,j}$ and $\Emb(\del_{i,j})=\del_{i+a-1,j}$.  
Set $\tilde {\mathbf D}_r:=\mathbf D(r,\tilde m,n)\in\sPD_{\tilde m\times n}$.
Similar to the proof of Proposition~
\ref{prp:eval-xL}
we have $\Emb(\tilde {\mathbf D}_r)=\mathbf D'_{m-a+1,r}$ and $\Emb(\tilde {\mathbf D}_r\cdot f)=\Emb(\tilde {\mathbf D}_r)\cdot \Emb(f)$ for $f\in \sP_{\tilde m\times n}$. 
Proposition~\ref{prp:eval-xL} for $\sPD_{\tilde m\times  n}$ yields 
$
\sum_{r=0}^{m-a+1}\tilde {\mathbf D}_r\cdot f=q^{2k'}f
$ for $f:=t_{i_1-a+1,j_1}\cdots t_{i_{k'}-a+1,j_{k'}}\in\sP_{\tilde m\times n}$. 
By applying $\Emb$ to both sides we obtain~\eqref{eq:tterjk}.
\end{proof}

\section{Some properties of polarization operators}
\label{sec-new6}
In this section we investigate invariance and generation properties of the  $\mathsf L_{i,j}$, the $\mathsf R_{i,j}$, and their variants.
\subsection{Invariants and the operators $\mathsf L_{i,j}$, $\mathsf R_{i,j}$}
\label{subsec:Now6.1}
Recall from~\eqref{eq:YtoZ}
that $\mathcal Y^\mathcal Z$ denotes the centralizer of $\mathcal Z$ in $\mathcal Y$. 
\begin{lem}
\label{lem:commwithUW}
$\End_\Bbbk(\sP)^{\Lmw}=\End_\Bbbk(\sP)^\Lm$ 
and\,
$\End_\Bbbk(\sP)^{\Rnw}=\End_\Bbbk(\sP)^\Rn$.
%
\end{lem}
\begin{proof}
We only give the proofs of the  two assertions for $\Lm$.  
The inclusion $\End_\Bbbk(\sP)^{\Lmw}\supseteq\End_\Bbbk(\sP)^\Lm$ is trivial because $\Lmw\sseq \Lm$. To prove 
$\End_\Bbbk(\sP)^{\Lmw}\sseq\End_\Bbbk(\sP)^\Lm$,  choose any $T\in\End_\Bbbk(\sP)^{\Lmw}$. 
From Proposition~\ref{prp:phiUxaphiUyb} it follows that $T$ commutes with \[
\phi_U(K_{2\eps_a}\otimes 1)=\phi_U\left((K_{\la_{L,a}}^{-1}\otimes 1)(K_{\la_{L,a+1}}^{}\otimes 1)\right) \quad\text{for }1\leq a\leq m,
\]
where we assume $K_{\la_{L,m+1}}:=1$.   From
Proposition~\ref{prp:glmglndecom}
(and also from Remark~\ref{rmk:actionformulas}) it follows that 
$\phi_U(K_{\eps_i}\otimes 1)$ is a  diagonalizable operator whose eigenvalues are powers of $q$.
In particular, the eigenspaces of $\phi_U(K_{2\eps_i}\otimes 1)$ and $\phi_U(K_{\eps_i}\otimes 1)$ are the same. Thus $T$ also commutes with $\phi_U(K_{\eps_i}\otimes 1)$. 
Finally, Proposition~\ref{prp:UwvsU(g)} implies that $T\in \End_\Bbbk(\sP)^\Lm$.
\end{proof}
As in Subsection~\ref{subsec-H-inv} set \[
\sPD_{(\epsilon_L)}:=\left\{D\in\sPD\,:\,
x\cdot D:=\epsilon_L(x)D\text{ for }x\in U_L\right\},
\] where $\epsilon_L$ denotes the counit of $U_L$. We define $\sPD_{(\epsilon_R)}$, $\left(\mathscr A_{k,l,n}\right)_{(\epsilon_R)}$, 
$\left(\mathscr A^\mathrm{gr}_{k,l,n}\right)_{(\epsilon_R)}$ and $\left(\mathscr A_{k,l,n}^{\mathrm{gr},(r,s)}\right)_{(\epsilon_R)}
$
 similarly (where $\epsilon_R$ denotes the counit of $U_R$).  
\begin{lem}
\label{prp:u-EiFiKi}
$\sPD^\Lmw=
\sPD^\Lm=\sPD_{(\epsilon_L)}
$
and  
$\sPD^\Rnw=\sPD^\Rn=\sPD_{(\epsilon_R)}$.
\end{lem}
\begin{proof}
From Lemma~\ref{lem:commwithUW} it follows that 
$\sPD^\Lmw=
\sPD^\Lm$ and 
$\sPD^\Rnw=\sPD^\Rn$. By Proposition~\ref{prp:xBDvsxD}, the action of $U_{LR}$ on $\sPD$ is the restriction of  the action of $U_{LR}$ on $\End_\Bbbk(\sP)$ that is defined in Lemma~\ref{rmk:Hopf-eps}. Thus 
Lemma~\ref{rmk:Hopf-eps} implies that
$\sPD^\Lm=\sPD_{(\epsilon_L)}$ and $
\sPD^\Rn=\sPD_{(\epsilon_R)}$.
\end{proof}

\begin{lem}
$\left(\mathscr A_{k,l,n}\right)_{(\epsilon_R)}$ is  a subalgebra of $\mathscr A_{k,l,n}$ and $\left(\mathscr A^\mathrm{gr}_{k,l,n}\right)_{(\epsilon_R)}$ 
 is a subalgebra of $\mathscr A_{k,l,n}^\mathrm{gr}$. 
\end{lem}
\begin{proof}
Follows immediately from the fact that both 
$\mathscr A_{k,l,n}$ and $\mathscr A_{k,l,n}^\mathrm{gr}$ are $U_R$-module algebras.
\end{proof}

Recall that by definition, $\mathscr A_{k,l,n}^\mathrm{gr}$ is a subalgebra of $\sPD^\mathrm{gr}:=\sPD^\mathrm{gr}_{m\times n}$ where $m:=\max\{k,l\}$. 
For
$1\leq i\leq k$ and $1\leq j\leq l$ define
$\tilde{\mathsf L}_{i,j}^\mathrm{gr}\in\mathscr A_{k,l,n}^\mathrm{gr}$ by
\begin{equation}
\label{eq:tildjf;lkjf}
\tilde{\mathsf L}_{i,j}^\mathrm{gr}:=\sum_{r=1}^n\tilde t_{i,r}\tilde \del_{j,r}
=
\sum_{r=1}^n t_{m-i+1,r} \del_{m-j+1,r}.
\end{equation}
Under the isomorphism of Corollary~\ref{cor:Agrkln-vs-gr(A)} the $\tilde{\mathsf L}_{i,j}^\mathrm{gr}$ correspond to the 
$\mathrm{gr}(\tilde{\mathsf L}_{i,j})\in\mathrm{gr}(\mathscr A_{k,l,n})$.

\begin{lem}
\label{lem:LijRijLLRR}
$\tilde{\mathsf L}^\mathrm{gr}_{i,j} \in \left(\mathscr A^\mathrm{gr}_{k,l,n}\right)_{(\epsilon_R)}$
and
$\tilde{\mathsf L}_{i,j} \in \left(\mathscr A_{k,l,n}\right)_{(\epsilon_R)}$
for $1\leq i\leq k$ and $1\leq j\leq l$. 
%
%
\end{lem}
\begin{proof}
Recall that $\mathscr A_{k,l,n}^\mathrm{gr}$ is a $U_R$-module algebra because it is a $U_R$-stable subalgebra of $\sPD^\mathrm{gr}$. For $\tilde{\mathsf L}^\mathrm{gr}_{i,j}$ the assertion follows from the formulas of Remark~\ref{rmk:actionformulas}. For example 
\begin{align*}
E_s\cdot \tilde{\mathsf L}_{i,j}^\mathrm{gr}
=E_s\cdot \sum_{r=1}^n \tilde t_{i,r}\tilde \del_{j,r}
&=\sum_{r=1}^n 
(E_s\cdot \tilde t_{i,r}) (K_s\cdot \tilde \del_{j,r})+
\sum_{r=1}^n 
\tilde t_{i,r} (E_s\cdot \tilde \del_{j,r}) \\
&=(-q^{-1}\tilde t_{i,n-s})(q\tilde \del_{j,n+1-s})+(\tilde t_{i,n-s})(\tilde\del_{j,n+1-s})=0.
\end{align*}
Since the map $\mathsf P_{k,l,n}:\mathscr A_{k,l,n}^\mathrm{gr}\to\mathscr A_{k,l,n}$ is a $U_R$-module homomorphism, we have 
$\tilde{\mathsf L}_{i,j} \in \left(\mathscr A_{k,l,n}\right)_{(\epsilon_R)}$.
\end{proof}

\begin{lem}

\label{U_R-on-Psinn}
The $U_L$-submodule of  $\sPD$ that is generated by $\mathsf{L}_{m,m}$ contains  $\mathsf{L}_{i,j}$ for $1\leq i,j\leq m$. Similarly, the $U_R$-submodule of $\sPD$ that is generated by $\mathsf R_{n,n}$ contains $\mathsf R_{i,j}$ for $1\leq i,j\leq n$.  
\end{lem}

\begin{proof}
We only give the proof for the assertion about the $U_R$-submodule generated by $\mathsf R_{n,n}$ (the other assertion  is proved similarly). Denote this submodule by $\mathscr M$. First we prove the following relations for the $U_R$-action:
\begin{equation}
\label{eq:3rels}
E_{j}\cdot\mathsf{R}_{i,j+1}
=\mathsf{R}_{i,j}
\text{ and }
F_{i}\cdot\mathsf{R}_{i+1,j}=-q\mathsf{R}_{i,j}
\text{ for }j\neq i\quad,\quad
F_i\cdot \mathsf R_{i+1,i}=-q\mathsf{R}_{i,i}+q^{-1}\mathsf R_{i+1,i+1}.\end{equation}
The proofs of these relations are similar and based on the explicit formulas given in 
Remark~\ref{rmk:actionformulas}.
For example using $\Delta(F_i)=F_i\otimes 1+K_i^{-1}\otimes F_i$ we have 
\begin{align*}
F_{i}\cdot \mathsf{R}_{i+1,i}=
F_i\cdot \sum_{r=1}^n t_{r,i+1}\del_{r,i}
&=
\sum \left(
(F_i\cdot t_{r,i+1})(\del_{r,i})+
(K_i^{-1}\cdot t_{r,i+1})(F_i\cdot \del_{r,i})\right)
\\
&=\sum_{r=1}^n
\left(
-qt_{r,i}\del_{r,i}+q^{-1}t_{r,i+1}\del_{r,i+1}
\right)
=-q\mathsf{R}_{i,i}+q^{-1}\mathsf{R}_{i+1,+1}.
\end{align*}
Since $\mathsf{R}_{n,n}\in\mathscr M$, 
from the second relation
in ~\eqref{eq:3rels} 
 for $j=n$ we obtain 
$\mathsf{R}_{i,n}\in\mathscr M$ for $i\leq n$. 
Then using the first relation in~\eqref{eq:3rels} successively for $j=n-1,\ldots,i+1$ we obtain $\mathsf{R}_{i,j}\in \mathscr M$ for $i<j$. 
The above argument can be repeated  with the roles of  $E_i$ and $F_i$ switched. This yields $\mathsf{R}_{i,j}\in\mathscr M$ for $i>j$. Finally, from $\mathsf{R}_{n,n}$ and the third relation in~\eqref{eq:3rels} we obtain $\mathsf{R}_{i,i}\in \mathscr M$ for $1\leq i\leq n$. 
\end{proof}

\begin{cor}
\label{cor:Rij-in-L}
$\mathsf L_{i,j}\in\Lmw$ for all $1\leq i,j\leq m$ and $\mathsf R_{i,j}\in \Rnw$ for $1\leq i,j\leq n$. 
\end{cor}
\begin{proof}
From
~\eqref{eq:phiUadYY} it follows that  $\Lmw$ is $U_L$-stable. 
By~\eqref{eq:Jahfl}
we have $\phi_U(K_{\la_{L,m}}\otimes 1)=1+(q^2-1)\mathsf L_{m,m}$, hence $\mathsf L_{m,m}\in \Lmw$. Hence by  Lemma~\ref{U_R-on-Psinn} we have $\mathsf L_{i,j}\in \Lmw$. The proof of  $\mathsf R_{i,j}\in\Rnw$ is similar.  
\end{proof}

\subsection{The $U_q(\gl_k)\otimes U_q(\gl_l)$-module decomposition of $\left(\mathscr A_{k,l,n}^\mathrm{gr}\right)_{(\epsilon_R)}$}
Given two irreducible $U_R$-modules $V_\la$ and $V_\mu$, the canonical isomorphism 
$
\left(V_\la^*\otimes V_\mu\right)_{(\epsilon_R)}\cong \Hom_{U_R}(V_\la,V_\mu)
$ implies that 
\begin{equation}
\label{eq:VlaVmu*}
\dim\left(V_\la^*\otimes V_\mu\right)_{(\epsilon_R)}=\begin{cases}
1& \text{if } \la=\mu,\\
0&\text{if }\la\neq \mu.
\end{cases}
\end{equation}
 Thus, from~\eqref{eq:PD(r,s)d} it follows that
as   $U_q(\gl_k)\otimes U_q(\gl_l)$-modules we have 
\begin{equation}
\label{eq:grPDrrUR22}
\left(\mathscr A^{\mathrm{gr},(r,s)}_{k,l,n}\right)_{(\epsilon_R)}=0
\ \text{ for }\ r\neq s\ \text{ and }\ 
\left(\mathscr A^{\mathrm{gr},(r,r)}_{k,l,n}\right)_{(\epsilon_R)}
\cong \bigoplus_{\la\in\Lambda_{d,r}} V_\la^*\otimes V_\la^{},\ \text{ where }\ d:=\min\{k,l,n\}.
\end{equation}

\section{The map ${{}\Gamma}_{k,l,n}$}
\label{sec:pfofThmA}
Let $k,l,n\geq 1$ be integers such that $k,l\leq n$. In this section we define a map \[
{{}\Gamma}_{k,l,n}:\sP_{k\times l}\to\mathscr A^\mathrm{gr}_{k,l,n}
\] that  is a bijection onto the subalgebra $\big(\mathscr A_{k,l,n}^\mathrm{gr}\big)_{(\epsilon_R)}$ of $\mathscr A_{k,l,n}^\mathrm{gr}$. 
A similar map was also used in~\cite{LZZ11}. 
 The ideas of the proofs of Lemma~\ref{lem:injGamgr} and Lemma~\ref{lem:PsiSURJ} are taken from~\cite{LZZ11}. Recall from Definition~\ref{dfn-PDPD1} that $\sPD^\mathrm{gr}_{n\times n}\cong \sP_{n\times n}\otimes \sD_{n\times n}$ as a vector space.

\subsection{Construction of ${{}\Gamma}_{k,l,n}$}
\label{subsec:Now7.1}
For $n\geq 1$ set 
\begin{equation}
\label{eq:Gammagr--n}
{}\Gamma_n:\sP_{n\times n}\to
\sPD_{n\times n}^\mathrm{gr}\ ,\ 
{}\Gamma_n:=
(1\otimes \iota)\circ \Delta_{\sP},
\end{equation}
where
$\iota:\sP_{n\times n}\to\sD_{n\times n}$ is the anti-isomorphism of bialgebras defined in~\eqref{eq:defn-iota-map} and  
$\Delta_{\sP}$ is the coproduct of $\sP_{n\times n}$. In particular in Sweedler's notation we have
${{}\Gamma}_n(u)=\sum u_1\otimes \iota(u_2)$ for $u\in\sP_{n\times n}$. 
\begin{lem}
\label{lem:injGamgr}
Let $\epsilon_\mathscr D$ be the counit of 
$\mathscr D_{n\times n}$. 
Then 
the map $1\otimes \epsilon_\mathscr D:\sPD^\mathrm{gr}_{n\times n}\to\sP_{n\times n}$ is a left inverse to ${{}\Gamma}_n$. In particular, ${{}\Gamma}_n$ is an injection. 
\end{lem}
\begin{proof}
This is equivalent to the relation $\sum \epsilon_\mathscr D(\iota(u_2))u_1=u$ for $u\in\sP_{n\times n}$. Since $\iota$ is a linear bijection, it suffices to verify that 
$
\sum \epsilon_\mathscr D(\iota(u_2))\iota(u_1)=\iota(u)
$. Since $\iota$ is an anti-isomorphism of coalgebras, the latter relation follows from the defining property of the counit $\epsilon_\mathscr D$.  
\end{proof}


Recall from Proposition~\ref{prp:existence(6)}(i)
that the assignments $\tilde{t}_{i,j}\mapsto \tilde{t}_{i,j}$ and $\tilde{\partial}_{i,j}\mapsto \tilde{\partial}_{i,j}$ result in an embedding of algebras
\begin{equation}
\label{eq:eandegrmn}
\Emb_{k\times l}^{n\times n}:\sPD_{k\times l}\to \sPD_{n\times n}.
\end{equation}
We denote the restriction of the map~\eqref{eq:eandegrmn}
 to the subalgebra $\sP_{k\times l}$ by the same notation, that is
\begin{equation}
\label{eq:emm-restricted}
\Emb_{k\times l}^{n\times n}:\sP_{k\times l}\to \sP_{n\times n}.
\end{equation}

\begin{lem}
\label{lem:gtrGMA}
For $\Emb=\Emb_{k\times l}^{n\times n}$  as
in~\eqref{eq:emm-restricted}
we have
${}\Gamma_n\left(\Emb(\sP_{k\times l})\right)\sseq \mathscr A_{k,l,n}^\mathrm{gr}$.

\end{lem} 
\begin{proof}
For a monomial $t_{i_1,j_1}\cdots t_{i_r,j_r}\in\sP_{n\times n}$ we have \begin{align*}\Delta(t_{i_1,j_1}\cdots t_{i_r,j_r})
&=\Delta(t_{i_1,j_1})\cdots \Delta(t_{i_r,j_r})
=\sum_{1\leq a_1,\ldots,a_r\leq n}
t_{i_1,a_1}\cdots t_{i_r,a_r}\otimes t_{a_1,j_1}\cdots t_{a_r,j_r}.
\end{align*}
Set $m':=n-k$ and $n':=n-l$. From the above equality it follows that 
\[
{{}\Gamma}_n\left(\Emb(t_{i_1,j_1}\cdots t_{i_r,j_r})\right)=
\sum_{1\leq a_1,\ldots,a_r\leq n}
t_{i_1+m',a_1}\cdots t_{i_r+m',a_r}\otimes \del_{j_r+n',a_r}\cdots \del_{j_1+n',a_1}.
\]
Thus
${{}\Gamma}_n(\Emb(t_{i_1,j_1}\cdots t_{i_r,j_r}))\in\mathscr A_{k,l,n}^\mathrm{gr}$ (see Remark~\ref{rmk:presentationofPDgr}).
\end{proof}
Lemma~\ref{lem:gtrGMA} justifies that  the following definition is valid.
\begin{dfn}
\label{dfn:PsI}
We define $
{{}\Gamma}_{k,l,n}:\sP_{k\times l}\to\mathscr A_{k,l,n}^\mathrm{gr}
$ to be the
 unique map that makes the diagram 
\[
\xymatrix@C+2pc{\sP_{k\times l} \ar[r]^{\Emb_{k\times l}^{n\times n}} \ar@{.>}[dr]_{{{}\Gamma}_{k,l,n}} & \sP_{n\times n}\ar[r]^{{{}\Gamma}_n} & \sPD_{n\times n}^\mathrm{gr}\\
 & \mathscr A_{k,l,n}^\mathrm{gr}\ar@{^{(}->}[ur]}
\]
commutative. 
\end{dfn}

\begin{lem}
\label{lem:Psi-1-1}
The map ${{}\Gamma}_{k,l,n}$ is injective.
\end{lem}
\begin{proof}
Since  ${\Emb}_{k\times l}^{n\times n}$  is an injective map, this follows from Lemma~\ref{lem:injGamgr}.
\end{proof}

\begin{lem}
\label{lem:PsiSURJ}
${{}\Gamma}_{k,l,n}(\sP_{k\times l}^{(d)})=\left(
\mathscr A^{\mathrm{gr},(d,d)}_{k,l,n}\right)_{(\epsilon_R)}$ for $d\geq 0$. 
\end{lem}
\begin{proof}
First we prove that
${{}\Gamma}_{k,l,n}(\sP_{k\times l})\sseq 
\left(\mathscr A^\mathrm{gr}_{k,l,n}\right)_{(\epsilon_R)}$.
By Definition~\ref{dfn:PsI}
 it suffices to prove that ${{}\Gamma}_n(\sP_{n\times n})\sseq \left(\sPD_{n\times n}^\mathrm{gr}\right)_{(\epsilon_R)}$. 
By standard properties of the antipode of $U_R$, 
\begin{equation}
\label{S-1NewEq}
\sum x_2S^{-1}(x_1)=\epsilon_R(x)1 \quad\text{for }x\in U_R.
\end{equation} It follows that
for $x\in U_R$ and
$u\in \sP_{n\times n}$ we have \begin{align*}
x\cdot{{}\Gamma}_n(u)&
=\sum 
(x_1\cdot u_1)
\otimes (x_2\cdot \iota(u_2))\\
&=
\sum
\big( u_{11}
\lag 
\iota(u_{12}),S^{-1}(x_1)\rag
\big)
\otimes 
\big(
\iota(u_2)_1
\lag 
\iota(u_2)_2,x_2
\rag
\big) 
& 
\text{(By~\eqref{eq:xoy=<><>u2} and~\eqref{eq:xoy=<><>u3})}
\\
&=
\sum \lag \iota(u_{2}),S^{-1}(x_1)\rag
\lag 
\iota(u_{3}),x_2\rag
u_{1}
\otimes \iota(u_{4})
&
\text{(By coassociativity)}
\\
&=\sum 
\epsilon_{R}(x)
\lag \iota(u_2),1\rag u_1\otimes \iota(u_3)
&
\text{(By~\eqref{eq:laxy)copro} and~\eqref{S-1NewEq})}
\\
&
=
\epsilon_{R}(x)\sum u_1
\otimes \iota(u_2)
=
\epsilon_{R}(x){{}\Gamma}_n(u).
&\text{(By counit relation of $\sD_{n\times n}$)}
\end{align*}
Thus we have proved $x\cdot {{}\Gamma}_n(u)=\epsilon_R(x)\,{{}\Gamma}_n(u)$, that is, 
${{}\Gamma}_n(u)\in\left(\sPD^\mathrm{gr}_{n\times n}\right)_{(\epsilon_R)}$.
From~\eqref{eq:Gammagr--n} it follows  that
 ${{}\Gamma}_{k,l,n}\left(\sP_{k\times l}^{(r)}\right)\sseq\mathscr A^{\mathrm{gr},(r,r)}_{k,l,n}$.
 Consequently,
\begin{equation}
\label{gammainclusssss-}
{{}\Gamma}_{k,l,n}\left(\sP_{k\times l}^{(r)}\right)\sseq\mathscr A^{\mathrm{gr},(r,r)}_{k,l,n}\cap\left(\sPD_{n\times n}^\mathrm{gr}\right)_{(\epsilon_R)}=
\left(
\mathscr A^{\mathrm{gr},(r,r)}_{k,l,n}
\right)_{(\epsilon_R)}.
\end{equation}    
By Lemma~\ref{lem:Psi-1-1} , to complete the proof 
it suffices to verify that  the two  sides of~\eqref{gammainclusssss-} have equal dimensions. 
Since $k,l\leq n$,  from~\eqref{eq:grPDrrUR22} and 
Proposition~\ref{prp:glmglndecom} it follows that both of these vector spaces have dimension equal to $\sum_{\la\in \Lambda_{d,r}}
\mathsf d(\la,k)\mathsf d(\la,l)$ where 
$\mathsf d(\la,k)$ (respectively, $\mathsf d(\la,l)$) denotes the dimension of the  $U_q(\gl_k)$-module (respectively, $U_q(\gl_l)$-module) associated to $\la$. 
\end{proof}

\section{The product $\star_{k,l,n}^{}$ on $\sP_{k\times l}$}
\label{sec:specialcasem=n}
Throughout this section we assume that  $m=n$ (so that 
$
U_L\cong U_R\cong U_q(\gl_n)$) and  $1\leq k,l\leq n$. 

\subsection{An explicit formula for the product of $\sPD_{n\times n}^\mathrm{gr}$}

Recall that  $\sP_{n\times n}$ and $\sD_{n\times n}$ are subalgebras of $U_q(\gl_n)^\bullet$ (see Definition~\ref{dfn:Ubul} and 
Definition~\ref{dfn-f--P-D}). 
Also recall that given $f,g\in U_q(\gl_n)^\bullet$, we define $\big\lag f\otimes g,\EuScript R^{(n)}\big\rag$ 
and $\big\lag f\otimes g,\underline{\EuScript R}^{(n)}\big\rag$
as in 
\eqref{eq:<fog,R>}. 

\begin{prp}
\label{prp:prod-of-PDgr}
Let $a,a'\in\sP_{n\times n}$ and $b,b'\in\sD_{n\times n}$.
Then the  product of 
$\sPD_{n\times n}^{\mathrm{gr}}$ satisfies  
\[
(a\otimes b)(a'\otimes b')=
\sum
\lag 
\iota(a_1')^\natural
\otimes
(b_1)^\natural,\EuScript R^{(n)}\rag
\lag
\iota(a_3')\otimes b_3,\EuScript R^{(n)}\rag
aa_2'\otimes b_2b',
\]
where $(\Delta_\sP\otimes 1)\circ\Delta_\sP(a')=\sum a'_1\otimes a'_2\otimes a'_3$ and 
$(\Delta_\sP\otimes 1)\circ\Delta_\sP(b)=\sum b_1\otimes b_2\otimes b_3$, with $\Delta_\sP$ and $\Delta_\sD$ denoting the coproducts of $\sP_{n\times n}$ and $\sD_{n\times n}$ respectively. 
\end{prp}
\begin{proof}
 We need to compute $\left(\check{\underline{\EuScript R}}_{LR}\right)_{B,A}(b\otimes a')$ where $B:=\sD_{n\times n}$ and $A:=\sP_{n\times n}$.
Recall from~\eqref{eq:dfnofRcheck} that 
$\check{\underline{\EuScript R}}_{LR}=(\check{\underline{\EuScript R}}_L)_{13}
(\check{\underline{\EuScript R}}_R)_{24}$. The map 
\[
\sP_{n\times n}\to
U_q(\gl_n)^\circ\ ,\ 
a\mapsto a\circ\xi_{-1/q}
\] intertwines between the $U_R$-module structures $\cR_\sP$ 
and  right translation 
on $U_q(\gl_n)^\circ$, in the sense of Remark~\ref{rmk:LRactionsHcirc}. 
Thus 
Lemma~\ref{lem:fog,Rexpand} implies that 
\begin{equation}
\label{eq:R2444-1}
(\underline{\EuScript R}_R)_{24}
(b\otimes a')=\sum b_1^{}\otimes a'_1
\lag b_2\otimes (a_2'\circ \xi_{-1/q}),\underline{\EuScript R}^{(n)}\rag.
\end{equation}
Similarly, the maps  
\[
\sD_{n\times n}\to U_q(\gl_n)^\circ\ ,\ 
b\mapsto b^\natural\quad\text{ and }\quad
\sP_{n\times n}\to 
U_q(\gl_n)^\circ
\ ,\ 
a\mapsto (a\circ \xi_{-q})^\natural
\]
intertwine the actions $\cL_\sD$  and $\cL_\sP$ (on $\sD_{n\times n}$ and 
$\sP_{n\times n}$ respectively)
with right translation. 
This is because the map $u\mapsto u^\natural$ on $U_q(\gl_n)^\circ$ that is defined in~\eqref{eq:naturalduali} is an anti-automorphism of the coalgebra structure of $U_q(\gl_n)^\circ$. 
Thus
\begin{equation}
\label{eq:R2444-2}
\left(\underline{\EuScript R}_{L}\right)_{13}(b\otimes a')
=\sum b_2^{} \otimes a'_2
\lag (b_1)^\natural\otimes (a_1'\circ\xi_{-q})^\natural,\underline{\EuScript R}^{(n)}\rag
.\end{equation}
From~\eqref{eq:R2444-1} and~\eqref{eq:R2444-2}
it follows that 
\begin{equation}
\label{eq:first-rpduct}
(a\otimes b)(a'\otimes b')=
\sum
\lag (b_1)^\natural \otimes (a_1'\circ\xi_{-q})^\natural,\underline{\EuScript R}^{(n)}\rag
\lag
b_3\otimes (a_3'\circ \xi_{-1/q}),\underline{\EuScript R}^{(n)}\rag
aa_2'\otimes b_2b'.
\end{equation}
Since $S^2(x)=\xi_{q^2}(x)$ for $x\in U_q(\gl_n)$, we have
$
\iota(v)\circ S=v\circ \xi_{-1/q}\circ S^2=
v\circ \xi_{-q}
$
for $v\in U_q(\gl_n)^\circ$, and  thus Lemma~\ref{lem:rel-S-natural} implies that
$
(v\circ \xi_{-q})^\natural=(\iota(v)\circ S)^\natural=\iota(v)^\natural\circ S^{-1}
$. By a similar argument  $v\circ 
\xi_{-1/q}=\iota(v)\circ S^{-1}$. 
Thus in~\eqref{eq:first-rpduct} we can substitute $(a_1'\circ\xi_{-q})^\natural$ by $\iota(a_1')^\natural\circ S^{-1}$ and $a_3'\circ\xi_{-1/q}$ by $\iota(a'_3)\circ S^{-1}$.
After these substitutions, the assertion of the proposition follows from~\eqref{lem:fxgoS}.
\end{proof}

\subsection{The product $\star_{k,l,n}^{}$ on $\sP_{k\times l}$ and the map 
${{}\Gamma}_{k,l,n}$}
\label{subsec:newprodt}
We start by defining a binary product 
\[
\sP_{n\times n}\otimes \sP_{n\times n}\to \sP_{n\times n}\ ,\ u\otimes v\mapsto u\star_n v.
\]
\begin{dfn}
\label{dfn:deffstarrrprd}
For $u,v\in \sP_{n\times n}$ we
set
\begin{equation}
\label{eq:starprod-symm-formula}
u\star_{n} v:=
\sum
\big\lag
\iota(v_1)^\natural\otimes
\iota(u_3)^\natural,
\EuScript R^{(n)}
\big\rag
\big\lag
\iota(v_3)\otimes \iota(u_2),\EuScript R^{(n)}
\big\rag
u_1v_2,
\end{equation}
where the  sum ranges over summands of $(\Delta_{\sP}\otimes 1)\circ \Delta_{\sP}(u)=\sum u_1\otimes u_2\otimes u_3$ and 
$(\Delta_{\sP}\otimes1)\circ  \Delta_{\sP}(v)=\sum v_1\otimes v_2\otimes v_3$.
\end{dfn}
For the next proposition recall that ${{}\Gamma}_n:\sP_{n\times n}\to\sPD_{n\times n}^\mathrm{gr}$ is the map defined in~\eqref{eq:Gammagr--n}. 
\begin{prp}
\label{prp:grPsi}
${{}\Gamma_n}(u\star_n v)={{}\Gamma_n}(u){{}\Gamma_n}(v)$ for $u,v\in \sP_{n\times n}$. 
\end{prp}

\begin{proof}
By Proposition~\ref{prp:prod-of-PDgr} for $a:=u_1$, $b:=\iota(u_2)$, $a':=v_1$ and $b':=\iota(v_2)$ we obtain
\begin{align*}
{{}\Gamma_n}(u){{}\Gamma_n}(v)
&
=
\sum (u_1 
\otimes \iota(u_2))(v_1\otimes \iota(v_2))\\
&
=
\sum \lag
\iota(v_1)^\natural\otimes
\iota(u_4)^\natural,\EuScript R^{(n)}\rag
\lag
\iota(v_3)\otimes \iota(u_2),\EuScript R^{(n)}\rag
u_1v_2\otimes \iota(u_3)\iota(v_4).
\end{align*}
Since $u\mapsto\iota(u)$ is an anti-automorphism of algebras, by~\eqref{eq:starprod-symm-formula} we also have
\begin{align*}
{{}\Gamma_n}(u\star_n v)& 
=
\sum
\lag
\iota(v_1)^\natural\otimes 
\iota(u_4)^\natural 
,
\EuScript R^{(n)}
\rag
\lag
\iota(v_4)\otimes 
\iota(u_3) 
,\EuScript R^{(n)}
\rag
u_1v_2\otimes \iota(v_3)\iota(u_2)
.\end{align*}
After changing the indices as $(u_1,u_2,u_3,u_4)=(u_1,u_{21},u_{22},u_3)
$ and $(v_1,v_2,v_3,v_4)=(v_1,v_2,v_{31},v_{32})$ by coassociativity, the equality ${{}\Gamma_n}(u\star_n v)={{}\Gamma_n}(u){{}\Gamma_n}(v)$ reduces to
\begin{equation}
\label{eq:fjlffjf-1}
\sum
\lag 
\iota(v_{31})\otimes 
\iota(u_{21}),\EuScript R^{(n)}\rag
\iota(u_{22})\iota(v_{32})
=
\sum
\lag
\iota(v_{32})\otimes 
\iota(u_{22}),\EuScript R^{(n)}\rag
\iota(v_{31})\iota(u_{21}).
\end{equation}
Set $f:=\iota(v_3)$ and $g:=\iota(u_2)$. Since $u\mapsto \iota(u)$ is an anti-automorphism of coalgebras, from~\eqref{lem:fxgoS} it follows 
that~\eqref{eq:fjlffjf-1} is equivalent to
\[
\sum \lag f_2\otimes g_2,\EuScript R^{(n)}\rag
g_1f_1=\sum \lag f_1\otimes g_1,\EuScript R^{(n)}\rag f_2g_2,
\]
which is a consequence of~\eqref{eq:g1f1R}.
\end{proof}

Recall that throughout this section $1\leq k,l\leq n$. 
\begin{lem}
\label{lem:e(Pm)pres}
Let $\Emb:=\Emb_{k\times l}^{n\times n}$ be as 
in~\eqref{eq:emm-restricted}. 
Then 
$\Emb(u)\star_n\Emb(v)\in\Emb(\sP_{k\times l})$ for $u,v\in\sP_{k\times l}$. 
\end{lem}

\begin{proof}
Set 
$D:= {{}\Gamma}_n(\Emb(u)\star_n \Emb(v))$. By Lemma~\ref{lem:PsiSURJ} we have
\[
{{}\Gamma}_n(\Emb(u))=
{{}\Gamma}_{k,l,n}(u)\in\left(
\mathscr A^\mathrm{gr}_{k,l,n}\right)_{(\epsilon_R)}
\quad\text{and}\quad  
{{}\Gamma}_n(\Emb(v))=
{{}\Gamma}_{k,l,n}(v)\in\left(
\mathscr A^\mathrm{gr}_{k,l,n}\right)_{(\epsilon_R)}
.\]
Since $\left(\mathscr A_{k,l,n}^\mathrm{gr}\right)_{(\epsilon_R)}$ is a subalgeba of $\mathscr A_{k,l,n}^\mathrm{gr}$,  from Proposition~\ref{prp:grPsi} it follows that  
$D\in\left(\mathscr A_{k,l,n}^\mathrm{gr}\right)_{(\epsilon_R)}$. 
Again by Lemma~\ref{lem:PsiSURJ}
there exists $w\in \sP_{k\times l}$ such that $D={{}\Gamma}_{k,l,n}(w)=\Gamma_n(\Emb(w))$.
From injectivity of $\Gamma_n$ (see Lemma~\ref{lem:injGamgr}) it follows that $\Emb(u)\star_n\Emb(v)=\Emb(w)$. Consequently we obtain $\Emb(u)\star_n\Emb(v)\in \Emb(\sP_{k\times l})$.  
\end{proof}

 Lemma~\ref{lem:e(Pm)pres} validates the following definition.
\begin{dfn}
\label{def:prodtPkl}
For $u,v\in\sP_{k\times l}$ we 
define a binary product 
\[
\sP_{k\times l}\otimes \sP_{k\times l}\to\sP_{k\times l}\ ,\ 
u\otimes v\mapsto u\star_{k,l,n}^{}v,
\]
by setting
$u\star_{k,l,n}^{} v:=\Emb^{-1}\left(\Emb(u)\star_n \Emb(v)\right)$ where $\Emb:=\Emb_{k\times l}^{n\times n}
$ is the map~\eqref{eq:emm-restricted}.
\end{dfn}

\begin{prp}
\label{prp:PsiuPsiv-Puv}
Let $u,v\in \sP_{k\times l}$. Then the following statements hold:
\begin{itemize}
\item[\rm (i)]
${{}\Gamma}_{k,l,n}(u\star_{k,l,n}^{} v)={{}\Gamma}_{k,l,n}(u){{}\Gamma}_{k,l,n}(v)$.

\item[\rm (ii)]
If $u\in\sP_{k\times l}^{(r)}$
and $v\in\sP_{k\times l}^{(s)}$ then 
$u\star_{k,l,n}^{} v\in \sP_{k\times l}^{(r+s)}$ and ${{}\Gamma}_{k,l,n}(u\star_{k,l,n}^{} v)\in\left(\mathscr A^{\mathrm{gr},(r+s,r+s)}_{k,l,n}\right)_{(\epsilon_R)}$. 
\end{itemize} 
\end{prp}

\begin{proof}
(i)
By Proposition~\ref{prp:grPsi} we have
\begin{align*}
{{}\Gamma}_{k,l,n}(u\star_{k,l,n}^{} v)&=
{{}\Gamma}_{k,l,n}(\Emb^{-1}(\Emb(u)\star_n\Emb(v)))\\
&=
{{}\Gamma}_{n}(\Emb(u)\star_n\Emb(n))=
{{}\Gamma}_{n}(\Emb(u))
{{}\Gamma}_{n}(\Emb(v))=
{{}\Gamma}_{k,l,n}(u){{}\Gamma}_{k,l,n}(v).
\end{align*}

(ii) By definitions of $\star_{k,l,n}^{}$ and 
${{}\Gamma}_{k,l,n}$ the assertions reduce to proving that for $u\in\sP_{n\times n}^{(r)}$ and 
$v\in\sP_{n\times n}^{(s)}$
 we have
$u\star_n v\in\sP_{n\times n}^{(r+s)}$ and 
${{}\Gamma}_n(u\star_n v)\in\sPD^{\mathrm{gr},(r+s,r+s)}_{n\times n}$. The latter assertions follow
 from~\eqref{eq:starprod-symm-formula} and the fact that for $d\geq 0$ we have $\Delta_\sP(\sP_{n\times n}^{(d)})\sseq
\sP_{n\times n}^{(d)}\otimes \sP_{n\times n}^{(d)}$. 
\end{proof}

\subsection{The map $\Upsilon$}
Let 
$
\Upsilon:\sP_{n\times n}\otimes \sP_{n\times n}\to \sP_{n\times n}\otimes \sP_{n\times n}
$ be the map defined by 
\begin{align}
\label{eq:USP}
\Upsilon(u\otimes v)&:= 
\sum\lag \iota(v_1)^\natural\otimes \iota(u_3)^\natural,{\EuScript R^{(n)}}
\rag
\lag
\iota(v_3)\otimes \iota(u_2),\EuScript R^{(n)}\rag u_1\otimes v_2.
\end{align}
\begin{lem}
\label{lem:star=moUps} 
Let $\mathsf m_{n\times n}:\sP_{n\times n}\otimes \sP_{n\times n}\to\sP_{n\times n}$ denote the usual product of the algebra $\sP_{n\times n}$.
Then the following statements hold.
\begin{itemize}
\item[\rm (i)]
 $
u\star_n v=\mathsf m_{n\times n}\circ\Upsilon(u\otimes v)$ for $u,v\in\sP_{n\times n}$.
\item[\rm (ii)]
$
u\star_{k,l,n}^{} v=\Emb^{-1}\big(\mathsf m_{n\times n} (\Upsilon(\Emb(u)\otimes \Emb(v)))\big)
$ for $u,v\in\sP_{k\times l}$ where $\Emb:=\Emb_{k\times l}^{n\times n}$ is as in~\eqref{eq:emm-restricted}.
\end{itemize}

\end{lem}
\begin{proof}
Straightforward from 
Definition~\ref{dfn:deffstarrrprd}.
\end{proof}

\begin{lem}
\label{lem:del(1)=0}
$\lag\del_{i_1,j_1}\cdots \del_{i_r,j_r},1\rag=0$ unless when $i_k=j_k$ for $1\leq k\leq r$. 
\end{lem}
\begin{proof}
Follows immediately from the definition of the $\del_{i,j}$ and the canonical pairing $\lag \cdot,\cdot\rag$ between $U_q(\gl_n)^\circ$ and $U_q(\gl_n)$. 
\end{proof}

\begin{prp}
\label{prp:compatib-Ups}
Let $\Emb:=\Emb_{k\times l}^{n\times n}$ be as in~\eqref{eq:emm-restricted}. Set $\mathscr W_{d_1,d_2}:=\sP_{k\times l}^{(d_1)}\otimes \sP_{k\times l}^{(d_2)}$ for $d_1,d_2\geq 0$ and 
$\mathscr W'_{d_1,d_2}:= (\Emb\otimes \Emb)(\mathscr W_{d_1,d_2})$. Then $\Upsilon(\mathscr W'_{d_1,d_2})\sseq \mathscr W'_{d_1,d_2}$.

\end{prp}

\begin{proof}
From the defining formula of $\Upsilon$ and the fact that the coproduct of $\sP_{n\times n}$ maps $\sP_{n\times n}^{(a)}$ into $\sP_{n\times n}^{(a)}\otimes \sP_{n\times n}^{(a)}$ we obtain $\Upsilon\left(\sP_{n,n}^{(a)}\otimes\sP_{n\times n}^{(b)}\right)\sseq \sP_{n,n}^{(a)}\otimes\sP_{n\times n}^{(b)}$. The claim follows if we prove that 
\begin{equation}
\label{eq:eotimese}
\Upsilon(\Emb(u)\otimes \Emb(v))
\in 
\Emb\otimes \Emb\left(\sP_{k\times l}\otimes \sP_{k\times l}\right)\quad\text{ for } 
u,v\in\sP_{k\times l}.
\end{equation}
It suffices to prove this assertion for monomials $u=t_{i_1,j_1}\cdots t_{i_r,j_r}$ and $v=t_{p_1,q_1}\cdots t_{p_s,q_s}$ in $\sP_{k\times l}$.  
Set $m':=n-k$ and $n':=n-l$. Then 
\[
((\Delta\otimes 1)\circ\Delta)(\Emb(u))=\sum \Emb(u)_1\otimes \Emb(u)_2\otimes \Emb(u)_3
,\] where
for indices $1\leq a_1,b_1,\ldots,a_r,b_r\leq n$ we have
\[
\Emb(u)_1=t_{m'+i_1,a_1}
\cdots t_{m'+i_r,a_r}\quad,\quad
\Emb(u)_2=t_{a_1,b_1}
\cdots t_{a_r,b_r}
\quad,\quad
\Emb(u)_3=t_{b_1,n'+j_1}
\cdots t_{b_r,n'+j_r}.
\]
Similarly, 
\[
((\Delta\otimes 1)\circ \Delta)(\Emb(v))=
\sum \Emb(v)_1\otimes \Emb(v)_2\otimes \Emb(v)_3,
\] where for indices $1\leq c_1,d_1,\ldots,c_s,d_s\leq n$ we have
\[
\Emb(v)_1=t_{m'+p_1,c_1}
\cdots t_{m'+p_s,c_s}\quad,\quad
\Emb(v)_2=t_{c_1,d_1}
\cdots t_{c_s,d_s}
\quad,\quad
\Emb(v)_3=t_{d_1,n'+q_1}
\cdots t_{d_s,n'+q_s}.
\]
In the rest of this proof we set 
\[
\mathsf a:=(a_1,\ldots,a_r)\ ,\ \mathsf b:=(b_1,\ldots,b_r)\ ,\ \mathsf c:=(c_1,\ldots,c_s)\ ,\ 
\mathsf d:=(d_1,\ldots,d_s).
\] 

\textbf{Step 1.}
From~\eqref{eq:USP} it follows that
\begin{equation}
\label{eq:explicitusv}
\Upsilon(\Emb(u)\otimes \Emb(v))=
\sum_{\mathsf{a,b,c,d}}
M(\mathsf{b,c})M'(\mathsf{a,b,d})
t_{m'+i_1,a_1}\cdots t_{m'+i_r,a_r}
\otimes t_{c_1,d_1}\cdots t_{c_s,d_s},
\end{equation}
where using the fact that the map $u\mapsto u^\natural$  of~\eqref{eq:naturalduali} induces an isomorphism between
$U(\gl_n)^\circ$ and $\left(U(\gl_n)^\circ\right)^\mathrm{cop}$ we have
\[
M(\mathsf{b,c}):=\lag 
\del_{m'+p_s,c_s}
\cdots
\del_{m'+p_1,c_1}
\otimes
\del_{b_r,n'+j_r}\cdots 
\del_{b_1,n'+j_1}\,,\,
\EuScript R^{(n)}
\rag
\]
and 
\[
M'(\mathsf{a,b,d}):=\lag
\del_{n'+q_s,d_s}\cdots \del_{n'+q_1,d_1}
\otimes \del_{b_r,a_r}\cdots \del_{b_1,a_1}
\,,\,\EuScript R^{(n)}\rag
.\]
Note that by definition, both of the $R$-matrix pairings
$M(\mathsf b,\mathsf c)$ and $M'(\mathsf a,\mathsf b,\mathsf d)$  
 correspond to the action of $U_q(\gl_n)$ on $\sD_{n\times n}\sseq U_q(\gl_n)^\circ$ by \emph{right} translation $\cR_\sD$ (See Subsection~\ref{subsec::quas}). 

\textbf{Step 2.} 
We prove that if $c_t\leq m'$ for some $1\leq t\leq s$ then 
$M(\mathsf{b,c})=0$. To this end, we investigate the effect of the action of $\EuScript R^{(n)}$ on the first component in $M(\mathsf{b,c})$, i.e., on $\del_{m'+p_s,c_s}
\cdots
\del_{m'+p_1,c_1}
$. 
Recall that 
  $\EuScript R^{(n)}$ acts by a product of 2-tensors of the form $E_\beta\otimes F_\beta$, where $\beta=\eps_{\ell_1}-\eps_{\ell_2}$ for $1\leq \ell_1<\ell_2\leq n$, followed by  $e^{h\sum_{i=1}^n H_{i}\otimes H_{i}}$ (which acts by scalars on tensor product of monomials in the $\del$'s and we can ignore it in the argument that follows). The $s$-fold coproduct of $E_\beta$ is a sum of $s$-tensors of the form $X:=X_s\otimes \cdots\otimes X_1$ with components in $ \{E_\beta,K_\beta,1\}$. From Remark~\ref{rmk:actionformulas} it follows that the action of $X$ on any monomial $\del_{m'+p_s,\bar c_s}\cdots 
\del_{m'+p_1,\bar c_1}$
does not increase the indices $\bar c_1,\ldots,\bar c_s$ and leaves the indices $n'+p_1,\ldots,n'+p_s$ unchanged. 
Thus Lemma~\ref{lem:del(1)=0} implies that $M(\mathsf{b,c})=0$.

\textbf{Step 3.}
We prove that if  $d_t\leq n'$ for some $1\leq t\leq s$ then $M'(\mathsf{a,b,d})=0$.
 The argument is similar to Step 2,  by investigating the action of root vectors $E_\beta$ on the first component of $M'(\mathsf{a,b,d})$, that is on $\del_{n'+q_s,d_s}\cdots \del_{n'+q_1,d_1}
$.

\textbf{Step 4.} We prove that if  $b_t\leq n'$ for some $1\leq t\leq r$ then $M(\mathsf{b,c})=0$. Again the argument is similar to Step 2. This time use use the fact that the action of the root vectors $F_\beta$  does not decrease the indices $n'+j_1,\ldots,n'+j_r$.

\textbf{Step 5.}
We prove that if  $a_t\leq n'$ for some $1\leq t\leq r$ then $ M(\mathsf{b,c})M'(\mathsf{a,b,d})=0$. The proof is slightly more complicated than Steps 2--4. By Step 4 we can assume that $\min\{b_1,\ldots,b_r\}\geq n'+1$. 
As in Steps 2--4 we can express
$M'(\mathsf{a,b,d})$ as a sum over the values
\begin{equation}
\label{eq:bigscalardelEdelF}
C_{\del_E,\del_F,\beta_1,\ldots,\beta_N}\big(\lag\left(E_{\beta_1}\cdots E_{\beta_N}
\right)\cdot\del_E\,,1\rag\big)
\big(\lag\left(F_{\beta_1}\cdots F_{\beta_N}\right)\cdot
\del_F\,,1\rag\big),
\end{equation}
where $\del_E:=\del_{n'+q_s,d_s}\cdots \del_{n'+q_1,d_1}$,
 $\del_F:=\del_{b_r,a_r}\cdots \del_{b_1,a_1}$ and
$C_{\del_E,\del_F,\beta_1,\ldots,\beta_N}$ is a scalar in $\Bbbk$ that results from the action of $e^{h\sum_{i=1}^n H_{i}\otimes H_{i}}$ (again, this scalar does not play a role in the argument that follows). For $\beta=\eps_{\ell_1}-\eps_{\ell_2}$ with $1\leq \ell_1<\ell_2\leq n$ we have 
\begin{equation}
\label{Ebeta-action}
E_\beta\cdot \del_{\ell',\ell''}=\begin{cases}
0& \text{ if }\ell_2\neq \ell'',\\
(-1)^{\ell_2-\ell_1-1}\del_{\ell',\ell_1}& \text{ if }\ell_2= \ell'',
\end{cases}
\end{equation}
and 
\begin{equation}
\label{Fbeta-action}
F_\beta\cdot \del_{\ell',\ell''}=\begin{cases}
0& \text{ if }\ell_1\neq \ell'',\\
(-1)^{\ell_2-\ell_1-1}\del_{\ell',\ell_2}& \text{ if }\ell_1= \ell''.
\end{cases}
\end{equation}
First suppose that there exists $1\leq N'\leq N$ such that $\beta_{N'}=\eps_{\ell_1}-\eps_{\ell_2}$ with 
$\ell_1\leq n'$. Then by~\eqref{Ebeta-action} and
an argument similar to Step 2 we obtain 
$\lag\left(E_{\beta_1}\cdots E_{\beta_N}
\right)\cdot\del_E\,,1\rag=0
$, hence the corresponding value~\eqref{eq:bigscalardelEdelF} vanishes. Next suppose that 
for all  $1\leq N'\leq N$ we have
$\beta_{N'}=\eps_{\ell_1}-\eps_{\ell_2}$ where $\ell_1\geq n'+1$. 
The
$r$-fold coproduct of $F_{\beta_1}\cdots F_{\beta_N}$ is a sum of $r$-tensors $X=X_r\otimes \cdots \otimes X_1$ whose components are products of the $F_{\beta_{N'}}$ and the $K_{\beta_{N'}}^{-1}$. If no $F_{\beta_{N'}}$ occurs in $X_t$  then
Lemma~\ref{lem:del(1)=0} implies that $\lag X\cdot \del_F\,,1\rag= 0$  unless $a_t=b_t\geq n'+1$. If at least one $F_{\beta_{N'}}$ occurs in $X_t$ then from~\eqref{Fbeta-action} it follows that 
$\lag X\cdot \del_F\,,1\rag= 0$ unless $a_t\geq n'+1$ (because we must have $a_t=\ell_1$). Since we have assumed that $a_t\leq n'$ we obtain  $\lag (F_{\beta_1}\cdots F_{\beta_N})\cdot \del_F\,,1\rag=0$. Thus all the values~\eqref{eq:bigscalardelEdelF} vanish and we have $M'(\mathsf{a,b,d})=0$. 

\textbf{Step 6.} From Steps 
2--5 it follows that the only 2-tensors on the right hand side of~\eqref{eq:explicitusv} that have a nonzero coefficient are those that
belong to $(\Emb\otimes \Emb)\left(\sP_{k\times l}\otimes \sP_{k\times l}\right)$.  This completes the proof of~\eqref{eq:eotimese}.
\end{proof}

\begin{prp}
\label{prp:Ups-is-bij}
$\Upsilon$ induces linear bijections $\sP_{n,n}^{(r)}\otimes\sP_{n\times n}^{(s)}\to \sP_{n,n}^{(r)}\otimes\sP_{n\times n}^{(s)}$ for all $r,s\geq 0$. 
\end{prp}
\begin{proof}
 Since $\sP_{n,n}^{(r)}\otimes\sP_{n\times n}^{(s)}$ is finite dimensional, it suffices to prove that $\Upsilon$ is an injection.
Set $\Upsilon^\iota:=(\iota^{-1}\otimes \iota^{-1})\circ\Upsilon\circ(\iota\otimes \iota)$ with $\iota$ as in~\eqref{eq:iota-uoXi}. Since $\iota$ is an antiautomorphism of bialgebras and preseves the grading of $\sP_{n\times n}$, we have
\[
\Upsilon^\iota:\sD_{n\times n}^{(r)}\otimes \sD_{n\times n}^{(s)}\to 
\sD_{n\times n}^{(r)}\otimes \sD_{n\times n}^{(s)}\quad,\quad
\Upsilon^\iota(u\otimes v)=\sum
\lag (v_3)^\natural\otimes (u_1)^\natural,\EuScript R^{(n)}\rag\lag v_1\otimes u_2,\EuScript R^{(n)}\rag u_3\otimes v_2.
\]
It suffices to prove injectivity of $\Upsilon^\iota$.

\textbf{Step 1.} 
Define $\Upsilon^{(1)}:\sD_{n\times n}^{(r)}\otimes \sD_{n\times n}^{(s)}
\to \sD_{n\times n}^{(r)}\otimes \sD_{n\times n}^{(s)}$  by
\[
\Upsilon^{(1)}(u\otimes v):=\sum
\lag
u_1\otimes v_1,\underline{\EuScript R}^{(n)}\rag u_2\otimes v_2
.\]
From Lemma~\ref{lem:fog,Rexpand}(i)
we have
\[
\sum f_1\otimes g_1\lag
g_2\otimes f_2,\underline{\EuScript R}^{(n)}\rag=\sigma\left(
\sum g_1\otimes f_1\lag
g_2\otimes f_2,\underline{\EuScript R}^{(n)}\rag
\right)=\sigma\circ \underline{\EuScript R}^{(n)}(g\otimes f)=
\big(\EuScript R^{(n)}\big)^{-1}(f\otimes g).
\] From this and Lemma~\ref{lem:fog,Rexpand}(ii) it follows that
\begin{align*}
\Upsilon^{(1)}\circ& \Upsilon^\iota (u\otimes v)
=\sum
\lag (v_4)^\natural \otimes (u_1)^\natural,\EuScript R^{(n)}\rag
\lag
v_1\otimes u_2,\EuScript R^{(n)}\rag
\lag u_3\otimes v_2,\underline{\EuScript R}^{(n)}\rag u_4\otimes v_3
\\
&=
\sum
\lag (v_3)^\natural \otimes (u_1)^\natural,\EuScript R^{(n)}\rag
\lag
\big({\EuScript R}^{(n)}\big)^{-1}
(v_1\otimes u_2),\EuScript R^{(n)}\rag
 u_3\otimes v_2
 =\sum 
 \lag (v_2)^\natural \otimes (u_1)^\natural,\EuScript R^{(n)}\rag
 u_2\otimes v_1.
\end{align*}
\textbf{Step 2.}
Define $\Upsilon^{(2)}:\sD_{n\times n}^{(r)}\otimes \sD_{n\times n}^{(s)}\to \sD_{n\times n}^{(r)}\otimes \sD_{n\times n}^{(s)}$ by $\Upsilon^{(2)}(u\otimes v):=u^\natural \otimes v$. 
Set \[
\breve\Upsilon:=\left(\Upsilon^{(2)}\right)^{-1}
\circ 
\left(\Upsilon^{(1)}\circ \Upsilon^\iota\right)\circ \Upsilon^{(2)}.
\] 
Since $u\mapsto u^\natural $ is an antiautomorphism of coalgebras on $\sD_{n\times n}$, we obtain
\[
\breve{\Upsilon}(u\otimes v)=
\sum \lag (v_2)^\natural\otimes u_2,\EuScript R^{(n)}\rag u_1\otimes v_1.
\]
\textbf{Step 3.} 
From Lemmas~\ref{lem:naturalonTD},~\ref{lem:howxc} and~\ref{lem:howiota} it follows that the assignment $v\mapsto (v\circ S)^\natural$ induces an isomorphism of coalgebras 
$\sP_{n\times n}\to \sD_{n\times n}$
that preserves the grading. 
Define a map \[\Upsilon^{(3)}:\sD_{n\times n}^{(r)}\otimes \sP_{n\times n}^{(s)}\to \sD_{n\times n}^{(r)}\otimes \sD_{n\times n}^{(s)}\ ,\ 
\Upsilon^{(3)}(u\otimes v):=u\otimes ((v\circ S)^\natural).
\] 
Using~\eqref{lem:fxgoS} we obtain
\[
\left(\Upsilon^{(3)}\right)^{-1}\circ \breve{\Upsilon}\circ\Upsilon^{(3)}(u\otimes v)=\sum 
\lag (v_2\circ S)\otimes u_2,\EuScript R^{(n)}\rag u_1\otimes v_1
=\sum \lag u_2\otimes v_2,\underline{\EuScript R}^{(n)}\rag u_1\otimes v_1.
\]
Lemma~\ref{lem:fog,Rexpand}(i) implies that 
$
\left(\Upsilon^{(3)}\right)^{-1}\circ \breve{\Upsilon}\circ\Upsilon^{(3)}(u\otimes v)=\underline{\EuScript R}^{(n)}(u\otimes v)
$.
Since the map $u\otimes v\mapsto
\underline{\EuScript R}^{(n)}(u\otimes v)$ is an injection,  $\Upsilon^\iota$ is also an injection. 
\end{proof}

\begin{cor}
\label{cor:star-surj}
Let $\Emb:=\Emb_{k\times l}^{n\times n}$ be as in~\eqref{eq:emm-restricted} and let 
$\mathsf m_{k\times l}:\sP_{k\times l}\otimes\sP_{k\times l}\to\sP_{k\times l}$ be the usual product of $\sP_{k\times l}$.
Then the following statements hold. 
\begin{itemize}
\item[\rm(i)] $u\star_{k,l,n}^{} v=
\mathsf m_{k\times l}\circ (\Emb^{-1}\otimes \Emb^{-1})\circ\Upsilon(\Emb(u)\otimes \Emb(v)) $
for $u,v\in \sP_{k\times l}$. 

\item[\rm (ii)]
For $r,s\geq 0$ the map $\sP_{k\times l}^{(r)}\otimes \sP_{k\times l}^{(s)}\to \sP_{k\times l}^{(r+s)}$ given by $u\otimes v\mapsto u\star_{k,l,n}^{} v$ is surjective.
\end{itemize}
\end{cor}
\begin{proof}
(i) By Proposition~\ref{prp:compatib-Ups} we have $\Upsilon(\Emb(u)\otimes \Emb(v))\in(\Emb\otimes \Emb)(\sP_{k\times l}\otimes \sP_{k\times l})$. The assertion follows from the relation $\mathsf m_{k\times l}=\Emb^{-1}\circ \mathsf m_{n\times n}\circ (\Emb\otimes \Emb)$ and Lemma~\ref{lem:star=moUps}(ii).  

(ii)
By Proposition~\ref{prp:Ups-is-bij} and 
Proposition~\ref{prp:compatib-Ups} the map
\[
\Upsilon_{k,l,n}:=(\Emb^{-1}\otimes \Emb^{-1})\circ \Upsilon\circ (\Emb\otimes \Emb)
\] is a linear bijection on $\sP_{k\times l}^{(r)}\otimes \sP_{k\times l}^{(s)}$. From (i) it follows that $u\star_{k,l,n}^{} v=\mathsf m_{k\times l}(\Upsilon_{k,l,n}(u\otimes v))$. 
The latter equality reduces the assertion to surjectivity of $\mathsf m_{k\times l}:\sP_{k\times l}^{(r)}\otimes \sP_{k\times l}^{(s)}\to \sP_{k\times l}^{(r+s)}$, which is a trivial statement. 
\end{proof}

\section{Proofs of Theorem~\ref{thm-Main-A} and Theorem~\ref{thm:C}}
\label{subsec"ThmAm=n}
We begin by describing our strategy for proving Theorems~\ref{thm-Main-A}
and~\ref{thm:C}.
Lemma~\ref{lem:etam,n}  implies  that Theorem~\ref{thm-Main-A}(ii) follows by symmetry from Theorem~\ref{thm-Main-A}(i). 
Furthermore,  Lemma~\ref{prp:u-EiFiKi}
implies that 
 Theorem~\ref{thm-Main-A}(i)  is the special case of Theorem~\ref{thm:C} for $k=l=m$. Thus, it suffices to prove Theorem~\ref{thm:C}. 
 
We now give an outline of the proof of Theorem~\ref{thm:C}. By Corollary~\ref{cor:Agrkln-vs-gr(A)} we have a $U_R$-equivariant isomorphism of $\Bbbk$-algebras $\mathrm{gr}(\mathscr A_{k,l,n})\cong \mathscr A_{k,l,n}^\mathrm{gr}$. 
Recall that by definition, $\mathscr A_{k,l,n}^\mathrm{gr}$ is a subalgebra of 
$\sPD^\mathrm{gr}:=\sPD^\mathrm{gr}_{m\times n}$ where $m:=\max\{k,l\}$. 
 Theorem~\ref{thm:C} for $\mathrm{gr}(\mathscr A_{k,l,n})$ is equivalent to the following assertion for  $\mathscr A^\mathrm{gr}_{k,l,n}$.

\begin{thmy}
\label{thm:B'}
The algebra $\left(\mathscr A_{k,l,n}^\mathrm{gr}\right)_{(\epsilon_R)}$ is generated by 
the
 $\tilde{\mathsf L}_{i,j}^\mathrm{gr}$
for
$1\leq i\leq k$ and $1\leq j\leq l$, where
 $\tilde{\mathsf L}_{i,j}^\mathrm{gr}$ is defined as in
\eqref{eq:tildjf;lkjf}.
\end{thmy}  
Let 
 $\mathcal B\sseq \mathscr A_{k,l,n}^\mathrm{gr}$ denote the subalgebra generated by the $\tilde{ \mathsf L}_{i,j}^\mathrm{gr}$ for $1\leq i\leq k$ and $1\leq j\leq l$. 
Lemma~\ref{lem:LijRijLLRR} implies that $\mathcal B\sseq \big(\mathscr A_{k,l,n}^\mathrm{gr}\big)_{(\epsilon_R)}$. 
Let $\mathscr A_{k,l,n}^{\mathrm{gr},(r,s)}$ be defined as in~\eqref{eq:Aklnrs}. 
 Since by~\eqref{eq:grPDrrUR22}
 we have $\left(\mathscr A^{\mathrm{gr},(r,s)}_{k,l,n}\right)_{(\epsilon_R)}=0$ for $r\neq s$,  
to prove Theorem~\ref{thm:B'} it suffices to verify that 
\begin{equation}
\label{eq:sPDsseqxdi}
\left(\mathscr A_{k,l,n}^{\mathrm{gr},(r,r)}\right)_{(\epsilon_R)}\sseq \mathcal B\quad\text{for }r\geq 0.
\end{equation}
 We will verify~\eqref{eq:sPDsseqxdi} in the case
$n\geq \max\{k,l\}$
in Subsection~\ref{subsec:specialm=n}
  and 
in the case $n<\max\{k,l\}$
in Subsection~\ref{subsec-prep-A(ii)}.
This completes the proof of Theorem~\ref{thm:B'} in both cases.
Then in Subsection~\ref{subsec:thmCforAkln} we
 reduce Theorem~\ref{thm:C} for $\mathscr A_{k,l,n}$ to Theorem~\ref{thm:B'}.

\subsection{Proof of Theorem~\ref{thm:B'} when $n\geq \max\{k,l\}$ }
\label{subsec:specialm=n}
We prove by induction on $r$
that
\begin{equation*}
\left(\mathscr A_{k,l,n}^{\mathrm{gr},(r,r)}\right)_{(\epsilon_R)}\sseq \mathcal B_r\quad\text{for }r\geq 0,
\end{equation*} 
where
\[
\mathcal B_r:=\spn_\Bbbk\left\{\tilde{\mathsf L}^\mathrm{gr}_{i_1,j_1}\cdots \tilde{\mathsf L}^\mathrm{gr}_{i_r,j_r}\,:\,1\leq i_1,\ldots,i_r\leq k\ ,\ 
1\leq j_1,\ldots,j_r\leq l\right\}
.
\]
For $r=0$ the assertion is trivial. For $r=1$,
 from Lemma~\ref{lem:PsiSURJ} it follows that
$\left(\mathscr A_{k,l,n}^{\mathrm{gr},(1,1)}\right)_{(\epsilon_R)}$ is spanned by the $\tilde{\mathsf L}^\mathrm{gr}_{i,j}=\Gamma_{k,l,n}(\tilde t_{i,j})$ for $1\leq i\leq k$ and $1\leq j\leq l$, where $\tilde{t}_{i,j}$ is defined as in~\eqref{eq:tildetd} for $a:=k$ and $b:=l$. 
Finally, assume $r>1$ and 
choose any $D\in \left(\mathscr A_{k,l,n}^{(r,r)}\right)_{(\epsilon_R)}$. By  Lemma~\ref{lem:PsiSURJ} we have  
 $D={{}\Gamma}_{k,l,n}(u)$ for some
$u\in\sP_{k\times l}^{(r)}$.
By Corollary~\ref{cor:star-surj}(ii) the linear map \[
\sP_{k\times l}^{(1)}\otimes \sP_{k\times l}^{(r-1)}\to\sP_{k\times l}^{(r)}\ ,\ 
u\otimes v\mapsto u\star_{k,l,n}^{} v
\] is a surjection. Thus, we can express $u$ as a sum of products of the form $u'\star_{k,l,n}^{} u''$ where $u'\in\sP_{k\times l}^{(1)}$ and 
$u''\in\sP_{k\times l}^{(r-1)}$. By Proposition~\ref{prp:PsiuPsiv-Puv}(i),
\[
{{}\Gamma}_{k,l,n}(u'\star_{k,l,n}^{} u'')={{}\Gamma}_{k,l,n}(u')\,{{}\Gamma}_{k,l,n}(u'').
\]
From Lemma~\ref{lem:PsiSURJ}  and the induction hypothesis
it follows that   
$\Gamma_{k,l,n}(u')\in\mathcal B_1$ and 
$\Gamma_{k,l,n}(u'')
\in\mathcal B_{r-1}$.
Consequently, $D={{}\Gamma}_{k,l,n}(u)=\sum {{}\Gamma}_{k,l,n}(u'\star_{k,l,n}^{} u'')
=\sum{{}\Gamma}_{k,l,n}(u')\,{{}\Gamma}_{k,l,n}(u'')
\in\mathcal B_r$.

\subsection{Proof of Theorem~\ref{thm:B'}  when $n< \max\{k,l\}$ }
\label{subsec-prep-A(ii)}
  
Set $\underline k:=\min\{k,n\}$ and $\underline l:=\min\{l,n\}$. 
We use  a reduction to  
Theorem~\ref{thm:B'}
for the case of $\mathscr A_{\underline k,\underline l,n}^\mathrm{gr}$, which follows from Subsection~\ref{subsec:specialm=n}. This technique is also  used in \cite{LZZ11}. However, 
the arguments of~\cite{LZZ11} do \emph{not} extend routinely to the present setting. The reason is that unlike~\cite{LZZ11}, the 
products of the generators $\tilde {\mathsf L}_{i,j}^\mathrm{gr}$ are not weight vectors for the Cartan subalgebras of $U_q(\gl_k)$ and $U_q(\gl_l)$. 
As explained below, in order to circumvent this technical difficulty  we use Proposition~\ref{prp:AotBUrUR}. 

Set $\underline m:=\max\{\underline k,\underline l\}$ and $m:=\max\{k,l\}$. Recall that by definition, 
$\mathscr A_{k,l,n}^\mathrm{gr}$ is 
a subalgebra of $\sPD^\mathrm{gr}=\sPD_{m\times n}^\mathrm{gr}$ and
$\mathscr A_{\underline k,\underline l,n}^\mathrm{gr}$ is a subalgebra of 
$\sPD_{\underline m\times n}^\mathrm{gr}$.  
Let \[
\Emb^\mathrm{gr}:=(\Emb^\mathrm{gr})_{\underline m\times n}^{m\times n}:\sPD_{\underline m\times n}^\mathrm{gr}\to\sPD_{m\times n}^\mathrm{gr}
\] be the map defined  in Proposition \ref{prp:existence(6)}(i).
By checking the images of generators of $\mathscr A_{\underline k,\underline l,n}^\mathrm{gr}$ we obtain
\[
\Emb^\mathrm{gr}\left(\mathscr A_{\underline k,\underline l,n}^\mathrm{gr}\right)\sseq \mathscr A_{k,l,n}^\mathrm{gr}
.\]
\begin{lem}
$\Emb^\mathrm{gr}\left(
\left(\mathscr A_{\underline k,\underline l,n}^\mathrm{gr}\right)_{(\epsilon_R)}
\right)
=
\left(\mathscr A_{k,l,n}^\mathrm{gr}\right)_{(\epsilon_R)}
\cap
\Emb^\mathrm{gr}\left(\mathscr A_{\underline k,\underline l,n}^\mathrm{gr}\right) $.
\end{lem}  
\begin{proof}
This follows from $U_R$-equivariance of the map $\Emb^\mathrm{gr}$ (see Proposition~\ref{prp:existence(6)}).
\end{proof}
Recall that $\sPD^\mathrm{gr}$ is a
module over  $U_{LR}\otimes U_{LR}=U_L\otimes U_R\otimes U_L\otimes U_R$. Let $U_L^{(k,l)}$ be the subalgebra of $U_{LR}\otimes U_{LR}$ defined by 
\[
U_L^{(k,l)}:=\boldsymbol\kappa_{k,m}(U_q(\gl_k))\otimes 1\otimes \boldsymbol\kappa_{l,m}(U_q(\gl_l))\otimes 1.
\]

\begin{prp} 
\label{prp:ULkl}
The $U_L^{(k,l)}$-submodule of $\sPD^\mathrm{gr}$ that is generated by  
$
\Emb^\mathrm{gr}\left(\left(\mathscr A_{\underline k,\underline l,n}^\mathrm{gr}\right)_{(\epsilon_R)}\right) $ is equal to
$\left(\mathscr A_{k,l,n}^\mathrm{gr}\right)_{(\epsilon_R)}
$.
\end{prp}

\begin{proof}
Set
$d:=\min\{\underline{k},\underline{l}\}=\min\{k,l,n\}$.
First note that by~\eqref{eq:VlaVmu*} and~\eqref{eq:grPDrrUR22} we have  isomorphisms of $U_q(\gl_k)\otimes U_q(\gl_l)$-modules 
\begin{equation}
\label{eq:AepsRR1}
\left(
\mathscr A_{k,l,n}^{\mathrm{gr},(r,s)}
\right)_{(\epsilon_R)}
=0\text{\quad for $r\neq s$\quad and\quad }
\left(
\mathscr A_{k,l,n}^{\mathrm{gr},(r,r)}
\right)_{(\epsilon_R)}\cong
\bigoplus_{\la\in\Lambda_{d,r}}
V_\la^*\otimes  V_\la^{},
\end{equation}
where $V_\la^*$ (respectively, $V_\la^{}$) denotes an  irreducible $U_q(\gl_k)$-module (respectively,
$U_q(\gl_l)$-module).
Similarly, using
the equivariance of $\Emb^\mathrm{gr}$ from 
Proposition~\ref{prp:existence(6)}(ii) we obtain 
\[
\left(\mathscr A_{k,l,n}^{\mathrm{gr}}\right)_{(\epsilon_R)}
\cap
\Emb^\mathrm{gr}\left(\mathscr A_{\underline k,\underline l,n}^{\mathrm{gr},(r,s)}\right)
=0\quad\text{ for }r\neq s,
\] and 
an isomorphism of $U_q\left(\gl_{\underline k}\right)\otimes U_q\left(\gl_{\underline l}\right)$-modules \begin{equation}
\label{eq:AepsRR2}
\left(\mathscr A_{k,l,n}^{\mathrm{gr}}\right)_{(\epsilon_R)}
\cap
\Emb^\mathrm{gr}\left(\mathscr A_{\underline k,\underline l,n}^{\mathrm{gr},(r,r)}\right) 
=
\Emb^\mathrm{gr}
\left(
\left(
\mathscr A_{\underline k,\underline l,n}^{\mathrm{gr},(r,r)}
\right)_{(\epsilon_R)}
\right)
\cong
\bigoplus_{\la\in\Lambda_{d,r}}
\bar{V}_\la^*\otimes  \bar V_\la^{},
\end{equation} 
where $\bar V_\la^*$ 
(respectively, $\bar V_\mu^{}$) is an  irreducible $U_q\left(\gl_{\underline k}\right)$-module (respectively,
$U_q\left(\gl_{\underline l}\right)$-module). In the latter relation we use the bar on $\bar V_\la^*$ and $\bar V_\la^{}$ to distinguish $U_q(\gl_{\underline k})$-modules from $U_q(\gl_k)$-modules and
$U_q(\gl_{\underline l})$-modules from $U_q(\gl_l)$-modules.
To complete the proof, we need to  verify that the summand 
$\bar V_\la^*\otimes \bar V_\la^{}$ 
of~\eqref{eq:AepsRR2} generates the summand $V_\la^*\otimes V_\la^{}$ of ~\eqref{eq:AepsRR1} as a $U_q(\gl_k)\otimes U_q(\gl_l)$-module.
In what follows, we prove the latter assertion.
  Let $v_\la^*$ be a highest weight vector of 
the $U_q(\gl_k)$-module $V_\la^*$ and let 
$v_\la^{}$ be a lowest weight vector of 
the $U_q(\gl_l)$-module
$V_\la^{}$. It suffices to prove that \begin{equation}
\label{eq:vbarV-ajld}
v_\la^*\otimes v_\la^{}\in \bar V_\la^*\otimes \bar V_\la.
\end{equation} 
The weight of $v_\la^{}$ with respect to the standard Cartan subalgebra of $U_q(\gl_l)$ is 
obtained by applying the longest element of the Weyl group $S_l$ to the coefficients of $q^{\sum_{i=1}^d\lambda_i\eps_i}$ (which is the highest weight of $\bar V_\lambda$). hence the weight of $v_\lambda$  is 
$q^{\sum_{i=1}^d\eps_{l-i+1}\la_i}$. By a similar reasoning,  the 
weight of $v_\la^*$ with respect to the standard Cartan subalgebra of $U_q(\gl_k)$ is 
$q^{-\sum_{i=1}^d\eps_{k-i+1}\la_i}$.
Since $k-\underline k\leq k-d$, we have $K_{\eps_i}\cdot v_\la^*=v_\la^*$ for $K_{\eps_i}\in U_q(\gl_k)$ satisfying $1\leq i\leq k-\underline k$. Similarly, from $l-\underline l\leq l-d$ it follows that $K_{\eps_i}\cdot v_\la^{}=v_\la^{}$
for $K_{\eps_i}\in U_q(\gl_l)$ satisfying $1\leq i\leq l-\underline l$. Next we express $v_\la^*\otimes v_\la^{}$ as a linear combination of  the basis of $\sPD^\mathrm{gr}$ that consists of the monomials~\eqref{eq:bassis} (see Proposition~\ref{diamond-app-gr}). Since $v_\la^*\otimes v_\la^{}\in\mathscr A_{k,l,n}^\mathrm{gr}$, the monomials that occur 
must satisfy 
\begin{equation}
\label{monom-1}
a_{i,r}=b_{j,r}=0\quad\text{for }
1\leq i\leq m-k,\ 1\leq j\leq m-l,\ 1\leq r\leq n.
\end{equation}
By Remark~\ref{rmk:actionformulas}
each of the occurring monomials is a joint eigenvector for the action of \[
\boldsymbol\kappa_{k,m}(K_{\eps_i})\otimes 1\otimes \boldsymbol\kappa_{l,m}(K_{-\eps_j})\otimes 1\quad\text{ where }1\leq i\leq k\text{ and }1\leq j\leq l,
\]
with eigenvalue \[q^{-\sum_{r=1}^n (a_{m-k+i,r}+b_{m-l+j,r})}.
\] 
If this eigenvalue is 1 for $i\leq k-\underline k$ and $j\leq l-\underline l$, then 
we must have 
\begin{equation}
\label{monom-2}
a_{m-k+i,r}=b_{m-l+j,r}=0\quad\text{ for }1\leq i\leq k-\underline k,\ 1\leq j\leq l-\underline l,\ 1\leq r\leq n.
\end{equation}
 From~\eqref{monom-1} and~\eqref{monom-2} (and the general fact that joint eigenfunctions with distinct eigenvalues are linearly independent) it follows that all of the occurring monomials belong to $\Emb^\mathrm{gr}(\mathscr A_{\underline k,\underline l,n}^\mathrm{gr})$. 
Consequently, $v_\la^*\otimes v_\la^{}$ belongs to the left hand side of~\eqref{eq:AepsRR2}. In addition,  $v_\la^*\otimes v_\la^{}$ is the tensor product of a  lowest weight vector for a $U_q(\gl_{\underline k})$-module isomorphic to $\bar V_\la^*$ and a highest weight vector for a
$U_q(\gl_{\underline l})$-module isomorphic to $\bar V_\la^{}$. From the decomposition of the right hand side 
of~\eqref{eq:AepsRR2} we obtain
that $v_\la^*\in \bar V_\la^*$ and $v_\la^{}\in\bar V_\la^{}$. This completes the proof of~\eqref{eq:vbarV-ajld}.  
 \end{proof}

We are now ready to complete the proof of~\eqref{eq:sPDsseqxdi}.
From Theorem~\ref{thm:B'} for $\mathscr A_{\underline k,\underline l,n}^\mathrm{gr}$ (which is established in Subsection~\ref{subsec:specialm=n}) it follows that 
\[
\Emb^\mathrm{gr}\left(
\left(\mathscr A_{\underline k,\underline l,n}^\mathrm{gr}\right)_{(\epsilon_R)}
\right)\sseq \mathcal B.\]
Thus, by Proposition~\ref{prp:ULkl} it suffices to prove that $\mathcal B$ is stable under the action of $U_L^{(k,l)}$. 
The key idea to prove this  is that the span of the generators of $\cB$ is stable under an algebra larger than $U_L^{(k,l)}$. This enables us to use 
Proposition~\ref{prp:AotBUrUR}. 
Let $\tilde U_L^{(k,l)}$ be the subalgebra of $U_{LR}\otimes U_{LR}=U_{L}\otimes U_R\otimes U_L\otimes U_{R}$  defined as follows:
\begin{itemize}
\item[(i)]
If $k\leq l$, then $\tilde U_L^{(k,l)}:=\oline U_{k,l}\otimes 1\otimes U_L\otimes 1$ where $\oline U_{k,l}$ is the subalgebra of $U_L= U_q(\gl_l)$ that is generated by \[
\{E_i\}_{i=1}^{l-1}\cup\{F_i\}_{i=l-k+1}^{l-1}\cup \{K_{\eps_i}\}_{i=1}^{l}.
\]
\item[(ii)]
If $k>l$, then 
$\tilde U_L^{(k,l)}:=U_L\otimes 1\otimes \oline U_{k,l}\otimes 1$ where 
$\oline U_{k,l}$ is the subalgebra of $U_L= U_q(\gl_k)$ that is generated by \[
\{E_i\}_{i=k-l+1}^{k-1}\cup\{F_i\}_{i=1}^{k-1}\cup \{K_{\eps_i}\}_{i=1}^{k}.
\]
\end{itemize}
Note that in both cases we have $U_L^{(k,l)}\sseq \tilde U_L^{(k,l)}$.
\begin{prp}
\label{prp:Holine-stab}
$\mathcal B$ is stable under the action of $\tilde U_L^{(k,l)}$. 
\end{prp}

\begin{proof}
This follows from Proposition~\ref{prp:AotBUrUR} by setting $H:=U_L$, $\cC:=\cC^{(m)}$ where $m:=\max\{k,l\}$, $\check R:=\check{\underline{\EuScript R}}_L$, $H':=U_R$, $\cC':=\cC^{(n)}$, $\check R':=\check{\underline{\EuScript R}}_R$, $A:=\sP$, $B:=\sD$, $\oline H:=\oline U_{k,l}$ 
and \[
\mathcal E:=\spn_\Bbbk\left\{\tilde{\mathsf L}_{i,j}^\mathrm{gr}\,:\,1\leq i\leq k\text{ and }1\leq j\leq l\right\}.
\]
Checking that  $\mathcal E$ 
is stable under the action of ${\tilde U}_L^{(k,l)}$ is a direct calculation based on Remark~\ref{rmk:actionformulas}. Also, 
according to Proposition~\ref{prp:Uqglnlocfin}
we can choose  $\omega_{V,W},\oline \omega_{V,W}\in H\otimes H$ satisfying the condition of  Definition~\ref{dfn:braidedfinite}(ii)
to be finite linear combinations of $2$-tensors of the form
\[
K F_\beta F_{\beta'}F_{\beta''} 
\cdots \otimes K' 
E_\beta E_{\beta'}E_{\beta''}
\cdots,
\]
where $\beta,\beta',\beta'',\ldots$ are positive roots (see Definition~\ref{dfn:univRmatgln}) and $K,K'$ are in the Cartan subalgebra. Verifying that the assumptions of Proposition~\ref{prp:AotBUrUR}(i) and
Proposition~\ref{prp:AotBUrUR}(ii)
 on $\omega_{V,W}$ and $\oline\omega_{V,W}$ hold is then a direct calculation based on the formulas that express the root vectors as commutators of the $E_i$ and the $F_i$ (see   Definition~\ref{dfn:univRmatgln}).
\end{proof}

\subsection{Proof of Theorem~\ref{thm:C} for $\mathscr A_{k,l,n}$  }
\label{subsec:thmCforAkln}
In this subsection we deduce 
Theorem~\ref{thm:C} for $\mathscr A_{k,l,n}$ from 
Theorem~\ref{thm:B'}. Let $\mathbb K$ be any field. 
As usual
a  $\mathbb K$-algebra $\cA$ is called \emph{filtered} if it has a filtration 
\[
\cA^0\sseq \cA^1\sseq \cA^2\sseq\cdots 
\] such that $\cA^i\cA^j\sseq \cA^{i+j}$ for $i,j\geq 0$. We assume that filtered algebras always satisfy $\cA^0=\mathbb K$. 
As usual $\mathrm{gr}(\cA):=\bigoplus_{i=-1}^\infty \cA^{i+1}/\cA^i$ denotes the associated graded algebra of $\cA$, where by convention $\cA^{-1}=0$. 
The following general lemma is standard and can be proved by induction.  
\begin{lem}\label{lem:alggrfilt}
Let $\cA$ be a filtered algebra and let $a_1,\ldots,a_r\in\cA^1$ be such that their images in
$\cA^1/\cA^0$ generate $\mathrm{gr}(\cA)$. Then $a_1,\ldots,a_r$ generate $\cA$. 
\end{lem}
The passage from Theorem~\ref{thm:B'} to Theorem~\ref{thm:C} relies on the folowing proposition.
\begin{prp}
\label{prp:filtgrAF}
Let $\cA:=\bigoplus_{i=0}^\infty \cA^{(i)}$ be a graded $\mathbb K$-algebra  and 
let $\cB$ be a filtered $\mathbb K$-algebra. Set $\cA^i:=\bigoplus_{j=0}^i \cA^{(j)}$ for $r\geq 0$, so that $\mathbb K:=\cA^0\sseq \cA^1\sseq \cA^2\sseq\cdots$ is a filtration of $\cA$.
Let $\mathsf F:\cA\to\cB$ be a filtration-preserving linear map such that 
$\mathrm{gr}(\mathsf F):\mathrm{gr}(\cA)\to \mathrm{gr}(\cB)$ is an isomorphism of algebras. 
Suppose that 
$a_1,\ldots,a_r\in\cA^{(1)}$ generate $\cA$. Then $\mathsf F(a_1),\ldots, \mathsf F(a_r)$ generate $\cB$. 
\end{prp}

\begin{proof}
Set $\oline{\mathsf F}:=\mathrm{gr}(\mathsf F)$. Since $\cA\cong\mathrm{gr}(\cA)$, we can consider $\oline {\mathsf F}$ as a map $\cA\to \mathrm{gr}(\cB)$. Set $b_i:=\mathsf F(a_i)$ 
for $1\leq i\leq r$. Then $b_i+\cB^0=\oline{\mathsf F}(a_i)$, hence $b_1+\cB^0,\ldots,b_r+\cB^0$ generate $\mathrm{gr}(\cB)$. Thus  Lemma~\ref{lem:alggrfilt} implies that $b_1,\ldots,b_r$ generate $\cB$. 
\end{proof}

We return to the proof of Theorem~\ref{thm:C} for $\cA_{k,l,n}$. We verify that the assumptions of Proposition~\ref{prp:filtgrAF} hold for $\cA:=\left(\mathscr A_{k,l,n}^\mathrm{gr}\right)_{(\epsilon_R)}$, $\cB:=
\left(\mathscr A_{k,l,n}\right)_{(\epsilon_R)}$ and $\mathsf F:=\mathsf P_{k,l,n}$. Since $\mathsf P_{k,l,n}:\mathscr A_{k,l,n}^\mathrm{gr}\to\mathscr A_{k,l,n}$ is an isomorphism of $U_R$-modules, we have
$\mathsf F(\cA)=\cB$.
For $r\geq 0$ set 
\begin{equation}
\label{eq:A(r)}
\cA^{(r)}:=\mathscr A_{k,l,n}^{\mathrm{gr},(r,r)}\cap \cA
\quad\ ,\  
\cA^r:=\bigoplus_{s=0}^r \cA^{(s)}
\ ,\ 
\tilde\cB^{(r)}:=
\mathsf P_{k,l,n}\left(\mathscr A_{k,l,n}^{\mathrm{gr},(r,r)}\right)\ \text{and}\  
\tilde{\cB}^{r}:=\bigoplus_{s=0}^{r}\tilde{\cB}^{(s)}
.
\end{equation}
Since the $U_R$-action on $\mathscr A_{k,l,n}^\mathrm{gr}$ leaves the subspaces $\mathscr A_{k,l,n}^{\mathrm{gr},(r,s)}$ stable,  we have $\cA=\bigoplus_{r=0}^\infty \cA^{(r)}$. 
Define a filtration on $\cB$ by setting $\cB^r:=\mathsf F(\cA^r)$ for $r\geq 0$. 
Since we also have $\cB=\mathsf F(\cA)$, the map  $\mathrm{gr}(\mathsf F):\mathrm{gr}(\cA)\to \mathrm{gr}(\cB)$ is an isomorphism of graded vector spaces.  Next we prove that the latter map is an isomorphism of algebras. To this end, it suffices to verify that 
\begin{equation}
\label{eq:FDFD'=}{\mathsf F}(D)
{\mathsf F}(D')-\mathsf F(DD')\in\cB^{i+j-1}\quad\text{for }
D\in \cA^{(i)}\text{ and }D'\in\cA^{(j)}.
\end{equation}
By Proposition~\ref{rmk:PtildeRvsPD} the left hand side of~\eqref{eq:FDFD'=} belongs to $\tilde{\cB}^{i+j-1}$. Since $\cA_{k,l,n}$ is a $U_R$-module algebra and $\mathsf F$ is a $U_R$-module homomorphism,  we have $\mathsf F(D),\mathsf F(D'),\mathsf F(DD')\in\cB$. It follows that the left hand side of~\eqref{eq:FDFD'=}  also belongs to $\cB$. But since the map $\mathsf P_{k,l,n}:\mathscr A_{k,l,n}^\mathrm{gr}\to \mathscr A_{k,l,n}$ is a bijection,
\[
\tilde{\cB}^{r}\cap \cB=\mathsf P_{k,l,n}\left(\mathscr A_{k,l,n}^{\mathrm{gr},(r,r)}\right)\cap\mathsf P_{k,l,n}\left(\cA\right)=
\mathsf P_{k,l,n}\left(
\mathscr A_{k,l,n}^{\mathrm{gr},(r,r)}\cap\cA
\right)
\subseteq \cB^{r}\quad\text{for }r\geq 0.
\]
For $r=i+j-1$ this implies the inclusion~\eqref{eq:FDFD'=}. Thus we have proved that
the assumptions of 
Proposition~\ref{prp:filtgrAF}
hold for $\cA$, $\cB$ and $\mathsf F:\cA\to\cB$  chosen as above.  
By Theorem~\ref{thm:B'} the $\tilde{\mathsf L}^\mathrm{gr}_{i,j}$ for $1\leq i\leq k$ and $1\leq j\leq l$ generate $\cA$, hence by Proposition~\ref{prp:filtgrAF}
 the 
 $\tilde{\mathsf L}_{i,j}$  generate $\cB$.  This completes the proof of Theorem~\ref{thm:C} for $\mathscr A_{k,l,n}$.

\section{Proof of Theorem~\ref{thm:MainthmB}}
\label{sec:genofLhbRhb}

In this section we give the proof of 
Theorem~\ref{thm:MainthmB}(i). The proof of 
Theorem~\ref{thm:MainthmB}(ii) is analogous.
As a byproduct, 
in  Corollary~\ref{cor:IZie}
we obtain explicit generators for $\phi_U^{-1}(\Lhb)$ and 
$\phi_U^{-1}(\Rhb)$. 
Henceforth we use $E_i$, $F_i$ and the $K_\la$ for $\la\in\Z\eps_1+\cdots +\Z\eps_n$ to denote elements of  $U_R=U_q(\gl_n)$.

\subsection{Parity condition on the $\la$}
\label{subsec:partial-ord}
For $\la,\mu\in\Z\eps_1+\cdots+\Z\eps_n$ 
expressed as $\la:=\sum_{i=1}^n\la_i\eps_i$ and 
$\mu:=\sum_{i=1}^n\mu_i\eps_i$
we define $\lag \la,\mu\rag$ as in~\eqref{eq:pairingg}.
We also set $\la<\mu $ if there exists $1\leq r< n$ such that   $\la_i=\mu_i$ for all $i\leq r$ and $\la_{r+1}<\mu_{r+1}$. This defines a total order on $\Z\eps_1+\cdots +\Z\eps_n$. The following lemma is trivial. 
\begin{lem}
\label{lem:10.0.1}
Let $\mathcal S$ be a finite subset of $\Z\eps_1+\cdots +\Z\eps_n$
and let $\la_{\max}$ denote the maximum of $\mathcal S$ with respect to $<$. Let $\gamma_1,\ldots,\gamma_n\in\Z$ be such that $\gamma_n\geq 1$ and $\gamma_i\geq 1+\max_{\la,\mu\in\mathcal S}\left\{
\sum_{i<j\leq n}|\la_j-\mu_j|\gamma_j\right\}$ for $i<n$. Set $\gamma:=\sum_{i=1}^n \gamma_i\eps_i$. Then   
$\lag \la_{\max},\gamma\rag>\lag\mu,\gamma\rag$ for all $\mu\in\mathcal S$ such that
$\mu\neq\la_{\max}$. 
\end{lem}

\begin{prp}
\label{prp:xinPD1}
Let $\cI$ be a finite subset of $\Z\eps_1+\cdots+\Z\eps_n$. Let   $x:=\sum_{\la\in\cI} c_\la K_\la\in U_{\g h,R}$ where $c_\la\in\Bbbk^\times$ for $\la\in\cI$, and  assume that $x\in\URw$. Then for all $\la:=\sum \la_i\eps_i\in\cI$ 
and $1\leq i\leq n-1$
we have $\la_i-\la_{i+1}\in 2\Z^{\geq 0}$.
\end{prp}

\begin{proof}
\textbf{Step 1.}
Set $D:=\phi_U(x)$. By~\eqref{eq:phiUadYY} we have  $\phi_U(\ad_y(x))=(1\otimes y)\cdot D$ for $y\in U_R$. Since $\sPD$ is a locally finite $U_R$-module,  $\phi_U(\ad_{U_R}(x))$ is a finite dimensional subspace of $\sPD$. Furthermore for every $f\in \sP$, if we set 
$W_f:=\ad_{U_R}(x)\cdot f:=\{\ad_y(x)\cdot f\,:\,y\in U_R\}$,
then $
\dim W_f\leq d_\circ
$
where
$d_\circ:=\dim (\phi_U(\ad_{U_R}(x)))
$.   
Note that the  upper bound $d_\circ$ on $\dim W_f$ is independent of $f$.

\textbf{Step 2.}
Fix $\alpha_i:=\eps_i-\eps_{i+1}$
where
$1\leq i\leq n-1$. It suffices to prove that $\lag \la,\alpha_i\rag\in2\Z^{\geq 0}$ for $\la\in\cI$. 
For $r\geq 1$ we have \[
\ad_{E_i^r}K_\la=\prod_{j=0}^{r-1}\left(1-q^{\lag \la,\alpha_i\rag-2j}\right)E_i^rK_{\la}K_i^{-r}.
\]
Now take a nonzero $U_{\g h,R}$-weight vector $f\in\sP$ of weight $q^{-\gamma}$ for $\gamma:=\sum_{i=1}^n \gamma_i\eps_i$, where $(\gamma_1,\ldots,\gamma_n)$ is an $n$-tuple of non-negative integers. 
We have  
\begin{align}
\label{eq:adE--i}
\notag
\ad_{E_i^r}(x)\cdot f
&=
\left(\sum_{\la\in\cI}
c_\la \ad_{E_i^r}K_\la\right)
\cdot 
f\\
&=
q^{r\lag \gamma,\alpha_i\rag}
\left(
\sum_{\la\in\cI}
c_\la
q^{-\lag \la,\gamma\rag } 
\prod_{j=0}^{r-1}\left(1-q^{\lag \la,\alpha_i\rag-2j}\right)
\right) E_i^r\cdot f.
\end{align}

\textbf{Step 3.}
For any $\la\in\cI$, if $\lag\la,\alpha_i\rag\in2\Z^{\geq 0}$ then $\prod_{j=0}^{r-1}\left(1-q^{\lag \la,\alpha_i\rag-2j}\right)=0
$ for all sufficiently large~$r$. Thus, if we set  $\cI':=\left\{\la\in \cI\,:\,\lag\la,\alpha_i\rag\not\in 2\Z^{\geq 0}\right\}$ then there exists $r_\circ=r_\circ(\cI)$ such that for all $r\geq r_\circ$ we have 
\begin{equation}
\label{eq:clhlfh;s}
\sum_{\la\in\cI}
c_\la
q^{-\lag \la,\gamma\rag } 
\prod_{j=0}^{r-1}\left(1-q^{\lag \la,\alpha_i\rag-2j}\right)
=
\sum_{\la\in\cI'}
c_\la
q^{-\lag \la,\gamma\rag } 
\prod_{j=0}^{r-1}\left(1-q^{\lag \la,\alpha_i\rag-2j}\right).
\end{equation}
Note that the lower bound $r_\circ$ is independent of $\gamma$. 

\textbf{Step 4.}
Assume that  $\cI'\neq\varnothing$. Choose $r_\circ\in\N$ according to Step 3. Without loss of generality we can also assume that $r_\circ\geq d_\circ$.  Next choose $r\geq r_\circ$. 
After possibly scaling $x$ by a nonzero element of the polynomial ring $\C[q]$ we can assume that the $c_\la$ are nonzero elements in  $\C[q]$. Let $\la_{\max}$ denote the maximum of $\cI'$ with respect to $<$.   
Choose $\gamma$ as in Lemma~\ref{lem:10.0.1} (with $\mathcal S:=\cI'$). Since the condition on the coefficient $\gamma_i$ only depends on $\gamma_j$ for $j>i$, we can also assume that $\gamma_i-\gamma_{i+1}\geq 1$. 
For $\la\in\cI'$ let $q^{N(r,\la)}$ be the lowest power of $q$ that occurs after expanding and simplifying  
$c_\la
q^{-\lag \la,\gamma\rag } 
\prod_{j=0}^{r-1}\left(1-q^{\lag \la,\alpha_i\rag-2j}\right)
$. 
We have  \[
N(r,\la_{\max})\leq -\lag \la_{\max},\gamma\rag+\deg c_{\la_{\max}}(q),\] 
because the lowest power $q^{N(r,\la)}$ is obtained as follows:  
from each factor 
$\left(1-q^{\lag \la,\alpha_i\rag-2j}\right)$ we can choose $1$ if 
${\lag \la,\alpha_i\rag-2j}>0$ and $q^{\lag \la,\alpha_i\rag-2j}$ otherwise.
For all other $\la\in\cI'$ we have 
\[
N(r,\la)\geq -\deg c_\la(q^{-1})-\lag \la,\gamma\rag-r|\lag \la,\alpha_i\rag|-r(r-1).
\]
By the choice of $\gamma$, for 
$\la\in\cI'$ such that $\la\neq\la_{\max}$ we have $\lag\la_{\max},\gamma\rag\geq 1+\lag \la,\gamma\rag$. Thus 
\[
\lag\la_{\max},k\gamma\rag\geq k+\lag \la,k\gamma\rag
\quad\text{
for }k\in\N.
\]
Next choose $k\in\N$ such that $k\geq 2r_\circ$ and 
 \begin{equation}
 \label{eq:k>=}k\geq \max_{\la\in\cI',\la\neq\la_{\max}}\big\{\deg c_{\la_{\max}}(q)+\deg c_\la(q^{-1})+2r_\circ|\lag \la,\alpha_i\rag|+2r_\circ(2r_\circ-1)\big\}.
 \end{equation}  If we substitute 
 $\gamma$ by $k\gamma$, 
 from~\eqref{eq:k>=} we obtain
 that 
$N(r,\la_{\max})<N(r,\la)$ for all $\la\in \cI'\bls\{\la_{\max}\}$ and $r_\circ\leq  r\leq 2r_\circ$.  Together with Step 3, this proves that for the latter choice of $\gamma$ we have 
\[
\sum_{\la\in\cI}
c_\la
q^{-\lag \la,\gamma\rag } 
\prod_{j=0}^{r-1}\left(1-q^{\lag \la,\alpha_i\rag-2j}\right)\neq 0\quad\text{for }r_\circ\leq r\leq 2r_\circ,
\]
because the coefficient of $q^{N(r,\la_{\max})}$ is nonzero. 

\textbf{Step 5.}  The $\gamma$ chosen at the end of  Step 4 satisfies $\gamma_i-\gamma_{i+1}\geq 2r_\circ$, or equivalently 
$\lag -\gamma,\alpha_i\rag\leq -2r_\circ$ (because $k\geq 2r_\circ$). Choose $f\in \sP$ of $U_{\g h,R}$-weight $q^{-\gamma}$ (for example $f:=t_{1,1}^{\gamma_1}\cdots t_{1,n}^{\gamma_n}$). A standard argument based on representation theory of $U_q(\g{sl}_2)$ implies $E_i^r\cdot f\neq 0$ for $0\leq r\leq 2r_\circ$. 
Since the vectors $E_i^{s}\cdot f$ for $0\leq s\leq 2r_\circ$ have distinct $U_{\g h,R}$-weights, they are linearly independent. 
From Step 2 and Step 4 it follows that
the vectors $\ad_{E_i^{s}}(x)\cdot f$ for $r_\circ\leq s\leq 2r_\circ$ 
are also linearly independent. Consequently, $\dim W_f\geq r_\circ+1\geq d_\circ+1$.  This contradicts Step 1. 
\end{proof}

\subsection{Proof of $\la_1\leq 0$}
We begin with the following observation.
\begin{rmk}
\label{rmk:Deigven}
Let $D\in\sPD$ and let $\mathsf a:=(a_1,\ldots,a_n)$ be an $n$-tuple of non-negative integers. We use  the notation
$\del^\mathsf a:=\del_{1,1}^{a_1}\cdots \del_{1,n}^{a_n}$
and 
$t^\mathsf a:=t_{1,n}^{a_n}\cdots t_{1,1}^{a_1}$
for an $n$-tuple of integers $(a_1,\ldots,a_n)$. 
Assume that $D\cdot t^\mathsf a=c t^\mathsf a$ for some $c\in\Bbbk$. 
Recall the basis of $\sPD$ that consists of the monomials~\eqref{eq:basisII}.
We can write $D$ as $D=D_1+D_2+D_3$ where
\begin{itemize}
\item[(i)] $D_1$ is a linear combination of 
basis vectors  of the form $t^{\mathsf b'}\del^{\mathsf b'}$ where $\mathsf b'$ is an $n$-tuple of non-negative integers,  
\item[(ii)] $D_2$ is a linear combination of basis vectors  of the form $t^{\mathsf a'}\del^{\mathsf b'}$ where $\mathsf a'$ and $\mathsf b'$  are $n$-tuples of non-negative integers and $\mathsf a'\neq \mathsf b'$, and

\item[(iii)] $D_3$ is a linear combination of the remaining basis vectors in~\eqref{eq:basisII}. 
\end{itemize} 
Using Lemma~\ref{lem:delijonarbr} 
and then 
Lemma~\ref{lem:Dbta(i)-(iii)}
we obtain $D\cdot t^\mathsf a=(D_1+D_2)\cdot t^\mathsf a =D_1\cdot t^\mathsf a$. 

\end{rmk}

\begin{ex}
\label{ex:1}
Set $\la:=\eps_1+\cdots+\eps_n$ and  $x:=K_\la\in U_R$. Then $x\cdot t_{1,1}^r=q^{-r}t_{1,1}^r$ for $r\geq 1$. From  Remark~\ref{rmk:Deigven},
Lemma~\ref{lem:Dbta(i)-(iii)}
and Remark~\ref{rmk:cab}
it follows that if $\phi_U(1\otimes x)\in\Rnw$ then the eigenvalue of $t_{1,1}^r$ with respect to $\phi_U(1\otimes x)$ should be a ratio of two polynomials such as  $\phi_1(q)/\phi_2(q)$ where $\deg\phi_2$ is bounded above (independently of $r$). Thus,
$\phi_U(1\otimes x)\not\in\Rnw$ and 
in particular $\Rnw\subsetneq\Rn$. Consequently, $K_\la$ is a locally finite element of $U_R$ that does not belong to $\URw$.  \end{ex}

\begin{prp}
\label{prp:xinPD2}
Let $\cI$ be a finite subset of $\Z\eps_1+\cdots+\Z\eps_n$. Let   $x:=\sum_{\la\in\cI} c_\la K_\la\in U_{\g h,R}$ where $c_\la\in\Bbbk^\times$ for $\la\in\cI$, and  assume that $x\in\URw$. Then for all $\la:=\sum \la_i\eps_i\in\cI$ we have $\la_1\leq 0$.
\end{prp}

\begin{proof}
Set $D:=\phi_U(x)$, so that $D\in \sPD$. 
Write $D=D_1+D_2+D_3$ as in Remark~\ref{rmk:Deigven}
 and suppose that $D_1=\sum_{\mathsf a\in\mathcal Z}\mathbf z_{\mathsf a}t^\mathsf a\del^\mathsf a$ where $\mathcal Z$ is a finite set of $n$-tuples of non-negative integers and the $\mathbf z_{\mathsf a}\in\Bbbk^\times $.
After scaling $x$ by a nonzero element of $\C[q]$ if necessary, we can assume that the $c_\la$ and the $\mathbf z_{\mathsf a}$ are nonzero polynomials in $q$.
Recall that 
$t^\gamma:=t_{1,n}^{\gamma_n}\cdots t_{1,1}^{\gamma_1}$
for
 $\gamma:=\sum_{i=1}^n\gamma_i\eps_i$ in $\Z\eps_1+\cdots+\Z\eps_n$. 
Then \[D\cdot t^\gamma=\sum_{\la\in\cI} c_\la q^{-\lag \la,\gamma\rag}t^\gamma.
\]   Also, 
 by Lemma~\ref{lem:Dbta(i)-(iii)} 
 and Remark~\ref{rmk:Deigven} we obtain $D\cdot t^\gamma=D_1\cdot t^\gamma=\sum_{\mathsf a\in\mathcal Z}\mathbf z_{\mathsf a}(q)\phi_{\mathsf a}(q^2)t^\gamma$, where the $\phi_\mathsf a$ are polynomials in $q$ with integer coefficients. Note that the $\mathbf z_{\mathsf a}$ are  independent of $\gamma$, but the $\phi_\mathsf a$ can depend on $\gamma$. 

Set $\tilde\la:=\la_{\max}$ where $\la_{\max}$ is the maximum of $\cI$ according to the total order introduced in Subsection~\ref{subsec:partial-ord}. 
By Lemma~\ref{lem:10.0.1} we can choose $\gamma$ such that 
we have $\lag \tilde\la,\gamma\rag>\lag\mu,\gamma\rag$ for all $\mu\in\cI\bls\{\tilde\la\}$. If the assertion of the proposition is not true, then $\tilde\la_1>0$ and thus  by choosing $\gamma_1$ sufficiently large we can also assume that $\lag \tilde\la,\gamma\rag\geq 1$. 
Thus, for all sufficiently large $k\in\N$ the lowest power of $q$ that occurs in $\sum_{\la\in\cI} c_\la q^{-\lag \la,k\gamma\rag}$ is from the summand $c_{\tilde\la}q^{-\lag \tilde\la,k\gamma\rag}$, and is equal to $d-k\lag \tilde\la,\gamma\rag$, where $d$ is the lowest power of $q$ that occurs in $c_{\tilde\la}$.
By comparing with $\sum_{\mathsf a}\mathbf z_{\mathsf a}(q)\phi_{\mathsf a}(q^2)$ it follows that \[
d-k\lag \tilde\la,\gamma\rag\geq 
\min_{\mathsf a\in\mathcal Z}\left\{-\deg \mathbf z_\mathsf a(q^{-1})\right\}.
\] The right hand side is independent  of $k$ and $\gamma$. However, this is a contradiction since  $k$ can be chosen arbitrarily large
and $\lag \tilde\la,\gamma\rag\geq 1$. \end{proof}

\subsection{Proof of $\la_1\in 2\Z^{\leq 0}$}
\label{subsec:10-4}
In this subsection we strengthen Proposition~\ref{prp:xinPD2}, as follows.

\begin{prp}
\label{prp:xinPD3}
Let $\cI$ be a finite subset of $\Z\eps_1+\cdots+\Z\eps_n$. Let   $x:=\sum_{\la\in\cI} c_\la K_\la\in U_{\g h,R}$ where $c_\la\in\Bbbk^\times$ for $\la\in\cI$, and  assume that $x\in\URw$. Then for every $\la:=\sum \la_i\eps_i\in\cI$ we have $\la_1\in 2\Z^{\leq 0}$.
\end{prp}

\begin{proof}
We assume that the assertion is false, and arrive at a contradiction.

\textbf{Step 1.}
Recall from Proposition~\ref{prp:xm-yn-EF} that the  $K_{\la_{R,b}}$  for $1\leq b\leq n$ are contained in $\URw$. The $K_\la\in U_{\g h,R}$ satisfying $\la_i-\la_{i+1}\in2\Z^{\geq 0}$ for $1\leq i\leq n-1$ and $\la_1\in2\Z^{\leq 0}$ can be expressed as products of the $K_{\la_{R,b}}$. Thus by Proposition~\ref{prp:xinPD1}
 and Proposition~\ref{prp:xinPD2}  we can assume that $\la_1\in\{-1,-3,-5,\ldots\}$  for all 
$\la\in\cI$.

\textbf{Step 2.} 
Set $D:=\phi_U(x)$ so that $D\in \sPD$. 
Write $D$ as $D=D_1+D_2+D_3$ according to Remark~\ref{rmk:Deigven}. 
Suppose that 
$D_1=\sum_{\mathsf b\in\mathcal Z}
\mathbf z_{\mathsf b}
t^\mathsf b\del^\mathsf b$, where $\mathcal Z$ is a finite set of $n$-tuples of non-negative integers and the $\mathbf z_{\mathsf b}\in\Bbbk^\times$. 
After scaling $x$ by a nonzero element of $\C[q]$  we can assume that the $c_\la$ and the $\mathbf z_{\mathsf b}$ are nonzero elements of $\C[q]$.
 We keep using the notation $t^\gamma$ for $\gamma:=\sum_{i=1}^n\gamma_i\eps_i$ from the proof of Proposition~\ref{prp:xinPD2}.
Then $D\cdot t^\gamma=\sum_{\la\in\cI} c_\la q^{-\lag \la,\gamma\rag}t^\gamma$. 
From Lemma~\ref{lem:Dbta(i)-(iii)}(ii) and Remark~\ref{rmk:cab}
it follows that $t^\mathsf b\del^{\mathsf b}\cdot t^\gamma=\phi_\mathsf b(q^2)t^\gamma$  
where $\phi_\mathsf b\in\C[q]$ and 
\begin{equation}
\label{eq:degdphib}
\deg\phi_\mathsf b=
\sum_{i=1}^n\gamma_i(b_1+\cdots+b_i)-\sum_{i=1}^n
\left(b_i(b_1+\cdots +b_i)
-
\frac{b_i(b_i-1)}{2}
\right).\end{equation}
For $\mathsf b\in\mathcal Z$ define $\la_\mathsf b\in\Z\eps_1+\cdots+\Z\eps_n$ by \[
\la_\mathsf b:=b_1\eps_1+(b_1+b_2)\eps_2+\cdots+(b_1+\cdots+b_n)\eps_n.
\]
 By~\eqref{eq:degdphib} we have
$
\deg\phi_{\mathsf b}=\lag \la_{\mathsf b},\gamma\rag+C(\mathsf b)$, 
where $C(\mathsf b)$ is independent of $\gamma$. 

\textbf{Step 3.}
Let $\tilde\la\in\cI$ be such that $-\tilde\la$ is the maximum of $-\cI:=\{-\la\,:\,\la\in\cI\}$ 
with respect  to the total order 
$<$ of  Subsection~\ref{subsec:partial-ord}. 
Using Lemma~\ref{lem:10.0.1} for $-\cI$, we can choose $\gamma$ such that 
$-\lag \tilde\la,\gamma\rag>-\lag \mu,\gamma\rag$ for all $\mu\in\cI\bls \{\tilde\la\}$.
Since $\tilde\la_1\in\{-1,-3,-5,\ldots\}$, by choosing the parity of $\gamma_1$ suitably  we can also assume that  $-\lag \tilde\la,\gamma\rag$ is an odd integer. Then for $k\in\N$ sufficiently large, the highest  power of $q$ that occurs in $\sum_{\la\in\cI} c_\la q^{-\lag \la,k\gamma\rag}$ is from
the summand $c_{\tilde\la} q^{-\lag\tilde \la,k\gamma\rag}$, and is  equal to $d-k\lag \tilde\la,\gamma\rag$, where $d:=\deg c_{\tilde\la}$. 

\textbf{Step 4.} 
Let $\mathsf b_{\max}\in\mathcal Z$  be such that $\la_{\mathsf b_{\max}}=\max\{\la_\mathsf b\,:\,\mathsf b\in\mathcal Z\}$, where the maximum is taken with respect to $<$. Note that the map $\mathsf b\mapsto\la_{\mathsf b}$ is an injection.
From~\eqref{eq:degdphib} and 
Lemma~\ref{lem:10.0.1} applied to the set $\{\la_\mathsf b\,:\,\mathsf b\in \mathcal Z\}$ it follows that we can choose $\gamma$ and $k$ in Step 3 such that the following additional property holds: the highest power of $q$ that occurs in 
$\sum_{\mathsf b\in\mathcal Z}\mathbf z_\mathsf b(q)
\phi_{\mathsf b}(q^2)$ is from the summand $\mathbf z_{\mathsf b_{\max}}(q)\phi_{\mathsf b_{\max}}(q^2)$, and is equal to $d'+2\deg\phi_{\mathsf b_{\max}}$, where $d':=\deg\mathbf z_{\mathsf b_{\max}}$. Note that the values $\deg \phi_{\mathsf b}$ depend on $\gamma$ and $k$, but the values $\deg \mathsf z_{\mathsf b}$ only depend on $x$ and in particular they are independent of the choices of $\gamma$ and $k$.

\textbf{Step 5.} Recall that $t^{k\gamma}$ is an eigenvector  of $D$, hence $D\cdot t^{k\gamma}=D_1\cdot t^{k\gamma}=\sum _{\mathsf b\in\mathcal Z}\mathbf z_{\mathsf b}(q)\phi_\mathsf b(q^2)t^{k\gamma}$ by Remark~\ref{rmk:Deigven}.
By comparing the highest power of $q$ in the eigenvalue of $t^{k\gamma}$ from  Step 3 and Step 4 it follows that
\begin{equation}
\label{eq:parritty}
d'+2\deg\phi_{\mathsf b_{\max}}
=
d-k\lag \tilde \la,\gamma\rag.
\end{equation}
Since $d'$ is independent of $\gamma$ and $k$, the parity of the left hand side of~\eqref{eq:parritty} does not change by varying $k$ and $\gamma$. However, recall that $\lag \tilde\la,\gamma\rag$ is an odd integer and the only constraint on $k$ is that it should be sufficiently large. Thus,  we can choose $k$ such that the parities of the two sides of~\eqref{eq:parritty} are different. This is a contradiction. 
\end{proof}

\subsection{Completing the proof of Theorem~\ref{thm:MainthmB}(i)}
Theorem~\ref{thm:MainthmB}(i) is an immediate consequence of the following corollary and Proposition~\ref{prp:phiUxaphiUyb}.

\begin{cor}
\label{cor:IZie}
Let $\cI$ be a finite subset of $\Z\eps_1+\cdots+\Z\eps_n$. Let $x:=\sum_{\la\in\cI}c_\la K_\la\in U_{R,\g h}$ where $c_\la\in\Bbbk^\times$ for $\la\in\cI$, and assume that $x\in \URw$. Then $x$ belongs to the subalgebra of $U_{\g h,R}$ that is generated by the $K_{\la_{R,b}}$ for  $1\leq b\leq n$. 
\end{cor}

\begin{proof}
Follows immediately from Proposition~\ref{prp:xinPD1}
and  
Proposition~\ref{prp:xinPD3}.
\end{proof}

Corollary~\ref{cor:IZie}
implies that
$\left\{K_{\la_{R,b}}\right\}_{b=1}^n$ is a generating set of the algebra  $\URw$. 
An analogous statement holds for $U_L$. That is,
$\left\{K_{\la_{L,a}} \right\}_{a=1}^m$
is a generating set of  the algebra $\ULw$.

\section{Appendix: commonly used notation}
\label{appendix}

In this section we list the commonly used symbols and notation along with the subsection in which each item is defined.

\medskip

\noindent\textbf{Introduction:} $\Bbbk$,  $U_L$, $U_R$, $U_{LR}$, $\Lm$, $\Rn$, $\Lmw$, $\Rnw$, $\mathsf L_{i,j}$, $\mathsf R_{i,j}$, $\mathcal Y^\mathcal Z$, $\Emb_{a\times b}^{m\times n}$,
$\tilde L_{i,j}$, $M^{\mathbf i}_{\mathbf j}$,
$\oline M^{\mathbf i}_{\mathbf j}$, $\mathbf D(r,a,b)$, $\mathbf D_{k,r}$, $\mathbf D'_{k,r}$, 
$\mathbf R_a$, $\mathbf L_b$, $\mathscr L_{\g h}$, $\mathscr R_{\g h}$, $\Lhb$, $\Rhb$.

\medskip
\noindent\textbf{Subsection~\ref{subsec:lofin}:}
$\ad_y(x)$, $\mathscr F(H,I)$, $\mathscr F(H)$.

\medskip
\noindent\textbf{Subsection~\ref{subsec:matrx}:}
$V^*$, $\lag v^*,v\rag$, $\sfm_{v^*,v}$, $\Delta^\circ$, $H^\circ$.

\medskip
\noindent\textbf{Subsection~\ref{subsec-H-inv}:}
$V_{(\epsilon)}$.

\medskip
\noindent\textbf{Subsection~\ref{subsec:twisted}:} $\check R_{V,W}$, $A\otimes_{\check R}B$, $A\otimes_{\check R,\psi} B$.

\medskip
\noindent\textbf{Subsection~\ref{subsec:prdcts}:}
$R_{V,W}$, $\omega_{V,W}$, $\oline\omega_{V,W}$.

\medskip
\noindent\textbf{Subsection~\ref{subsec::quas}:}
$H^\circ_\cC$, $\lag f\otimes g,R\rag$.

\medskip
\noindent\textbf{Subsection~\ref{subsec:Uqglndf}:} $U_q(\gl_n)$, $\eps_i$, 
$\llbracket a,b\rrbracket$, 
$K_{\eps_i}$, $K_i$, $K_\la$, $E_i$, $F_i$, $U_{\g h,L}$, $U_{\g h,R}$.

\medskip
\noindent\textbf{Subsection~\ref{subsec:R-matUq}:}
$\cC^{(n)}$, $\EuScript R^{(n)}$, $\underline{\EuScript R}^{(n)}$, $\mathrm{Exp}_q$, $\lag\mu,\nu\rag$, 
$\check{\EuScript R}^{(n)}$, $\underline {\check{ \EuScript R}}^{(n)}$.

\medskip
\noindent\textbf{Subsection~\ref{subsec:naturalUqcirc}:}
  $x^\natural$.

\medskip
\noindent\textbf{Subsection~\ref{subsec:PnnnDnan}:} $\Delta_n^+$, $\mathsf E_{i,j}$, $\sP_{n\times n}$, $\sD_{n\times n}$, $t_{i,j}$, $\del_{i,j}$, $t_{i,j}^\natural$, $\del_{i,j}^\natural$, 
$\cL_\sP$, $\cR_\sP$, $\cL_\sD$, $\cR_\sD$, $\Delta_\sP$, $\Delta_\sD$, $\xi_c$, $\Xi$, $\iota$, $\underline\iota$.

\medskip
\noindent\textbf{Subsection~\ref{subsec:Now3.5}:}
 $\sP$,  $\sP_{m\times n}$, $\sD$, $\sD_{m\times n}$, $\phi_U$.
 
\medskip
\noindent\textbf{Subsection~\ref{subsec:ULR-mdec}:}
$\ell(\la)$, $\sP^{(d)}$, $\sD^{(d)}$, $V_\la$, $\Lambda_{d,r}$, $|\la|$.

\medskip
\noindent\textbf{Subsection~\ref{subsec:sPDgr}:}
$\cC_L$, $\cC_R$, 
$\EuScript R_L$, $\underline{\EuScript R}_L$, $\EuScript R_R$, $\underline{\EuScript R}_R$, 
$\check{\EuScript R}_L$, $\underline{\check{\EuScript R}}_L$, $\check{\EuScript R}_R$, $\underline{\check{\EuScript R}}_R$,
$\EuScript R_{LR}$, $\underline{\EuScript R}_{LR}$,
$\check{\EuScript R}_{LR}$, $\underline{\check{\EuScript R}}_{LR}$,
$\sPD^\mathrm{gr}$, $\sPD$. 

\medskip
\noindent\textbf{Subsection~\ref{subsec:Now3.9}:}
$\mathscr A_{k,l,n}$, $\mathscr A_{k,l,n}^\mathrm{gr}$,  $\boldsymbol\kappa_{r,n}$, $(\Emb^\mathrm{gr})_{m'\times n'}^{m\times n}$. 

\medskip
\noindent\textbf{Subsection~\ref{subsec:sPDgr6}:}
$\phi_{PD}$, $X\cdot f$.

\medskip
\noindent\textbf{Subsection~\ref{subsec:now3.11}:} $\sP^{(\leq k)}$, 
$\prec$, 
$\mathbf c(a)$.

\medskip
\noindent\textbf{Subsection~\ref{subsec:Pkln}:}
$\mathsf P$, $\mathsf P_{k,l,n}$, 
$\sPD^{\mathrm{gr},(r,s)}$, $\mathscr A_{k,l,n}^{\mathrm{gr},(r,s)}$,

\medskip
\noindent\textbf{Subsection~\ref{subsec:Pi}:}
$\ULw$, $\URw$, $\ULRw$.

\medskip
\noindent\textbf{Subsection~\ref{subsec:relthms}:} $\eta_{m,n}$.

\medskip
\noindent\textbf{Subsection~\ref{subsec:sometech}:} $\lhd$, 
$\mathbf c(a,b)$, $t^\mathsf a$, $\del^\mathsf a$.

\medskip
\noindent\textbf{Subsection~\ref{subsec:CEin}:}
$\la_{L,a}$, $\la_{R,b}$, $K_{\la_{L,a}}$, $K_{\la_{R,b}}$

\medskip
\noindent\textbf{Subsection~\ref{subsec:Now6.1}:}
$\sPD_{(\epsilon_L)}$,
$\sPD_{(\epsilon_R)}$, 
 $\epsilon_L$, $\epsilon_R$, 
 $\left(\mathscr A_{k,l,n}\right)_{(\epsilon_R)}$,  
$\left(\mathscr A^\mathrm{gr}_{k,l,n}\right)_{(\epsilon_R)}$,
$\left(\mathscr A_{k,l,n}^{\mathrm{gr},(r,s)}\right)_{(\epsilon_R)}
$,
 $\sPD^\mathscr L$, $\sPD^\mathscr R$,  
 $\sPD^{\Lmw}$, $\sPD^{\Rnw}$, $\tilde{\mathsf L}_{i,j}^\mathrm{gr}$.
 
\medskip
\noindent\textbf{Subsection~\ref{subsec:Now7.1}:}
$\Gamma_n$, $\Gamma_{k,l,n}$. 

\medskip
\noindent\textbf{Subsection~\ref{subsec:newprodt}:} $u\star_n v$, $u\star_{k,l,n} v$.

\end{document}